\documentclass[11pt, reqno]{amsart}
\allowdisplaybreaks

\usepackage{comment}
\usepackage[rgb,dvipsnames]{xcolor}
\usepackage{stmaryrd}
\usepackage{amsfonts}
\usepackage{amssymb}
\usepackage{amsmath,amsthm}
\usepackage{latexsym}
\usepackage{amscd}
\usepackage{esint}
\usepackage{mathrsfs}
\usepackage{dsfont}
\usepackage{setspace}
\usepackage{enumitem}
\usepackage[margin=1.2in]{geometry}
\usepackage{bm}

\usepackage{hyperref}
\usepackage{comment}
\usepackage{enumitem}

\newtheorem{prop}{Proposition}[section]
\newtheorem{thm}[prop]{Theorem}
\newtheorem{lemm}[prop]{Lemma}
\newtheorem{coro}[prop]{Corollary}

\newtheorem*{claim*}{Claim}
\newtheorem*{lemm*}{Lemma}
\newtheorem*{thm*}{Theorem}

\theoremstyle{definition}
\newtheorem{defi}[prop]{Definition}
\newtheorem{rmk}[prop]{Remark}

\newcommand{\CC}{\mathbb{C}}

\newcommand{\NN}{\mathbb{N}}

\newcommand{\RR}{\mathbb{R}}

\newcommand{\ZZ}{\mathbb{Z}}

\newcommand{\cA}{\mathcal A}
\newcommand{\cB}{\mathcal B}
\newcommand{\cC}{\mathcal C}

\newcommand{\cE}{\mathcal E}
\newcommand{\cF}{\mathcal F}

\newcommand{\cH}{\mathcal H}

\newcommand{\cK}{\mathcal K}
\newcommand{\cL}{\mathcal L}
\newcommand{\cM}{\mathcal M}
\newcommand{\cN}{\mathcal N}

\newcommand{\cP}{\mathcal P}

\newcommand{\cS}{\mathcal S}

\newcommand{\cU}{\mathcal U}
\newcommand{\cV}{\mathcal V}
\newcommand{\cW}{\mathcal W}
\newcommand{\cX}{\mathcal X}
\newcommand{\cY}{\mathcal Y}
\newcommand{\cZ}{\mathcal Z}

\def\bB{\mathbf{B}}
\def\bT{\mathbf{T}}

\def\bt{\mathbf{t}}
\def\bn{\mathbf{n}}
\def\bN{\mathbf{N}}

\DeclareMathOperator{\Span}{span}

\DeclareMathOperator{\Div}{div}
\DeclareMathOperator{\loc}{loc}
\DeclareMathOperator{\osc}{osc}
\DeclareMathOperator{\re}{Re}

\DeclareMathOperator{\dist}{dist}

\DeclareMathOperator{\proj}{proj}

\DeclareMathOperator{\Ind}{Ind}

\DeclareMathOperator{\vol}{\text{vol}}

\newcommand{\ep}{\varepsilon}
\newcommand{\bep}{\overline{\varepsilon}}
\newcommand{\blambda}{\overline{\lambda}}

\newcommand{\bangle}[1]{\big\langle #1 \big\rangle}
\newcommand{\pa}[2]{\frac{\partial #1}{\partial #2}}

\newcommand{\paop}[1]{\pa{}{#1}}

\newcommand{\rom}[1]{\expandafter\romannumeral #1}
\newcommand{\Rom}[1]{\uppercase\expandafter{\romannumeral #1}}

\setlist[enumerate]{leftmargin = 2em}
\setlist[itemize]{leftmargin = 2em}
\allowdisplaybreaks

\title[Existence of capillary CMC disks]{Existence of constant mean curvature disks in $\RR^3$ with capillary boundary condition}
\author{Da Rong Cheng}
\address{Department of Mathematics, University of Miami, Coral Gables, FL 33155}
\email{darong.cheng@miami.edu}

\begin{document}

\begin{abstract} 
We extend Struwe's result~\cite{Struwe1988} on the existence of free boundary constant mean curvature disks to almost every prescribed boundary contact angle in $(0, \pi)$. Specifically, let $\Sigma$ be a surface in $\RR^3$ diffeomorphic to the sphere, and let $\Sigma'$ be a convex surface enclosing $\Sigma$. Given $\tau \in (-1, 1)$ and a constant $H \geq 0$ below the infimum of the mean curvature of $\Sigma'$, we show that for almost every $r \in (0, 1)$, in the region enclosed by $\Sigma'$ there exists a branched immersed disk with constant mean curvature $rH$ whose boundary meets $\Sigma$ at an angle with cosine value $r\tau$. Moreover, the constant mean curvature disks we construct have index at most $1$. 
\end{abstract}

\maketitle 


\section{Introduction}\label{sec:intro}
\subsection{Statement}\label{subsec:theorem-statement}
Let $\Sigma$ be a surface in $\RR^3$ which is the image of the unit sphere $S^2$ under a smooth diffeomorphism from $\RR^3$ to itself, and suppose $\Sigma'$ is a closed convex surface. We denote by $\Omega$ and $\Omega'$ the bounded open sets enclosed by $\Sigma$ and $\Sigma'$, respectively, and assume that $\Sigma \subset \overline{\Omega'}$. By the convexity of $\Sigma'$, its mean curvature with respect to the unit normal pointing into $\Omega'$, denoted $H_{\Sigma'}$, is everywhere positive. For all $H \in \RR$ satisfying
\begin{equation}\label{eq:H-requirement}
0 \leq H < \inf_{y \in \Sigma'} H_{\Sigma'}(y) =: \underline{H'},
\end{equation}
Struwe showed in his seminal work~\cite{Struwe1988} that for almost every $r \in (0, 1)$, there exists a branched immersed disk with constant mean curvature $rH$ whose boundary meets $\Sigma$ orthogonally. In~\cite{Cheng22} we improved Struwe's existence theorem to include a Morse index upper bound. In this paper we continue this line of work and construct constant mean curvature (CMC) disks whose boundaries meet $\Sigma$ at prescribed constant angles that are not necessarily $\frac{\pi}{2}$. Such surfaces are usually referred to as capillary CMC surfaces.

For the construction we adopt the mapping approach. That is, we obtain the desired surfaces as images of mappings into $\RR^3$ that satisfy a specific system of PDEs and boundary conditions. To state the boundary value problem of interest, let $\bB$ denote the open unit disk in $\RR^2$ and suppose $H \in [0, \underline{H'})$ and $\tau \in (-1, 1)$, with $\tau$ being the cosine value of the contact angle. Then what we seek is a non-constant, weakly conformal branched immersion $u: (\overline{\bB}, \partial \bB) \to (\RR^3, \Sigma)$ such that
\begin{equation}\label{eq:H-cmc-bvp}
\left\{
\begin{array}{l}
\Delta u = H \cdot u_{x^1} \times u_{x^2}, \text{ in } \bB\\
u_r + \tau \bN(u) \times u_\theta \perp T_u\Sigma, \text{ on }\partial \bB,
\end{array}
\right.
\end{equation}
where $\bN:\Sigma \to \RR^3$ denotes the unit normal to $\Sigma$ pointing out of $\Omega$, and we write $\bN(u)$ to mean $\bN \circ u$. Indeed, if $u$ is a weakly conformal solution to~\eqref{eq:H-cmc-bvp}, then at points where it is an immersion, the PDE implies that the mean curvature vector of the image surface is $H$ times a unit normal, while taking the cross product of the boundary condition with $u_\theta$ on the right shows that $\bN(u)$ and $u_r \times u_\theta$, the latter being normal to the image of $u$, make an angle with cosine value $\tau$. Our result is the following.
\begin{thm}\label{thm:main-1} Given $H \in [0, \underline{H'})$ and $\tau \in (-1, 1)$, for almost every $r \in (0, 1)$ there exists a smooth, non-constant, weakly conformal branched immersion $u: (\overline{\bB}, \partial \bB) \to (\overline{\Omega'}, \Sigma)$ that solves~\eqref{eq:H-cmc-bvp} with $rH$ and $r\tau$ in place of $H$ and $\tau$, and has index at most one.
\end{thm}
\begin{rmk}\label{rmk:main-index}
The meaning of the index in the last assertion is explained below. See also Section~\ref{subsec:index-notation}, particularly Remark~\ref{rmk:index-statement}, for its precise definition.
\end{rmk}
\begin{rmk}\label{rmk:tau}
Since the case $\tau = 0$ is already contained in~\cite{Struwe1988},~\cite{Fraser2000} and~\cite{Cheng22}, we restrict our attention to the case $\tau \neq 0$ when giving the proof of Theorem~\ref{thm:main-1} in Section~\ref{sec:proof}. Also note that in the case $H = 0$, we obtain capillary \textit{minimal} disks from Theorem~\ref{thm:main-1}. Finally, finding conditions on $\Sigma$ under which existence holds for \textit{all} $r \in (0, 1)$ is something we plan to take up in a future work.
\end{rmk}
Below we make some further remarks on Theorem~\ref{thm:main-1} before clarifying the meaning of the index at the end of the statement. The purpose of requiring that $H < \underline{H'}$ is, as before, to prevent the scenario where the CMC surface we obtain is a sphere rather than a disk. Also, by Lemma~\ref{lemm:weakly-conformal} below and Corollary~\ref{coro:branch} from Section~\ref{subsec:branch-points}, for any non-constant solution $u$ of~\eqref{eq:H-cmc-bvp} we have that (i) $u$ is weakly conformal and (ii) $u$ fails to be an immersion only at isolated points, across which the Gauss map extends smoothly. In this paper, we refer to maps satisfying (ii) as branched immersions. Proving Theorem~\ref{thm:main-1} is thus reduced to finding a non-constant smooth map from $(\overline{\bB}, \partial \bB)$ to $(\overline{\Omega'}, \Sigma)$ that solves~\eqref{eq:H-cmc-bvp} and satisfies the asserted index bound. 

To explain which index is being referred to, we digress briefly to discuss the variational aspect of~\eqref{eq:H-cmc-bvp}. To set up notation we extend $\bN$ first to a tubular neighborhood of $\Sigma$ via the nearest point projection and then to $\RR^3$ via cutting off, obtaining a vector field $X$ on $\RR^3$. For a map $u$ from $(\overline{\bB},\partial \bB)$ to $(\RR^3, \Sigma)$ and a path $\gamma$ of such maps that starts at a constant map and ends at $u$, we define 
\[
E_{f, \tau}(u, \gamma) = D_\tau(u) + V_f(\gamma) +  \tau V_{\Div X}(\gamma),
\]
where
\begin{equation}\label{eq:D-tau-definition}
D_\tau(u) = \int_{\bB} \big[\frac{|\nabla u|^2}{2} - \tau X(u) \cdot u_{x^1} \times u_{x^2} \big] dx,
\end{equation}
and $f:\RR^3 \to [0, H]$ is a compactly supported smooth function which is identically $H$ on a neighborhood of $\overline{\Omega'}$. The quantity $V_h(\gamma)$, with $h$ a given smooth function on $\RR^3$, is roughly speaking the integral of $h$ over the region swept out by the path $\gamma$, which for sufficiently smooth paths is given by
\[
V_{h}(\gamma) = \int_{[0, 1] \times \bB} \gamma^*(h\vol_{\RR^3}).
\]
Here we are viewing $\gamma$ as a map from $([0, 1] \times \overline{\bB}, [0, 1] \times \partial \bB)$ into $(\RR^3, \Sigma)$. We refer to Section~\ref{subsec:perturbed} below for more details on $V_h$. Among its properties is that switching to a different path $\widetilde{\gamma}$ leading from a constant map to $u$ changes the value of $V_h$ by an integer multiple of $\int_{\Omega}h$, with the integer depending only on $\gamma$ and $\widetilde{\gamma}$, and not on $h$. Consequently the derivative of $E_{f, \tau}$ depends only on $u$, and a direct computation shows that critical points of $E_{f, \tau}$ which map into $\overline{\Omega'}$ are, at least formally, solutions to the boundary value problem~\eqref{eq:H-cmc-bvp}. 

We may now elaborate on the meaning of the index in Theorem~\ref{thm:main-1}. Upon replacing $\int_{\bB} \frac{|\nabla u|^2}{2} dx$ in the definition of $E_{f, \tau}$ by the mapping area and using Stokes' theorem to get
\[
\tau V_{\Div X}(\gamma) -  \int_{\bB}\tau X(u) \cdot u_{x^1} \times u_{x^2} = \tau \cdot \text{(area of the part of $\Sigma$ swept out by $\gamma|_{[0, 1] \times \partial \bB}$)},
\]
we obtain the well-known capillary functional, whose second variation, initially defined for vector fields along $u$ which are tangent to $\Sigma$ on $\partial \bB$, is known to reduce to a bilinear form defined for real-valued functions. (See Section~\ref{subsec:index-notation}.) It is the index of this latter bilinear form that is referred to in Theorem~\ref{thm:main-1}. 

To end this part of the introduction, we show that smooth solutions of~\eqref{eq:H-cmc-bvp} are weakly conformal. The method of proof is the classical Hopf differential argument.
\begin{lemm}\label{lemm:weakly-conformal}
Suppose $u: (\overline{\bB}, \partial \bB) \to (\RR^3, \Sigma)$ is a smooth solution of~\eqref{eq:H-cmc-bvp} with the constant $H$ replaced by $f\circ u$, where $f: \RR^3 \to \RR$ is any given smooth function. Then $u$ is weakly conformal.
\end{lemm}
\begin{proof}
Let $\varphi(z) = z^2 \bangle{u_z, u_z}$. By assumption we have $\Delta u = f(u) u_{x^1} \times u_{x^2}$, so $\bangle{\Delta u, u_z} = 0$ and hence $\varphi$ is holomorphic on $\bB$. By the boundary condition in~\eqref{eq:H-cmc-bvp} and the fact that $u(\partial \bB) \subset \Sigma$, we see that $\bangle{u_r, u_\theta} = 0$, so $\varphi$ takes real values on $\partial \bB$. Consequently $\varphi$ is a constant, which must be zero since $|u_z|^2$ is integrable on $\bB$. Thus $\varphi$ vanishes identically, and we are done.
\end{proof}
\subsection{Context}
The mapping approach to producing CMC surfaces has a long history. Among the early works are those by Heinz~\cite{Heinz1954}, Wente~\cite{Wente1969} and Hildebrandt~\cite{Hildebrandt1970} on the Plateau problem for CMC disks in $\RR^3$, followed by the proof of the Rellich conjecture by Struwe~\cite{Struwe1985,Struwe1986} and independently Brezis-Coron~\cite{BrezisCoron1984}. In~\cite{Struwe1984}, Struwe applied the perturbation method of Sacks-Uhlenbeck~\cite{Sacks-Uhlenbeck1981} to obtain free boundary ($\tau = 0$) minimal ($H = 0$) disks in $\RR^3$. Later, using the parabolic version of~\eqref{eq:H-cmc-bvp} to regularize sweepouts instead of perturbing the relevant functional, Struwe~\cite{Struwe1988} constructed free boundary CMC disks in $\RR^3$ for \textit{almost every} $H$ up to the mean curvature of the smallest round sphere enclosing $\Sigma$. Building upon joint work with Zhou on CMC spheres~\cite{cz-cmc}, in~\cite{Cheng22} we returned to the Sacks-Uhlenbeck method and improved Struwe's result in~\cite{Struwe1988} to include an index bound, which allowed us to remove the ``almost every'' from his theorem under suitable convexity assumptions on $\Sigma$. Theorem~\ref{thm:main-1} continues this line of work by addressing the case of capillary CMC disks. 

In a similar development, but in the realm of geometric measure theory, Li-Zhou-Zhu~\cite{Li-Zhou-Zhu2021} proved the existence of capillary CMC surfaces for any prescribed $H \in \RR$ and $\tau \in (-1, 1)$ and for generic choices of metrics in a compact 3-manifold $M$ with boundary, pushing further the Almgren-Pitts theory for closed CMC hypersurfaces established by Zhou and Zhu in~\cite{ZhouZhu-cmc,ZhouZhu-pmc}. The case where $H = 0$ and $M$ is a convex domain in $\RR^3$ is also independently obtained by De Masi and De Philippis~\cite{DMDP2021}.

We have limited the above brief summary to just the results that we find most closely related to this paper. For more on the heat flow method developed by Struwe, see for instance his fundamental work on the harmonic map flow~\cite{Struwe1985-heat} and his survey article~\cite{Struwe1988survey}. For more detailed surveys of previous existence results on boundary value problems specifically for \textit{minimal} surfaces, see for instance the introduction of~\cite{Li-Zhou2021} or~\cite{Lin-Sun-Zhou2020}. For previous existence results on \textit{closed} CMC surfaces, we refer to the account in~\cite{ZhouZhu-cmc}. 

\subsection{Outline}\label{subsec:outline}
Besides an increased level of technicality due to additional terms associated with $\tau$, the one place where the argument in this paper differs most from the free boundary case is the removal of isolated boundary singularities for solutions of~\eqref{eq:H-cmc-bvp} with finite Dirichlet integral (Section~\ref{subsec:remove}), which requires a slightly different strategy primarily because the boundary condition in~\eqref{eq:H-cmc-bvp} is not immediately amenable to reflection techniques. The pivotal estimate is Proposition~\ref{prop:decay-ODE}, which is inspired by precedents in other contexts such as~\cite{Parker1996},~\cite{Lin-Wang1998},~\cite{ColdingMinicozzi2008} and~\cite{Lin-Sun-Zhou2020}. One central ingredient in the proof of Proposition~\ref{prop:decay-ODE} is Lemma~\ref{lemm:poincare-q}, which takes the place of the Wirtinger inequalities used in the works just cited, and which we manage to establish without reflecting the solution. 

An outline of the proof of Theorem~\ref{thm:main-1} follows. As already noted, the system~\eqref{eq:H-cmc-bvp} is the Euler-Lagrange equation of the functional $E_{f, \tau}$. To find a critical point of the latter, we adopt the approach in~\cite{Cheng22}, which extends the method used in~\cite{cz-cmc} and draws upon earlier works by Struwe~\cite{Struwe1984,Struwe1988}, Fraser~\cite{Fraser2000}, and of course Sacks-Uhlenbeck~\cite{Sacks-Uhlenbeck1981}. We first consider a family of perturbed functionals defined by
\begin{equation}\label{eq:perturbed-intro}
E_{\ep, p, f, \tau}(u, \gamma) = D_{\ep, p, \tau}(u) + V_f(\gamma) + \tau V_{\Div X}(\gamma),
\end{equation}
where 
\[
D_{\ep, p, \tau}(u) = D_\tau(u)  + \frac{\ep^{p-2}}{p}\int_{\bB} \big(1 + |\nabla u|^2 - 2\tau X(u) \cdot u_{x^1} \times u_{x^2}\big)^{\frac{p}{2}}\ dx,
\]
and we abbreviate $D_{\ep, p, 0}$ as $D_{\ep, p}$. With the help of standard estimates for oblique derivative problems involving the Laplacian, we manage to show as in~\cite{Cheng22} that, provided $p > 2$ is sufficiently close to $2$ depending only on $\Sigma$ and $\tau$, critical points of $E_{\ep, p, f, \tau}$ are smooth up to the boundary and satisfy gradient estimates which are uniform in $\ep$ in the absence of energy concentration. (See Section~\ref{subsec:smoothness} and Section~\ref{subsec:a-priori}.) We eventually fix such a $p$, while $\ep$ is to be sent to zero. In addition to the gradient estimates, by the choice of $f$, the convexity of $\Sigma'$ and the maximum principle, we show that the images of critical points are contained in a bounded neighborhood of $\overline{\Omega'}$. (See Proposition~\ref{prop:maximum-principle}(a).)

We then consider the collection $\cP$ of ``non-contractible'' paths $\gamma: ([0, 1] \times \overline{\bB}, [0, 1] \times \partial \bB) \to (\RR^3, \Sigma)$ that start and end at constant maps, and set
\[
\omega_{\ep, p, f, \tau} = \inf_{\gamma \in \cP}\big[ \sup_{t \in [0, 1]}E_{\ep, p, f, \tau}(\gamma(t), \gamma|_{[0, t]}) \big].
\]
As in~\cite{cz-cmc,Cheng22}, to show that $\omega_{\ep, p, f, \tau}$ is a critical value we need to produce a minimizing sequence of paths in $\cP$ whose near-maximal slices satisfy certain uniform $D_{\ep, p}$-bounds. Such paths are again found by applying Struwe's monotonicity trick, the crucial observation this time being that the function
\begin{equation}\label{eq:mono-object}
r \mapsto \frac{E_{\ep, p, rf, r\tau}(u, \gamma)}{r} = \frac{D_{\ep, p, r\tau}(u)}{r} + V_f(\gamma) + \tau V_{ \Div X}(\gamma)
\end{equation}
is non-increasing. Verifying this monotonicity requires more computation than in the free boundary case, since now $D_{\ep, p, r\tau}$ depends on $r$. (See Lemma~\ref{lemm:D-difference}(a).) Combining the monotonicity trick with essentially the same deformation procedure used in~\cite{cz-cmc} and~\cite{Cheng22} to avoid high-index critical points, and applying a standard pseudo-gradient flow argument, we get that for almost every $r \in (0, 1)$, there exist a sequence $(\ep_j)$ converging to zero, a constant $C_0$ independent of $j$, and, for each $j$, a non-constant critical point $u_j$ of $E_{\ep_j, p, rf, r\tau}$ so that $D_{\ep_j, p}(u_j) \leq C_0$ and that the index of the second variation, denoted $(\delta^2 E_{\ep_j, p, rf, r\tau})_{u_j}$, is at most $1$. (See Proposition~\ref{prop:perturbed-existence}.)

Next we want to pass to the limit as $j \to \infty$, and we only explain the case where energy concentration occurs, since that is the more difficult out of the two cases. At a point of energy concentration, we rescale $u_j$ suitably to obtain a non-constant, finite-energy solution $v:(\overline{\RR^2_+}, \partial \RR^2_+) \to (\RR^3, \Sigma)$ to a boundary value problem of the form~\eqref{eq:H-cmc-bvp} but posed on the upper half plane. As in~\cite{Cheng22}, in this process we have to ensure that no perturbation terms are left in the limit, and that the limiting solution is indeed defined on $\RR^2_+$ as opposed to $\RR^2$. (See Proposition~\ref{prop:scale-comparison}.) We then invoke the removable singularity theorem already mentioned (Corollary~\ref{coro:removal-singularity}), so that pulling $v$ back by a conformal map and applying the maximum principle a second time (Proposition~\ref{prop:maximum-principle}(b)) yields a non-constant smooth solution $\widetilde{v}: (\overline{\bB}, \partial \bB) \to (\overline{\Omega'}, \Sigma)$ of~\eqref{eq:H-cmc-bvp} with $H, \tau$ replaced by $rH, r\tau$. 

To finish, we use the standard logarithmic cutoff trick to pass the index upper bound on $(\delta^2 E_{\ep_j, p, rf, r\tau})_{u_j}$ to the second variation of the unperturbed functional $E_{rf, r\tau}$ at $\widetilde{v}$. The index comparison results in Section~\ref{sec:index-comparison} then allows us to further transfer the index bound to the desired bilinear form. We remark that similar index comparisons have been carried out previously in various different contexts. The earliest example is perhaps found in the work of Ejiri-Micallef~\cite{Ejiri-Micallef2008}. See~\cite{Lima2022,cz-cmc,Cheng22} for a few other examples.  
\subsection{Organization}\label{subsec:organization}
In Section~\ref{sec:variational} we set up notation, define the perturbed functional, and establish the smoothness of its critical points. Section~\ref{sec:PDE} is devoted to various a priori estimates, as well as issues related to branch points and removable singularities. In Section~\ref{sec:index-comparison} we establish the index comparison results mentioned above. In Section~\ref{sec:perturbed-existence} we construct critical points of the perturbed functional with index at most $1$. In Section~\ref{sec:proof} we put everything together and prove Theorem~\ref{thm:main-1}. Appendix~\ref{sec:differentiability} mainly concerns the monotonicity of~\eqref{eq:mono-object}. Appendix~\ref{sec:structure} establishes the form of the Euler-Lagrange equation of $E_{\ep, p, f, \tau}$ in suitable local coordinates. Appendix~\ref{sec:second-variation} contains two lengthy calculations involving the capillary functional, one standard and the other due to Ros-Souam~\cite{RosSouam1997}, reproduced in our notation for the reader's convenience. Appendix~\ref{sec:linear} collects some standard $W^{2, p}$ estimates for a certain type of oblique derivative problems on the upper half plane and in the unit disk. 
\section{The perturbed functional}\label{sec:variational}
\subsection{Notation}\label{subsec:function-spaces}
With $\Sigma, \Sigma'$ and $\Omega, \Omega'$ as defined in the introduction, we let $\bN: \Sigma\to \RR^3$ be the unit normal to $\Sigma$ pointing out of $\Omega$. Also, we fix a pre-compact tubular neighborhood $\cW$ of $\Sigma$ on which the nearest point projection $\Pi$ has bounded derivatives of all orders, and write $\bN(y)$ to mean $\bN(\Pi(y))$ when $y \in \cW$. In addition, we adopt the list of notation found at the start of~\cite[Section 2.1]{Cheng22}, so that for instance $\RR^2_+$ denotes the open upper half plane, $\bB$ the open unit disk in $\RR^2$, and $\bB^+$ and $\bT$ the intersection of $\bB$ with $\RR^2_+$ and $\partial \RR^2_+$, respectively. Finally, having fixed $\Sigma, \Sigma'$ and $\cW$, in all the estimates throughout this paper, we do not explicitly mention their influence on the constants and smallness thresholds.

For $p \in (2, 3]$ to be determined later, we define
\[
\cM_p = \{u \in W^{1, p}(\bB; \RR^3)\ |\ u(x) \in \Sigma \text{ for all }x \in \partial \bB\}.
\]
The space $\cM_p$ is a smooth, closed submanifold of the Banach space $W^{1, p}(\bB; \RR^3)$. Given $u \in \cM_p$, the tangent space to $\cM_p$ at $u$ is given by
\[
T_u\cM_p = \{\varphi \in W^{1, p}(\bB; \RR^3)\ |\ \varphi(x) \in T_{u(x)} \Sigma \text{ for all }x \in \partial \bB\},
\]
which is a closed subspace of $W^{1, p}(\bB; \RR^3)$ with a closed complement. Restricting the $W^{1, p}$-norm to each $T_u \cM_p$ yields a Finsler structure on $\cM_p$. As explained for example in~\cite{ACS2018}, an atlas of coordinate charts for $\cM_p$ can be obtained as follows. Taking the product metric on $\cW$ and interpolating it with the standard metric $g_{\RR^3}$ via a cutoff function, we obtain a smooth Riemannian metric $\overline{g}$ on $\RR^3$ which agrees with $g_{\RR^3}$ away from $\Sigma$, with respect to which $\Sigma$ is totally geodesic. Letting $\exp^{\overline{g}}$ denote the exponential map of $\overline{g}$, then for each $u \in \cM_p$, 
\[
\Theta_u: \varphi \mapsto \exp^{\overline{g}}_{u}(\varphi)
\]
maps a small enough ball around the origin in $T_u \cM_p$ diffeomorphically onto a neighborhood of $u$ in $\cM_p$. 

\begin{rmk}\label{rmk:h-geodesics}
We mention another use of the auxiliary metric $\overline{g}$. Since $\overline{g}$ agrees with $g_{\RR^3}$ away from $\Sigma$, there exists a constant $\delta_0 > 0$ and a smooth function
\[
c: [0, 1] \times \{(y, y') \in \RR^3 \times \RR^3\ |\ |y - y'| < \delta_0\} \to \RR^3
\]
with the following properties: 
\begin{enumerate}
\item[(c1)] $c$ has bounded derivatives of all orders on $[0, 1] \times \{(y, y') \in \RR^3 \times \RR^3\ |\ |y - y'| < \delta_0\} $.
\vskip 1mm
\item[(c2)] $c(\cdot, y, y')$ is the unique minimizing $\overline{g}$-geodesic from $y$ to $y'$.
\vskip 1mm
\item[(c3)] $c(s, y, y') \in \Sigma$ for all $s \in [0, 1]$ and $y, y' \in \Sigma$.
\end{enumerate}
Thanks to the above properties, given $u, v \in \cM_p$ with $\|u - v\|_{\infty} < \delta_0$, if we define $l_{u, v}:[0, 1] \to \cM_p$ by 
\[
[l_{u, v}(s)](x) = c(s, u(x), v(x)),
\]
then $l_{u, v}$ is a $C^{1}$-path. Moreover, for some universal constant $C$, we have the following two estimates for all $u, v$ as above. First, 
\begin{equation}\label{eq:speed-pointwise}
|l_{u, v}'(s)(x)| \leq C|u(x) - v(x)|, \text{ for all }s \in [0, 1] \text{ and } x \in \bB.
\end{equation}
Secondly,
\begin{equation}\label{eq:gradient-pointwise}
|\nabla (l_{u, v}(s))(x)| \leq C(|\nabla u(x)| + |\nabla v(x)|),  \text{ for all }s \in [0, 1] \text{ and almost every } x \in \bB.
\end{equation}
\end{rmk}
We now turn to a discussion of the coordinate charts for $\Sigma$ to be used when deriving boundary estimates. Specifically, given $\delta \in (0, 1)$, with the help of local isothermal coordinates on $\Sigma$, we may find a finite open covering $\cF$ of $\Sigma$ along with some $\rho > 0$ such that, first of all, every $U \in \cF$ is contained in $\cW$. Secondly, for all $y \in \Sigma$, some $U \in \cF$ contains $B_{\rho}(y)$. Finally, and most importantly, for each $U \in \cF$ there exists a diffeomorphism $\Phi = \Phi_U$ which maps it onto an open set $V \subset \RR^3$ and has the properties listed below:
\begin{enumerate}
\item[(i)] $\Phi$ locally flattens $\Sigma$ in the sense that 
\[
\Phi(U \cap \Omega) = V \cap \{y^3 > 0\},\ \ \Phi(U \cap \Sigma) = V \cap \{y^3 = 0\}.
\]
Moreover, we have 
\[
(d\Phi)_{y}(\bN(y)) = -\paop{y^3} \text{ for all }y \in U.
\] 
\vskip 1mm
\item[(ii)] Letting $\Psi$ denote the inverse of $\Phi$, then $(d\Psi)_y\big(  \paop{y^1} \big)$ and $(d\Psi)_y\big(  \paop{y^2} \big)$ are both orthogonal to $\bN(\Psi(y))$ for all $y \in V$, and additionally are orthogonal and of equal length when $y \in V \cap \{y^3 = 0\}$.
\vskip 1mm
\item[(iii)] There exists some $A \in SO(3)$ such that 
\begin{equation}\label{eq:isometry-deviation}
\big|(d\Phi)_{y} - A \big|  + \big|[(d\Phi)_y]^{-1} - A^{-1} \big| \leq \delta, \text{ for all }y \in U.
\end{equation}
Here we have operator norms on the left hand side. Thus, for example, it follows that
\[
\big||(d\Psi)_y(\xi)| - |\xi|\big| \leq \delta |\xi|, \text{ for all }\xi \in \RR^3 \text{ and }y \in V.
\]
\vskip 1mm
\item[(iv)] $\Phi$ and $\Psi$ have bounded derivatives of all orders on $U$ and $V$ respectively. 
\end{enumerate}
\vskip 2mm
Whenever such a chart is introduced, we adopt the following set of notation:
\begin{enumerate}
\vskip 1mm
\item[(n1)] $g_{ij} = \pa{\Psi}{y^i} \cdot \pa{\Psi}{y^j}$. Note from (i) and (ii) that for all $y \in V$, we have 
\[
g_{13}(y) = g_{23}(y) = 0,\ \ g_{33}(y) = 1,
\]
and on $V \cap \{y^3 = 0\}$ we have additionally that
\[
g_{22}(y) = g_{11}(y),\ \ g_{12}(y) = 0.
\]
Also, we use $\Gamma_{jk}^i$ to denote the Christoffel symbols of $g$.
\vskip 1mm
\item[(n2)] $P_{ijk} = \pa{\Psi}{y^i} \cdot \pa{\Psi}{y^j} \times \pa{\Psi}{y^k}$. Also, we let $P^{i}_{jk} = g^{il}P_{ljk}$.
\vskip 1mm
\item[(n3)] $J^i_{j} = -g^{il}P_{lj3} = -P^{i}_{j3}$. Note that for any $i, j \in \{1, 2, 3\}$ we have 
\[
J^{i}_3(y) = J^{3}_i(y) = 0 \text{ for all }y \in V,
\] 
and 
\[
J^{i}_{j}(y) = -J^{j}_{i}(y) \text{ for all }y \in V \cap \{y^3 = 0\}.
\]
\vskip 1mm
\item[(n4)] For a function $h: \RR^3 \to \RR$, we write $\widehat{h}(y)$ to mean $h(\Psi(y))$. For a vector field $X: \RR^3 \to \RR^3$, we let $\widehat{X}(y) = [(d\Psi)_{y}]^{-1}\big(X(\Psi(y))\big)$.
\vskip 1mm
\end{enumerate}
By (iii), there exist dimensional constants $C$ such that for $y \in V$ and $\xi, \eta, \zeta \in \RR^3$ we have
\[
\begin{split}
(1 - C\delta)|\xi|^2 \leq\ & g_{ij}(y)\xi^i \xi^j \leq (1 + C\delta)|\xi|^2,\\
P_{ijk}(y)\xi^i \eta^j \zeta^k \leq\ & (1 + C\delta)|\xi||\eta||\zeta|,
\end{split}
\]
Consequently, for all $\tau \in (-1, 1)$, there exists a choice of $\delta \in (0, 1)$ so that for all $U \in \cF$, with $V = \Phi_U(U)$, the following three estimates hold. First of all,
\begin{equation}\label{eq:norm-close}
\frac{1 - |\tau|}{2}|\xi|^2 \leq g_{ij}(y) \xi^i_\alpha \xi^j_\beta \delta_{\alpha\beta} - \tau P_{ijk}(y)\eta^k \xi^i_\alpha \xi^j_\beta \epsilon_{\alpha\beta} \leq 2|\xi|^2
\end{equation}
for all $y \in V$, $\xi = (\xi^i_{\alpha})_{1 \leq \alpha \leq 2, 1 \leq i \leq 3} \in \RR^{3 \times 2}$, and $\eta \in \RR^3$ such that $|\eta| \leq 1 + \delta$. Secondly, 
\begin{equation}\label{eq:J-close}
 |J(y)\xi| \leq \frac{1 + |\tau|}{2|\tau|}|\xi|, \text{ for all }y \in V,\ \ \xi \in \RR^3,
\end{equation}
where we are viewing $J(y)$ as a linear map from $\RR^3$ to $\RR^3$. Thirdly, 
\begin{equation}\label{eq:metric-close}
\frac{2 + 2|\tau|}{3 + |\tau|}|\xi|^2 \leq g_{ij}(y)\xi^i \xi^j \leq \frac{3 + |\tau|}{2 + 2|\tau|}|\xi|^2, \text{ for all }y \in V,\ \ \xi \in \RR^3.
\end{equation}
\begin{rmk}\label{rmk:rho-tau}
Below, whenever $\tau \in (-1, 1)$ is given, we fix such a $\delta \in (0, 1)$ and write $\rho_\tau$ for the corresponding $\rho$.
\end{rmk}
Besides locally flattening $\Sigma$, we will also have occasion to do the same with the boundary of $\bB$. When the boundary point of interest is $e_1 = (1, 0)$, thought of as $1 \in \CC$, we let $F$ be the conformal map from $(\RR^2_+, \partial \RR^2_+)$ onto $(\bB, \partial \bB \setminus \{-e_1\})$ that sends $2\sqrt{-1}$ to $0$ and sends $0$ to $1$. That is, take
\[
F(z) = \frac{2\sqrt{-1} - z}{2\sqrt{-1} + z}.
\]
Then there exists $r_{\bB} \in (0, \frac{1}{4})$ such that
\begin{equation}\label{eq:F-radius-comparable}
\bB_{\frac{3r}{4}}(e_1) \cap \bB \subset F(\bB_{r}^+) \subset \bB_{\frac{5r}{4}}(e_1) \cap \bB, \text{ for all }r \in (0, r_{\bB}],
\end{equation}
and that, letting $\lambda$ be denote the conformal factor so that $F^*g_{\RR^2} = \lambda^2 g_{\RR^2}$, we have
\begin{equation}\label{eq:lambda-bound}
\frac{3}{4} \leq \lambda(x) \leq \frac{5}{4}, \text{ for all }x \in \bB^+_{r_{\bB}}.
\end{equation}
To flatten $\partial \bB$ near other boundary points than $e_1$, we simply compose $F$ with rotations. 
\subsection{Defining the perturbed functional}\label{subsec:perturbed}
As seen in~\eqref{eq:perturbed-intro}, the perturbed functional consists of energy and volume terms. We start with an explanation of the latter. Let $h$ be a smooth, bounded, real-valued function on $\RR^3$ such that $\int_{\Omega}h > 0$. Below we briefly summarize the properties of the functional $V_h$, discussed in more detail in~\cite{Cheng22}. Given a continuous path $\gamma \in C^{0}([0, 1]; \cM_p)$, there is an induced continuous map
\[
G_\gamma: ([0, 1] \times \overline{\bB}, [0, 1] \times \partial \bB) \to (\RR^3, \Sigma).
\]
Since $p > 2$, when $\gamma$ happens to be a piecewise $C^1$ path, we may define
\begin{equation}\label{eq:volume-definition}
V_{h}(\gamma) = \int_{[0, 1] \times \bB} G_{\gamma}^*(h \cdot \vol_{\RR^3}).
\end{equation}
It is shown in~\cite[Lemma 2.1]{Cheng22} that, with $\delta_0$ as in Remark~\ref{rmk:h-geodesics}, if $\gamma, \widetilde{\gamma}:[0, 1] \to \cM_p$ are two piecewise $C^1$ paths such that $\gamma(0) = \widetilde{\gamma}(0)$, $\gamma(1) = \widetilde{\gamma}(1)$ and $\|\gamma(t) - \widetilde{\gamma}(t)\|_{\infty} < \delta_0$ for all $t \in [0, 1]$, then 
\[
V_h(\gamma) = V_h(\widetilde{\gamma}).
\]
Hence we may extend $V_h$ to a well-defined functional on continuous paths by approximation.
We refer to~\cite[Section 2.2]{Cheng22} for details on this extension, and for the proof of the following properties of the $V_h$ functional thus extended.
\begin{lemm}\label{lemm:volume-properties}
The functional $V_{h}$ has the following properties.
\begin{enumerate}
\item[(a)] Let $\gamma, \widetilde{\gamma} \in C^0([0, 1]; \cM_p)$ be two homotopic paths. Then $V_{h}(\gamma) = V_{h}(\widetilde{\gamma})$.
\vskip 1mm
\item[(b)] Given $\gamma, \widetilde{\gamma} \in C^0([0, 1]; \cM_p)$ such that $\gamma(0)$ and $\widetilde{\gamma}(0)$ are constant maps while $\gamma(1) = \widetilde{\gamma}(1)$, there exists $k \in \ZZ$ depending only on $\gamma$ and $\widetilde{\gamma}$ such that 
\[
V_{h}(\gamma) - V_{h}(\widetilde{\gamma}) = k \int_{\Omega}h,
\]
for any bounded smooth function $h:\RR^3 \to \RR$ with $\int_{\Omega}h > 0$.
\end{enumerate}
\end{lemm}
\begin{proof}
See~\cite{Cheng22}, Corollary 2.2 and Lemma 2.3, respectively, for parts (a) and (b).
\end{proof}
\begin{rmk}\label{rmk:volume-l}
In the notation of Remark~\ref{rmk:h-geodesics}, for all $u, v \in \cM_p$ such that $\|u - v\|_{\infty} < \delta_0$ we have by the definition~\eqref{eq:volume-definition} and the estimates~\eqref{eq:speed-pointwise} and~\eqref{eq:gradient-pointwise} that
\[
V_h(l_{u, v}) \leq  C \|h\|_{\infty} \|u - v\|_{\infty} \cdot (\|\nabla u\|_{2}^2 + \|\nabla v\|_{2}^2),
\]
where $C$ depends only on the constants in~\eqref{eq:speed-pointwise} and~\eqref{eq:gradient-pointwise}.
\end{rmk}
Moving on to the term $D_{\ep, p, \tau}$ in the perturbed functional, we let $X: \RR^3 \to \RR^3$ be a smooth, compactly supported vector field such that
\begin{equation}\label{eq:X-requirement-1}
X(y) = \bN(\Pi(y)) \text{ for all }y \text{ in a neighborhood of } \Sigma,
\end{equation}
and that
\begin{equation}\label{eq:X-requirement-2}
|X(y)| \leq 1 \text{ for all }y \in \RR^3.
\end{equation}
Given $\tau \in (-1, 1), \ep \in (0, 1]$ and $p \in (2, 3]$, for $y \in \RR^3$ and $\xi  = (\xi^i_{\alpha})_{1 \leq \alpha \leq 2, 1 \leq i \leq 3} \in \RR^{3 \times 2}$, we let
\[
\begin{split}
F_\tau(y, \xi) =\ & \frac{|\xi|^2}{2} - \frac{\tau}{2}  \mu_{ijk} X^k(y)  \xi^{i}_{\alpha} \xi^j_{\beta}\epsilon_{\alpha\beta},\\
G_\tau(y, \xi) =\ & (1 + 2F_\tau(y, \xi))^{\frac{p}{2}},
\end{split}
\]
where $\mu_{ijk}$ and $\epsilon_{\alpha\beta}$ are alternating and satisfy $\mu_{123} = 1 = \epsilon_{12}$. Then, for $u \in \cM_p$, we define
\[
D_{\ep, p, \tau}(u) =  \int_{\bB} [F_\tau(u, \nabla u) + \frac{\ep^{p-2}}{p}G_\tau(u, \nabla u) ]\ dx,
\]
and write $D_{\ep, p}$ for $D_{\ep, p, 0}$. That is, we set
\[
D_{\ep, p}(u) = \int_{\bB} [\frac{|\nabla u|^2}{2} + \frac{\ep^{p-2}}{p}(1 + |\nabla u|^2)^{\frac{p}{2}}] dx.
\]
Since $|X| \leq 1$ everywhere, from the definition of $F_\tau$ we have
\begin{equation}\label{eq:F-estimates}
\frac{1 - |\tau|}{2}|\xi|^2 \leq F_\tau(y, \xi) \leq \frac{1 + |\tau|}{2}|\xi|^2 \leq |\xi|^2
\end{equation}
for all $y \in \RR^3$ and $\xi \in \RR^{3 \times 2}$. As $p \in (2, 3]$, it follows that
\begin{equation}\label{eq:F-p-estimates}
(1 - |\tau|)^2 \leq  \frac{F_\tau(y, \xi) + \frac{\ep^{p-2}}{p}G_\tau(y, \xi)}{\frac{|\xi|^2}{2} + \frac{\ep^{p-2}}{p}(1 + |\xi|^2)^{\frac{p}{2}}}  \leq (1 + |\tau|)^2 \leq 4,
\end{equation}
again for all $y \in \RR^3$ and $\xi \in \RR^{3 \times 2}$, and, consequently,
\begin{equation}\label{eq:D-p-estimates}
(1- |\tau|)^2\cdot D_{\ep, p}(u) \leq D_{\ep, p, \tau}(u) \leq 4\cdot D_{\ep, p}(u).
\end{equation}
The lemma below collects two more estimates for $D_{\ep, p, \tau}$. Part (b) is used in the proof of Proposition~\ref{prop:mountain-pass}, which is the basic tool with which we extract critical points of the perturbed functional from suitable sequences of paths in $\cM_p$, while part (a) leads to the monotonicity property (see Section~\ref{subsec:extract}) that allows us to produce such paths.
\begin{lemm}\label{lemm:D-difference}
The functional $D_{\ep, p, \tau}$ has the following properties.
\begin{enumerate}
\item[(a)] For all $\tau_0 \in (0, 1)$, there exists a constant $b = b(\tau_0) > 0$ such that for all $\tau_1, \tau_2 \in [-\tau_0, \tau_0]$ satisfying $\tau_1 < \tau_2$ and $\tau_1\cdot \tau_2 > 0$, we have
\begin{equation}\label{eq:D-tau-difference}
\frac{b(\tau_0)}{\tau_0^2} \cdot D_{\ep, p}(u) \leq \frac{1}{\tau_2 - \tau_1}\Big(\frac{D_{\ep, p, \tau_1}(u)}{\tau_1} - \frac{D_{\ep, p, \tau_2}(u)}{\tau_2}\Big).
\end{equation}
\vskip 1mm
\item[(b)] For all $u, v \in \cM_p$ we have
\begin{equation}\label{eq:D-difference}
\big| D_{\ep, p, \tau}(u) - D_{\ep, p, \tau}(v) \big| \leq C [1 + D_{\ep, p}(u) + D_{\ep, p}(v)] \cdot \|u - v\|_{1, p},
\end{equation}
where $C$ depends only on $p$ and the $C^1$ norm of $X$.
\end{enumerate}
\end{lemm}
\begin{proof}
See Appendix~\ref{sec:differentiability}.
\end{proof}

We introduce one last notation before giving the definition of the perturbed functional. Given $u \in \cM_p$, we let
\[
\cE(u) = \{\gamma \in C^0([0, 1]; \cM_p)\ |\ \gamma(0) = \text{constant},\ \gamma(1) = u\}.
\]
That is, $\cE(u)$ is the space of continuous paths in $\cM_p$ leading from a constant map to $u$. It is shown in~\cite[Section 2.3]{Cheng22} that $\cE(u)$ is non-empty.  
\begin{defi}\label{defi:perturbed-definition}
Let $X: \RR^3 \to \RR^3$ and $f: \RR^3 \to \RR$ be smooth and compactly supported functions such that $X$ satisfies~\eqref{eq:X-requirement-1} and~\eqref{eq:X-requirement-2}, and that $\int_{\Omega}f > 0$, unless $f$ is identically zero. Also, suppose $\ep \in (0, 1], p \in (2, 3]$ and $\tau \in (-1, 1)$. For $u \in \cM_p$ and $\gamma \in \cE(u)$, we define 
\[
E_{\ep, p, f, \tau}(u, \gamma) = D_{\ep, p, \tau}(u) + V_{f}(\gamma) + \tau V_{\Div X}(\gamma).
\]
\end{defi}
\begin{rmk}
Further restrictions on $X$ and $f$ will be introduced when necessary. Also, it is to prove Theorem~\ref{thm:main-1} when $H = 0$ that we allow $f$ to be identically zero, in which case it is understood that $V_f(\gamma) = 0$ for all paths $\gamma$ in $\cM_p$. 
\end{rmk}

As explained in~\cite[Section 2.3]{Cheng22}, on each simply-connected neighborhood $\cA$ in $\cM_p$, choosing some $u_0 \in \cA$ and $\gamma_0 \in \cE(u_0)$ as a reference, we may define a functional on $\cA$ by
\[
E_{\ep, p, f, \tau}^\cA(u) = E_{\ep, p, f, \tau}(u, \gamma_0 + l_u),
\]
where $l_u$ is any path in $\cA$ leading from $u_0$ to $u$, and ``$+$'' denotes concatenation of paths. Since $\cA$ is simply-connected, by Lemma~\ref{lemm:volume-properties}(a) the choice of $l_u$ is irrelevant and $E_{\ep, p, f, \tau}^\cA$ is well-defined. We refer to $E_{\ep, p, f, \tau}^\cA$ as the local reduction of $E_{\ep, p, f, \tau}$ induced by $(u_0, \gamma_0)$ on $\cA$. By Lemma~\ref{lemm:volume-properties}(b), switching to a different reference pair alters $E_{\ep, p, f, \tau}^\cA$ by a constant integer multiple of $\int_{\Omega}(f + \tau \Div X )$. Nonetheless, we omit the dependence of $E_{\ep, p, f, \tau}^\cA$ on $(u_0, \gamma_0)$ from the notation since the choice is usually clear from the context. Finally, since $p > 2$, it is standard to verify that $E^{\cA}_{\ep, p, f, \tau}$ is twice continuously (Fr\'echet) differentiable on $\cA$.

We end this section by defining the first and second variation of $E_{\ep, p, f, \tau}$ and stating the first variation formula. Given $u \in \cM_p$, we choose an arbitrary simply-connected neighborhood $\cA$ of it, consider a local reduction $E^{\cA}_{\ep, p, f, \tau}$ on $\cA$, and define the first variation of $E_{\ep, p, f, \tau}$ at $u$ by 
\[
(\delta E_{\ep, p, f, \tau})_u = (\delta E^\cA_{\ep, p, f, \tau})_u.
\]
By Lemma~\ref{lemm:volume-properties}(b) this is well-defined, and we say that $u$ is a critical point of $E_{\ep, p, f, \tau}$ if $(\delta E_{\ep, p, f, \tau})_u = 0$, in which case $(\delta E_{\ep, p, f, \tau}^{\cA})_{u} = 0$ for any local reduction whose domain contains $u$, and we define the second variation of $E_{\ep, p, f, \tau}$ at $u$ by
\[
(\delta^2 E_{\ep, p, f, \tau})_u = (\delta^2 E^\cA_{\ep, p, f, \tau})_u.
\]
Again this does not depend on the choice of local reduction. Before giving the first variation formula as promised (Lemma~\ref{lemm:first-variation}), we recall two standard identities, one for immediate use and the other for Section~\ref{subsec:second-variation}. 
\begin{lemm}\label{lemm:X-identities}
Suppose $p > 2$. For all $u, \varphi \in W^{1, p}(\bB; \RR^3)$, we have
\begin{equation}\label{eq:volume-form-Lie}
\begin{split}
&(\Div X)_u\ \varphi \cdot u_{x^1} \times u_{x^2}\\
=\ &  (\nabla X)_u(\varphi) \cdot u_{x^1} \times u_{x^2} + \varphi \cdot (\nabla X)_u(u_{x^1}) \times u_{x^2} + \varphi \cdot u_{x^1} \times (\nabla X)_u(u_{x^2}),
\end{split}
\end{equation}
and 
\begin{equation}\label{eq:volume-form-Lie-2}
\begin{split}
&- \varphi \cdot \big[ (\nabla^2 X)_{u}(\varphi, u_{x^1}) + (\nabla X)_u(\varphi_{x^1})\big] \times u_{x^2} -  \varphi \cdot  u_{x^1} \times \big[ (\nabla^2X)_u(\varphi, u_{x^2}) +  (\nabla X)_u(\varphi_{x^2}) \big]\\
=\ & \varphi \cdot \big[  (\nabla X)_u(u_{x^1}) \times \varphi_{x^2} + \varphi_{x^1} \times (\nabla X)_u(u_{x^2})  \big]\\ 
&+  (\nabla^2 X)_u(\varphi, \varphi) \cdot u_{x^1} \times u_{x^2} + (\nabla X)_u(\varphi) \cdot (\varphi_{x^1} \times u_{x^2} + u_{x^1} \times \varphi_{x^2})\\
& - [\nabla (\Div X)]_u(\varphi)\ \varphi \cdot u_{x^1} \times u_{x^2} -  (\Div X)_u\ \varphi \cdot (\varphi_{x^1} \times u_{x^2} + u_{x^1} \times \varphi_{x^2}),
\end{split}
\end{equation}
almost everywhere in $\bB$.
\end{lemm}
\begin{proof}
It suffices to prove the identities when $u, \varphi$ are in $C^1(\overline{\bB}; \RR^3)$. The identity~\eqref{eq:volume-form-Lie} results from differentiating 
\[
(e^{tA})^*\vol_{\RR^3} =  \det (e^{tA}) \cdot \vol_{\RR^3}
\]
at $t = 0$ with $A$ being $(\nabla X)_{u}$ evaluated at the point of interest. To obtain~\eqref{eq:volume-form-Lie-2}, we replace $u$ by $u+ t\varphi$ in~\eqref{eq:volume-form-Lie}, differentiate with respect at $ t=0$, and rearrange.
\end{proof}
\begin{lemm}\label{lemm:first-variation}
Letting $h = f + \tau \Div X$, the first variation formula of $E_{\ep, p, f, \tau}$ is given by 
\begin{equation}\label{eq:first-variation}
\begin{split}
(\delta E_{\ep, p, f, \tau})_u(\varphi) =\ & \int_{\bB} [1 + \ep^{p-2}(1 + 2F_{\tau}(u, \nabla u))^{\frac{p}{2} - 1}] \big( \bangle{\nabla u, \nabla \varphi} - \tau  X(u) \cdot (\varphi_{x^1} \times u_{x^2} + u_{x^1} \times \varphi_{x^2}) \big)\\
& - \tau \int_{\bB}[1 + \ep^{p-2}(1 + 2F_{\tau}(u, \nabla u))^{\frac{p}{2} - 1}]   (\nabla X)_u(\varphi) \cdot u_{x^1} \times u_{x^2}\\
& + \int_{\bB} h(u)\varphi \cdot u_{x^1} \times u_{x^2},
\end{split}
\end{equation}
for all $\varphi \in T_u\cM_p$. Consequently, defining $\nabla^\perp u = (-u_y, u_x)$, smooth critical points of $E_{\ep, p, f, \tau}$ satisfy
\begin{equation}\label{eq:EL}
\begin{split}
&\Delta u + (p-2) \cdot \frac{\ep^{p-2}(1 + 2F_\tau(u, \nabla u))^{\frac{p}{2} - 1} }{1 + \ep^{p-2}(1 + 2F_\tau(u, \nabla u))^{\frac{p}{2} - 1}}\cdot \Big\langle \frac{\nabla (F_{\tau}(u, \nabla u))}{1 + 2F_{\tau}(u, \nabla u)}, \nabla u - \tau X(u) \times \nabla^\perp u\Big\rangle\\
=\ &\Big( \frac{h(u)}{1 + \ep^{p-2}(1 + 2F_\tau(u, \nabla u))^{\frac{p}{2} - 1}} - \tau \Div X(u) \Big)\ u_{x^1} \times u_{x^2} \text{ in }\bB,
\end{split}
\end{equation}
along with the boundary condition
\begin{equation}\label{eq:EL-bc}
(u_r + \tau X(u) \times u_\theta) \perp T_u \Sigma \text{ on }\partial \bB.
\end{equation}
\end{lemm}
\begin{proof}
The formula~\eqref{eq:first-variation} is obtained by direct computation. The equation~\eqref{eq:EL} and boundary condition~\eqref{eq:EL-bc} follow upon integrating by parts with the help of the first identity in Lemma~\ref{lemm:X-identities}.
\end{proof}
\subsection{Regularity of critical points}\label{subsec:smoothness}
Let $H_1 >0$ be chosen so that
\[
\sup_{\RR^3}(|\nabla X| + |\nabla^2 X|) + \sup_{\RR^3}( |f| + |\nabla f|) \leq H_1.
\]
Below, we prove that critical points of $E_{\ep, p, f, \tau}$ are smooth on $\overline{\bB}$ provided $p$ is sufficiently close to $2$. 
\begin{prop}[Interior regularity]\label{prop:interior-regularity}
There exists $p_0 \in (2, 3]$ depending only on $\tau$ such that if $u \in \cM_p$ is a critical point of $E_{\ep, p, f, \tau}$ for some $p \in (2, p_0]$, then $u$ is smooth in $\bB$.
\end{prop}
\begin{proof}
The divergence and non-divergence forms of the Euler-Lagrange equation of $E_{\ep,p,f, \tau}$, derived in Lemma~\ref{lemm:first-variation}, allows us to obtain interior regularity in essentially the same way as in~\cite{Sacks-Uhlenbeck1981}. The details are omitted.
\end{proof}
We next address boundary regularity. Suppose $u \in \cM_p$ is a critical point of $E_{\ep, p, f, \tau}$ and take $x_0 \in \partial \bB$, which we assume without loss of generality to be $e_1 := (1, 0)$. To keep track of the scale on which we have estimates, we choose $r \in (0, r_\bB)$ and $L > 0$ such that 
\begin{equation}\label{eq:W1p-scale-invariant-bound}
r^{p-2} \int_{\bB_r(e_1) \cap \bB} |\nabla u|^p \leq L^p.
\end{equation}
Then, with $\mu = 1 - \frac{2}{p}$, Sobolev embedding gives
\[
r^{\mu}[u]_{\mu; \bB_r(e_1) \cap \bB} \leq S_pr^{1 - \frac{2}{p}} \|\nabla u\|_{p; \bB_r(e_1) \cap \bB}.
\]
Therefore, given $\sigma \in (0, 1)$, whenever $z \in \bB_{\sigma r}(e_1) \cap \bB$, we deduce from~\eqref{eq:W1p-scale-invariant-bound} that
\[
|u(z) - u(e_1)| \leq |z - e_1|^\mu [u]_{\mu; \bB_{r}(e_1) \cap \bB} \leq  S_{p}\sigma^\mu L.
\]
Thus, choosing $\sigma_1 \in (0, \frac{1}{4})$ such that
\begin{equation}\label{eq:sigma-1-definition}
S_p (2\sigma_1)^\mu L < \rho_{\tau},
\end{equation}
where $\rho_{\tau}$ is the constant defined in Remark~\ref{rmk:rho-tau}, we see that $u(\bB_{2\sigma_1 r}(e_1) \cap \bB)$ is contained in some $U \in \cF$, so that we may introduce a chart $\Phi: U \to V$ as described in Section~\ref{subsec:function-spaces} and consider the objects $\Psi, g, P, \widehat{X}$, and so on. In addition, we define
\[
B_y = [(d\Psi)_y]^{-1} \circ (\nabla X)_{\Psi(y)} \circ (d\Psi)_y \ \ \text{ and }\ \   C_y = [(d\Psi)_y]^{-1} \circ (\nabla^2 \Psi)_y,
\]
so that each $B_y$ is a linear map from $\RR^3$ to $\RR^3$, and each $C_y$ is a symmetric bilinear map on $\RR^3$ taking values in $\RR^3$. Note that since $X$ has length at most $1$ and agrees with $\bN \circ \Pi$ on a neighborhood of $\Sigma$, we have by~\eqref{eq:isometry-deviation} that
\begin{equation}\label{eq:X-hat-length}
|\widehat{X}(y)| \leq 1 + \delta, \text{ for all }y \in V,
\end{equation}
and that $\widehat{X}= -\paop{y^3}$ near $\{y^3 = 0\}$ in $V$. Regarding the map $u$ itself, we let
\[
\widehat{u}(x) = \Phi(u(x)), \text{ for }x \in \bB_{2\sigma_1r}(e_1) \cap \bB,
\]
and, with $F$ being the conformal map introduced at the end of Section~\ref{subsec:function-spaces},
\begin{equation}\label{eq:v-definition}
v(x) = \widehat{u}(F(rx))\ \ \text{ and }\ \  \blambda(x) = \lambda(rx), \text{ for }x \in \bB^+_{\sigma_1},
\end{equation}
where we recall that $\lambda$ is the conformal factor such that $F^*g_{\RR^2} = \lambda^2 g_{\RR^2}$. Finally we write 
\[
\bep = \frac{\ep}{r}.
\]

To express the Euler-Lagrange equation in terms of $v$ in a more manageable way, for $(x, z, \xi) \in \bB_{\sigma_1}^+ \times V \times \RR^{3 \times 2}$, we define
\[
N(x, z, \xi) = (\blambda(x))^{-2}\big( g_{ij}(z) \delta_{\alpha\beta} - \tau  P_{ijk}(z)\widehat{X}^k(z)  \epsilon_{\alpha\beta} \big) \xi^i_{\alpha} \xi^j_{\beta},
\]
and let
\begin{equation}\label{eq:E-L-coefficients}
\begin{split}
A_{i\alpha}(x, z, \xi) = \ & [1 + (\bep)^{p-2}(r^2 + N)^{\frac{p}{2} - 1}] \cdot \big( g_{ij}(z)\xi^j_\alpha - \tau P_{ijk}(z)\widehat{X}^k(z)\xi^j_\beta \epsilon_{\alpha\beta} \big),\\
A_i(x, z, \xi)=\ & \frac{1}{2}\widehat{h}(z) P_{ijk}(z) \xi^j_\alpha \xi^k_\beta \epsilon_{\alpha\beta}+ [1 + (\bep)^{p-2}(r^2 + N)^{\frac{p}{2} - 1}] \cdot \frac{1}{2}g_{jk,i}(z)\xi^j_{\alpha} \xi^k_{\beta}\delta_{\alpha\beta}\\
& -\tau [1 + (\bep)^{p-2}(r^2 + N)^{\frac{p}{2} - 1}] P_{mjk}(z) \big( C^m_{il}(z) \widehat{X}^k(z)\xi^l_\alpha\xi^j_\beta + \frac{1}{2}B_{im}(z) \xi^j_\alpha \xi^k_\beta\big) \epsilon_{\alpha\beta},
\end{split}
\end{equation}
where $\widehat{h} = (f + \tau \Div X) \circ \Psi$. For use in the non-divergence form of the Euler-Lagrange equation, we define
\begin{equation}\label{eq:E-L-coefficients-2}
\begin{split}
E^{ij}_{\alpha\beta}(x, z, \xi) =\ & \frac{g^{il}(z)\pa{A_{l\alpha}}{\xi^j_\beta}(x, z, \xi)}{1 + \bep^{p-2}(r^2 + N)^{\frac{p}{2}-1}} - \big(  \delta_{ij}\delta_{\alpha\beta}  - \tau P^i_{jk}(z)\widehat{X}^k(z)\epsilon_{\alpha\beta} \big),
\\
b_i(x, z, \xi) = \ & \frac{g^{il}(z)\big(A_l(x, z, \xi) - \pa{A_{l\alpha}}{z^k}(x, z, \xi)\xi^{k}_{\alpha} - \pa{A_{l\alpha}}{x^\alpha}(x, z, \xi) \big)}{1 + \bep^{p-2}(r^2 + N)^{\frac{p}{2} - 1}}.
\end{split}
\end{equation}
The next two lemmas are both proved in Appendix~\ref{sec:structure}, with Lemma~\ref{lemm:coefficient-bounds} containing estimates on the coefficients that are crucial to the proof of boundary regularity.
\begin{lemm}\label{lemm:first-variation-simplified}
Let $u \in \cM_p$ be a critical point of $E_{\ep, p, f, \tau}$ for some $p \in (2, p_0]$. Then, in the above notation and with $v$ as defined by~\eqref{eq:v-definition} and $\sigma_1$ chosen to satisfy~\eqref{eq:sigma-1-definition}, we have
\begin{equation}\label{eq:E-L-fermi-2}
\int_{\bB^+_{\sigma_1}} A_{i\alpha}(x, v, \nabla v) \partial_{\alpha}\psi^i + A_i(x, v, \nabla v) \psi^i = 0,
\end{equation}
for all $\psi \in W^{1, p}(\bB_{\sigma_1}^+;\RR^3)$ such that $\psi = 0$ on $\partial \bB_{\sigma_1} \cap \RR^2_+$ and $\psi^3 = 0$ on $\bT_{\sigma_1}$. Moreover, pointwise in $\bB^+_{\sigma_1}$ we have
\begin{equation}\label{eq:E-L-fermi-non-div}
\Delta v^i + E_{\alpha\beta}^{ij}(x, v, \nabla v)\partial_{\alpha\beta}v^j = b_i(x, v, \nabla v).
\end{equation}
\end{lemm}
\begin{proof}
See Appendix~\ref{sec:structure}.
\end{proof}
\begin{lemm}\label{lemm:coefficient-bounds}
For all $x \in \bB^+_{\sigma_1}$, $z \in V$ and $\xi \in \RR^{3 \times 2}$, the following bounds hold:
\vskip 1mm
\begin{enumerate}
\item[(a)] 
\begin{equation}\label{eq:N-norm}
\frac{1-|\tau|}{4}|\xi|^2 \leq  N(x, z, \xi) \leq 4|\xi|^2.
\end{equation}
\vskip 1mm
\item[(b)] 
\begin{equation}\label{eq:leading-coefficient}
\begin{split}
A_{i\alpha}\xi^{i}_\alpha \geq\ & \frac{(1 - |\tau|)^2}{8} \cdot [1 + (\bep)^{p-2}(r^2 + |\xi|^2)^{\frac{p}{2} - 1}]  |\xi|^2,\\
\pa{A_{i\alpha}}{\xi^{j}_\beta}\eta^i_\alpha \eta^j_\beta \geq\ & \frac{(1 - |\tau|)^2}{8} \cdot [1 + (\bep)^{p-2}(r^2 + |\xi|^2)^{\frac{p}{2} - 1}]  |\eta|^2.
\end{split}
\end{equation}
\vskip 1mm
\item[(c)] 
\begin{equation}\label{eq:coefficient-bounds}
\begin{split}
|A_{i\alpha}| + |\xi| \Big|  \pa{A_{i\alpha}}{\xi^j_\beta}  \Big| + \Big|  \pa{A_{i\alpha}}{z^m}  \Big| + \Big| \pa{A_{i\alpha}}{x^\gamma} \Big| \leq\ & C_{\tau, H_1} [1 + (\bep)^{p-2}(r^2 + |\xi|^2)^{\frac{p}{2} - 1}] |\xi|,\\
|A_{i}| + |\xi| \Big|  \pa{A_{i}}{\xi^j_\beta}  \Big| + \Big|  \pa{A_{i}}{z^m}  \Big| + \Big| \pa{A_i}{x^\gamma} \Big| \leq\ & C_{\tau, H_1}[1 + (\bep)^{p-2}(r^2 + |\xi|^2)^{\frac{p}{2} - 1}] |\xi|^2.
\end{split}
\end{equation}
\vskip 1mm
\item[(d)] 
\begin{equation}\label{eq:E-b-e-bound}
\begin{split}
|E^{ij}_{\alpha\beta}| \leq\ & C_{\tau}(p-2),\ \ |b_i| \leq C_{\tau, H_1}(|\xi| + |\xi|^2).
\end{split}
\end{equation}
\end{enumerate}
\end{lemm}
\begin{proof}
See Appendix~\ref{sec:structure}.
\end{proof}
\begin{prop}[Boundary regularity]\label{prop:boundary-regularity}
There exists $p_1 \in (2, p_0]$ depending only on $\tau$ so that if $u \in \cM_p$ is a critical point of $E_{\ep, p, f, \tau}$ for some $p \in (2, p_1]$ and if $x_0 \in \partial \bB$ is any boundary point, then $u$ is smooth on a neighborhood of $x_0$ relative to $\overline{\bB}$.
\end{prop}
\begin{proof}
The structure of the proof is largely the same as the corresponding result in~\cite{Cheng22}. Nonetheless we include the details to show how Appendix~\ref{sec:linear} is relevant. We prove in two steps that $u$ is $W^{2, 4}$ near $x_0$. Higher regularity is then standard.
\vskip 1mm
\noindent\textbf{Step 1: $W^{2, 2}$-estimates}
\vskip 1mm
As mentioned above, without loss of generality we assume that $x_0 = e_1$. Choosing $r \in (0, r_{\bB})$ and $L > 0$ such that~\eqref{eq:W1p-scale-invariant-bound} holds, we define $\sigma_1$ by~\eqref{eq:sigma-1-definition} and define $v$ in~\eqref{eq:v-definition}. Then $v^3 = 0$ on $\bT_{\sigma_1}$. Also, by~\eqref{eq:W1p-scale-invariant-bound},~\eqref{eq:lambda-bound} and~\eqref{eq:isometry-deviation} we have
\begin{equation}\label{eq:W1p-bound-for-v}
\int_{\bB^+_{\sigma_1}} |\nabla v|^p \leq Cr^{p-2}\int_{\bB_{2\sigma_1 r}(e_1) \cap \bB} |\nabla u|^{p-2} \leq CL^p,
\end{equation}
for some universal constant $C$. Thus, letting $\mu  = 1 - \frac{2}{p}$, we see from Sobolev embedding that 
\begin{equation}\label{eq:v-osc-bound}
\osc_{\bB^+_{\sigma} }v \leq C_{p}L\cdot \sigma^\mu,
\end{equation}
which together with Lemma~\ref{lemm:coefficient-bounds} puts the equation~\eqref{eq:E-L-fermi-2} into the class treated in Chapter 4 of the monograph~\cite{LadyzhenskayaUraltseva1968} by Ladyzhenskaya and Ural'tseva. In particular, substituting test functions of the form $\psi = (v - v(0))\zeta^2$ into~\eqref{eq:E-L-fermi-2}, which are admissible since $v^3 = 0$ on $\bT_{\sigma_1}$, we see using~\eqref{eq:v-osc-bound} and Lemma~\ref{lemm:coefficient-bounds} and following the arguments in~\cite[Chapter 4, Lemma 1.3]{LadyzhenskayaUraltseva1968} that there exist $C > 0$ and $\sigma_2 < \sigma_1$ depending only on $\tau, H_1, p$ and $L$ such that whenever $\rho \leq \sigma_2$ and $\zeta \in W^{1, p}_{0}(\bB_\rho)$, we have
\begin{equation}\label{eq:L-U}
\int_{\bB^+_{\rho}} [1 + \bep^{p-2}(r^2 + |\nabla v|^2)^{\frac{p}{2} - 1}] |\nabla v|^2 \zeta^2 \leq C\rho^{2\mu} \int_{\bB^+_{\rho}} [1 + \bep^{p-2}(r^2 + |\nabla v|^2)^{\frac{p}{2} - 1}] |\nabla \zeta|^2.
\end{equation}
Thanks to~\eqref{eq:L-U} and Lemma~\ref{lemm:coefficient-bounds}, by a standard difference quotient argument (see~\cite[Chapter 4, Section 5]{LadyzhenskayaUraltseva1968}) we deduce from~\eqref{eq:E-L-fermi-2} the following: There exists $\sigma_3 < \frac{\sigma_2}{2}$ such that
\begin{equation}\label{eq:2nd-derivative-tangential}
\begin{split}
\int_{\bB^+_{\sigma_3}} (1 + \bep^{p-2}(r^2 + |\nabla v|^2)^{\frac{p}{2} - 1}) |\nabla v_{x^1}|^2 \leq\ & C\int_{\bB^+_{2\sigma_3}} |\nabla v|^2 + \bep^{p-2}(r^2 + |\nabla v|^2)^{\frac{p}{2}}\\
\leq\ & C \cdot D_{\ep, p}(u; \bB_{4\sigma_3 r}(e_1) \cap \bB).
\end{split}
\end{equation}
Here again the threshold $\sigma_3$ and the constants $C$ in~\eqref{eq:2nd-derivative-tangential} depend only on $H_1, \tau, p, L$. At this point we recall~\eqref{eq:E-L-fermi-non-div} from Lemma~\ref{lemm:first-variation-simplified} and rearrange it to get
\[
\Big( \delta_{ij} + E^{ij}_{22}(x, v, \nabla v) \Big)v^j_{x^2 x^2} = b_i(x, v, \nabla v) - v^i_{x^1 x^1} - \sum_{(\alpha, \beta) \neq (2, 2)} E_{\alpha\beta}^{ij}(x, v, \nabla v)\partial_{\alpha\beta}v^j.
\]
By Lemma~\ref{lemm:coefficient-bounds}(d), provided $p_1 - 2$ is sufficiently small depending on $\tau$, we may invert the matrix $\big(\delta_{ij} + E^{ij}_{22}(x, v, \nabla v) \big)_{1 \leq i, j \leq 3}$ and add $|\nabla v_{x^1}|$ to both sides to deduce that
\begin{equation}\label{eq:v-pointwise-hessian}
|\nabla^2 v| \leq C_{H_1, \tau}(|\nabla v| + |\nabla v|^2 + |\nabla v_{x^1}|).
\end{equation}
Thus, if $q \in [p, 4)$ is such that $\nabla v \in L^{q}$, we have upon taking the $L^{\frac{q}{2}}$-norm on $\bB^+_{\sigma_{3}}$ of both sides of~\eqref{eq:v-pointwise-hessian} and using~\eqref{eq:2nd-derivative-tangential} as well as the Sobolev embedding $W^{1, \frac{q}{2}} \to L^{\frac{2q}{4 - q}}$ that
\begin{equation}\label{eq:to-be-iterated}
\begin{split}
\|\nabla v\|_{\frac{2q}{4 - q}; \bB^+_{\sigma_3}} \leq\ & C_{q, p, H_1, \tau, L} \cdot \big( \|\nabla v\|_{\frac{q}{2}; \bB^+_{\sigma_3}} + \|\nabla v\|_{q;\bB^+_{\sigma_3}}^2 + D_{\ep, p}(u; \bB_{4\sigma_3 r}(e_1) \cap \bB) \big).
\end{split}
\end{equation}
Since $p > 2$ and since~\eqref{eq:W1p-bound-for-v} bounds $\|\nabla v\|_{p; \bB^+_{\sigma_3}}$, we may start with $q = p$ and iterate~\eqref{eq:to-be-iterated} for a finite number of times depending on $p$ to obtain a bound on $\|\nabla v\|_{4; \bB^+_{\sigma_3}}$ and, via taking the $L^2$-norm of both sides of~\eqref{eq:v-pointwise-hessian}, a bound on $\|\nabla^2 v\|_{2; \bB^+_{\sigma_3}}$. The resulting estimate has the form
\begin{equation}\label{eq:W14-W22-estimate}
\|\nabla v\|_{4; \bB^+_{\sigma_3}} + \| \nabla^2 v \|_{2; \bB^+_{\sigma_3}} \leq C(p, H_1, \tau, L, D_{\ep, p}(u; \bB_{4\sigma_3 r}(e_1) \cap \bB)).
\end{equation}
In particular, $A_{i\alpha}(x, v, \nabla v)$ lies in $(L^q \cap W^{1, r})(\bB^+_{\sigma_3})$ for all $q < \infty$ and $r < 2$. Using $(\cdot)|_{\bT_{\sigma_3}}$ to denote the Sobolev trace, we deduce from~\eqref{eq:E-L-fermi-2} and~\eqref{eq:E-L-fermi-non-div} that $A_{i, 2}(x, v, \nabla v)|_{\bT_{\sigma_3}} = 0$ for $i = 1, 2$. Recalling also that $\widehat{X} = -\paop{y^3}$ near $V \cap \{y^3 = 0\}$, and hence $P_{ijk}(v)\widehat{X}^k(v) = -P_{ij3}(v)$ near $\bT_{\sigma_3}$, we get
\[
\begin{split}
& \sum_{j = 1}^2 \big[g_{ij}(v)v^j_{x^2} - \tau  P_{ij3}(v)v^j_{x^1} \big]\big|_{\bT_{\sigma_3}} = 0\ \text{ for }i = 1, 2.
\end{split}
\]
Here $j$ need only be summed up to $2$ because $g_{i3}$ and $P_{i33}$ both vanish identically on $V$. As $[g_{ij}(y)]_{1 \leq i, j \leq 2}$ is invertible, it follows that
\begin{equation}\label{eq:E-L-fermi-bc}
\begin{split}
& \sum_{j = 1}^2 \big[v^i_{x^2} + \tau  J^i_j(v)v^j_{x^1} \big]\big|_{\bT_{\sigma_3}}  = 0\ \text{ for } i = 1, 2,
\end{split}
\end{equation}
where recall that $J^{i}_j = -g^{il}P_{lj3}$.
\vskip 1mm
\noindent\textbf{Step 2: $W^{2, 4}$-regularity}
\vskip 1mm
To continue, with $s < \sigma_3$ to be determined, we let $\eta$ be a cutoff function such that $\eta = 1$ on $\bB_{\frac{s}{2}}$ and vanishes outside of $\bB_{s}$, and that $|\nabla \eta| \leq Cs^{-1}$, with $C$ being a dimensional constant. For any $\zeta \in C^\infty_c(\bB_{\frac{s}{2}})$, by~\eqref{eq:E-L-fermi-non-div} we have
\begin{equation}\label{eq:E-L-non-div-cutoff}
\begin{split}
&\Delta (\zeta v^i) + \eta \cdot E^{ij}_{\alpha\beta}(x, v, \nabla v)\partial_{\alpha\beta}(\zeta v^j) - (\zeta v^i) \\
=\ & \zeta \cdot b_i(x, v, \nabla v) - \zeta v^i + 2\nabla \zeta \cdot \nabla v^i + v^i\Delta \zeta \\
& + E^{ij}_{\alpha\beta}(x, v, \nabla v)\big( \partial_{\alpha}\zeta \partial_\beta v^j +  \partial_{\beta}\zeta \partial_{\alpha} v^j + v^j \partial_{\alpha\beta}\zeta \big) =: \widetilde{f}_i.
\end{split}
\end{equation}
As for the boundary conditions, we have $\zeta v^3 = 0$ on $\partial \RR^2_+$. Also, letting
\[
e_{ij}(z) = \tau \big[J^i_j(z) - J^i_j(v(0))\big],
\]
\begin{equation}\label{eq:B0-definition}
(B_0 w)^i = w^i_{x^2} + \tau \sum_{j = 1}^2 J^i_j(v(0)) w^j_{x^1}, \text{ for }i = 1, 2,
\end{equation}
and denoting by $(\cdot)_{\bB^+_{s}}$ the average over $\bB^+_{s}$, we have from~\eqref{eq:E-L-fermi-bc} that, on $\partial \RR^2_+$ and for $i = 1, 2$, 
\begin{equation}\label{eq:E-L-bc-cutoff}
\begin{split}
& [B_0(\zeta v)]^i + \eta \cdot \sum_{j = 1}^2 e_{ij}(v) [ (\zeta v^j)_{x^1} - ((\zeta v^j)_{x^1})_{\bB^+_{s}}]\\
=\ & -\eta \cdot \sum_{j = 1}^2 e_{ij}(v)((\zeta v^j)_{x^1})_{\bB^+_{s}} + v^i \zeta_{x^2} + \tau \sum_{j = 1}^2 J^i_j(v)v^j\zeta_{x^1} =:\widetilde{g}_i.
\end{split}
\end{equation}
Below, whenever $e_{ij}$ and $J_i^j$ appear, it is understood that $i, j$ range from $1$ to $2$, and we omit the summation symbols in $j$.

We next reformulate~\eqref{eq:E-L-non-div-cutoff} and~\eqref{eq:E-L-bc-cutoff} as a fixed point condition and use the contraction mapping principle to improve the regularity of $\zeta v$. Specifically, for $q \geq 2$ we introduce the following Banach spaces:
\[
\cX_q = \{w \in W^{2, q}(\RR^2_+; \RR^3)\ \big|\ (w^3)|_{\partial \RR^2_+} = 0 \},\ \ \ \cY_q = L^q(\RR^2_+; \RR^3) \times W^{1- \frac{1}{q}, q}(\partial \RR^2_+; \RR^2),
\]
where $\cY_q$ is normed as in~\eqref{eq:Zp-norm}. Since $v^3(0) = 0$, by the observations made about $J$ in Section~\ref{subsec:function-spaces} and the estimate~\eqref{eq:J-close}, the operator $B_0$ defined in~\eqref{eq:B0-definition} belongs to the class described in Definition~\ref{defi:B0-defi}, with $|a| \leq \frac{1 + |\tau|}{2}$. Thus, by Lemma~\ref{lemm:inhomogeneous-estimate} and Lemma~\ref{lemm:inhomogeneous}, as well as standard $W^{2, p}$-theory for the Dirichlet problem, the bounded linear operator
\[
(\Delta -1, B_0|_{\partial \RR^2_+}): w \mapsto (\Delta w - w, (B_0w)|_{\partial \RR^2_+})
\] 
from $\cX_q$ to $\cY_q$ has a bounded inverse. Next, in accordance with~\eqref{eq:E-L-non-div-cutoff} and~\eqref{eq:E-L-bc-cutoff}, we define
\[
\begin{split}
(Lw)^i =\ & \Delta w^i + \eta \cdot E^{ij}_{\alpha\beta}(x, v, \nabla v)\partial_{\alpha\beta}w^j - w^i, \text{ for }i = 1, 2, 3,\\
(Bw)^i =\ & (B_0w)^i + \eta \cdot  e_{ij}(v)[w^j_{x^1} - (w^j_{x^1})_{\bB^+_{s}}], \text{ for }i = 1, 2.
\end{split}
\]
\begin{claim*}
The assignment
\[
(L, B|_{\partial \RR^2_+}): w \mapsto (Lw, (Bw)|_{\partial \RR^2_+})
\] 
defines a bounded linear operator from $\cX_q$ to $\cY_q$. Moreover, for all $w \in \cX_q$ we have
\begin{equation}\label{eq:fixed-point-estimate-1}
\|(\Delta - 1 - L)w\|_{q; \RR^2_+} \leq C_{\tau}(p-2)\|\nabla^2 w\|_{q; \RR^2_+},
\end{equation}
and
\begin{equation}\label{eq:fixed-point-estimate-2}
\| (B_0 - B)w \|_{1, q; \RR^2_+}
\leq \big( C_{p,\tau, q}L s^{1 - \frac{2}{p}}   +  C_{\tau ,q} s^{1 - \frac{1}{q}} \|\nabla v\|_{2q; \bB^+_{s}} \big) \|\nabla^2 w\|_{q; \RR^2_+}.
\end{equation}
\end{claim*}
\noindent We justify the claim at the end of this proof. Assuming it for now, we define $T_q: \cX_q \to \cX_q$ by
\[
T_qw = w + (\Delta - 1, B_0|_{\partial \RR^2_+})^{-1}\big( \widetilde{f} - Lw, (\widetilde{g} - Bw)|_{\partial \RR^2_+} \big),
\]
which makes sense since the pair $(\widetilde{f}, \widetilde{g}|_{\partial\RR^2_+})$, defined in~\eqref{eq:E-L-non-div-cutoff} and~\eqref{eq:E-L-bc-cutoff}, lies in $\cY_q$ for all $q \geq 2$ thanks to~\eqref{eq:W14-W22-estimate} and Sobolev embedding. This fact about $(\widetilde{f}, \widetilde{g}|_{\partial\RR^2_+})$ together with Lemma~\ref{lemm:intersection} also shows that given $q, r \geq 2$, both $T_q$ and $T_r$ preserve $\cX_q \cap \cX_r$, and in fact the two operators agree on the intersection. Thus, below we drop the subscripts and simply write $T$ for $T_q$. Finally, note that $\zeta v$ is a fixed point of $T$ on $\cX_2$.

To determine the choice of $s$, observe that given $w_1, w_2 \in \cX_q$, we have
\begin{equation}\label{eq:contraction-bvp}
\left\{
\begin{array}{l}
(\Delta - 1)(Tw_1 - Tw_2) = (\Delta -1-L)(w_1 - w_2),\\
\big[B_0 (Tw_1 - Tw_2)\big]\big|_{\partial \RR^2_+} = \big[(B_0 - B)(w_1 - w_2)\big]\big|_{\partial \RR^2_+},
\end{array}
\right.
\end{equation}
so that by Lemma~\ref{lemm:inhomogeneous-estimate} with $p = q$ and $a_0 = \frac{1 + |\tau|}{2}$, as well as the estimates in the above claim, we get
\[
\begin{split}
\|Tw_1 - Tw_2\|_{2, q} \leq\ & C_{q, \tau}(p-2) \|\nabla^2 w_1 - \nabla^2 w_2\|_{q; \RR^2_+}\\
& + \big( C_{p,\tau, q}L s^{1 - \frac{2}{p}}   +  C_{\tau ,q} s^{1 - \frac{1}{q}} \|\nabla v\|_{2q; \bB^+_{s}} \big) \|\nabla^2 w_1 - \nabla^2 w_2\|_{q; \RR^2_+}.
\end{split}
\]
Decreasing $p_1$ so that in the first term on the right hand side we have
\[
(C_{2, \tau} + C_{4, \tau}) (p_1 - 2) < \frac{1}{4}, 
\]
and recalling that~\eqref{eq:W14-W22-estimate} provides a bound on $\|\nabla v\|_{2q; \bB^+_{\sigma_3}}$, we can choose $s$ small enough so that for both $q = 2$ and $q = 4$ we have
\[
\|Tw_1  - Tw_2\|_{2, q} \leq \frac{1}{2}\|w_1 -w_2\|_{2, q}, \text{ for all }w_1, w_2 \in \cX_q.
\]
Combining this with the fact that $T$ preserves $\cX_2 \cap \cX_4$, we conclude that the unique fixed point of $T$ on $\cX_2$, which must be $\zeta v$, in fact lies in $\cX_2 \cap \cX_4$. Consequently $v \in W^{2, 4}_{\loc}(\bB^+_{\frac{s}{2}} \cup \bT_{\frac{s}{2}})$ since $\zeta \in C^{\infty}_c(\bB_{\frac{s}{2}})$ is arbitrary. 

The coefficients $E(x, v, \nabla v)$ and $J(v)$ in~\eqref{eq:E-L-fermi-non-div} and~\eqref{eq:E-L-fermi-bc} are now sufficiently regular for us to begin inductively improving the differentiability of $v$ using linear estimates in a way that does not require further decreasing $p_1$. The details of this last step are omitted.
\end{proof}
\begin{proof}[Proof of Claim]
To see that $(L, B|_{\partial \RR^2_+})$ is bounded from $\cX_q$ to $\cY_q$, the only nontrivial point is to estimate $\eta \cdot e(v)[w_{x^1} - (w_{x^1})_{\bB^+_{s}}]$ in $W^{1, q}$ by the $W^{2, q}$ norm of $w$, which can be achieved with the help of the Poincar\'e inequality. More precisely, recalling our choice of $\eta$ and noting that on $\bB^+_{s}$ we have
\[
|e_{ij}(v)| \leq C_{p, \tau}L  s^{1 - \frac{2}{p}},\ \ \big| \nabla (e_{ij}(v)) \big| \leq C_{\tau} |\nabla v|,
\]
the first inequality being a consequence of~\eqref{eq:v-osc-bound}, we can estimate as follows:
\[
\begin{split}
\| \eta \cdot e(v)[w_{x^1} - (w_{x^1})_{\bB^+_{s}}]  \|_{1, q; \RR^2_+} 
\leq\  & C_{p, \tau}L s^{1 - \frac{2}{p}} \cdot (1 + s^{-1}) \cdot \|w_{x^1} - (w_{x^1})_{\bB^+_{s}} \|_{q; \bB^+_{s}}\\
& + C_{\tau}\|\nabla v\|_{2q; \bB^+_{s}} \|w_{x^1} - (w_{x^1})_{\bB^+_{s}} \|_{2q; \bB^+_{s}}\\
& +C_{p, \tau} L s^{1 - \frac{2}{p}}\cdot \|\nabla^2 w_1 - \nabla^2 w_2\|_{q; \bB^+_{s}}.
\end{split}
\]
Since $q \geq 2$, applying the Poincar\'e inequality to the first two terms on the right hand side, and paying attention to the powers of $s$ thus introduced, we obtain 
\[
\| \eta \cdot e(v)[w_{x^1} - (w_{x^1})_{\bB^+_{s}}]  \|_{1, q; \RR^2_+}
\leq \big( C_{p,\tau, q}L s^{1 - \frac{2}{p}}   +  C_{\tau ,q}\|\nabla v\|_{2q; \bB^+_{s}}\cdot s^{1 - \frac{1}{q}} \big) \|\nabla^2 w\|_{q; \RR^2_+}.
\]
Therefore $(L, B|_{\partial\RR^2_+}): \cX_q \to \cY_q$ is bounded, and we have also established~\eqref{eq:fixed-point-estimate-2}. The estimate~\eqref{eq:fixed-point-estimate-1} is a direct consequence of Lemma~\ref{lemm:coefficient-bounds}(d). 
\end{proof}
\begin{rmk}\label{rmk:quantitative-W22}
It follows from~\eqref{eq:W14-W22-estimate} and the relation between $v$ and $u$ that for all $x_0 \in \partial \bB$, whenever $r \in (0, r_{\bB})$ and $L > 0$ are such that~\eqref{eq:W1p-scale-invariant-bound} holds with $e_1$ replaced by $x_0$, there holds
\begin{equation}\label{eq:W22-estimate-for-u}
r^{\frac{1}{2}}\|\nabla u\|_{4; \bB_{\frac{3}{4}\sigma_3 r}(x_0) \cap \bB} + r\|\nabla^2 u\|_{2; \bB_{\frac{3}{4}\sigma_3 r}(x_0) \cap \bB} \leq C(p, H_1, \tau, L, D_{\ep, p}(u; \bB_{4\sigma_3 r}(x_0) \cap \bB)),
\end{equation}
where $\sigma_3$ is as determined in the course of the above proof and depends only on $p, L, \tau$ and $H_1$. This estimate is used in Section~\ref{sec:proof}, and is also the reason we assume an $L^4$-bound to start with in Proposition~\ref{prop:W14-W2q}.
\end{rmk}
\begin{rmk}\label{rmk:C0-regularity}
Essentially the same argument as in the proof of Proposition~\ref{prop:boundary-regularity} shows the following: Suppose $u$ lies in $W^{1, q}(\bB; \RR^3) \cap C^0(\overline{\bB}, \partial \bB; \RR^3, \Sigma)$ for some $q > 2$ and satisfies
\begin{equation}\label{eq:weak-EL}
\begin{split}
0 =\ & \int_{\bB} \bangle{\nabla u, \nabla \varphi} - \tau  X(u) \cdot (\varphi_{x^1} \times u_{x^2} + u_{x^1} \times \varphi_{x^2})\\
& - \tau \int_{\bB} (\nabla X)_u(\varphi) \cdot u_{x^1} \times u_{x^2} + \int_{\bB} h(u)\varphi \cdot u_{x^1} \times u_{x^2}, 
\end{split}
\end{equation}
for all test functions $\varphi \in W^{1, q}(\bB; \RR^3)$ such that $\varphi(x) \in T_{u(x)}\Sigma$ for all $x \in \partial \bB$. Then $u$ is smooth on $\overline{\bB}$. Going back to~\eqref{eq:weak-EL} and integrating by parts using the identity~\eqref{eq:volume-form-Lie}, we see that $u$ is a solution in the classical sense of
\begin{equation}\label{eq:ep-0-classical}
\left\{
\begin{array}{l}
\Delta u = f(u)u_{x^1} \times u_{x^2} \text{ in }\bB,\\
u_r + \tau X(u) \times u_{\theta} \perp T_u\Sigma \text{ on }\partial \bB.
\end{array}
\right.
\end{equation}
These observations are used several times in Section~\ref{sec:proof}.
\end{rmk}
\subsection{A Palais-Smale type condition}
Here, as in~\cite{Cheng22,cz-cmc} we prove that $E_{\ep, p, f, \tau}$ satisfies a version of the Palais-Smale condition where we assume that $D_{\ep, p}$, as opposed to the functional itself, is bounded along the sequence.
\begin{prop}\label{prop:Palais-Smale}
Suppose $(u_m)$ is a sequence in $\cM_p$ such that $D_{\ep, p}(u_m) \leq C_0$ for some $C_0$ independent of $m$ and that
\begin{equation}\label{eq:Palais-Smale}
\lim_{m \to \infty}\| (\delta E_{\ep, p, f, \tau})_{u_m} \| = 0.
\end{equation}
Then a subsequence of $(u_m)$ converges strongly in $W^{1, p}$ to a critical point $u$ of $E_{\ep, p, f, \tau}$ such that $D_{\ep, p}(u) \leq C_0$.
\end{prop}
\begin{proof}
By assumption $(u_m)$ is bounded in $W^{1, p}$. Thus, taking a subsequence if necessary, we may assume that $(u_m)$ converges weakly in $W^{1, p}$ and uniformly to some $u \in \cM_p$. In particular, there exists $r \in (0, r_{\bB}]$ such that for all $m$ sufficiently large and $x_0 \in \partial \bB$, 
\[
u_m(\bB_{2r}(x_0) \cap \bB) \cup u(\bB_{2r}(x_0) \cap \bB) \subset B_{\rho_{\tau}}(u(x_0)).
\]
Fixing $x_0 \in \partial \bB$, we may then introduce $\Phi, \Psi, g, F, ...$ as in Section~\ref{subsec:function-spaces} and define a new sequence $(v_m)$ by letting 
\[
\widehat{u}_m(x) = \Phi(u_m(x)) \text{ for }x \in \bB_{2r}(x_0) \cap \bB,
\]
and
\[
v_m(x) = \widehat{u}_m(F(rx)) \text{ for }x \in \bB^+. 
\]
Given $\psi \in W^{1, p}(\bB^+)$ such that $\psi^3 = 0$ on $\bT$ and $\psi = 0$ on $\partial \bB \cap \RR^2_+$, we define
\[
\widehat{\psi}(x)  =
\left\{
\begin{array}{l}
\psi(\frac{1}{r}F^{-1}(x)) \text{ if }x \in F(\bB_{r}^+),\\
0 \text{ on }\bB \setminus F(\bB^+_{r}).
\end{array}
\right.
\]
and let $\psi^{(m)} = (d\Psi)_{\widehat{u}_m}(\widehat{\psi})$. Then $\psi^{(m)} \in T_{u_m}\cM_p$. Moreover, by Lemma~\ref{lemm:first-variation-simplified}, or rather the proof of~\eqref{eq:E-L-fermi-2}, we have in the notation introduced in~\eqref{eq:E-L-coefficients} that
\[
(\delta E_{\ep, p, f, \tau})_{u_m}(\psi^{(m)}) = \int_{\bB^+} A_{i\alpha}(x, v_m, \nabla v_m) \partial_{\alpha}\psi^i + A_i(x, v_m, \nabla v_m) \psi^i.
\]
Next, take a cutoff function $\zeta \in C^{\infty}_{c}(\bB)$ so that $\zeta = 1$ on $\bB_{\frac{1}{2}}$, and define
\[
\psi_{mn} = (v_m - v_n)\zeta^2.
\]
With the help of Lemma~\ref{lemm:coefficient-bounds} and Young's inequality, we obtain positive constants $c_\tau$ and $C_{\tau, H_1}$ so that
\begin{equation}\label{eq:P-S-big}
\begin{split}
&(\delta E_{\ep, p, f, \tau})_{u_m}(\psi_{mn}^{(m)}) - (\delta E_{\ep, p, f, \tau})_{u_n}(\psi_{mn}^{(n)})\\
\geq\ & c_\tau\int_{\bB^+} [1 + \bep^{p-2}(r^2 + |\nabla v_n|^2 + |\nabla v_m|^2)^{\frac{p}{2} - 1}] \zeta^2 |\nabla v_n - \nabla v_m|^2 \\
&- C_{\tau, H_1} \int_{\bB^+} [1 + \bep^{p-2}(r^2 + |\nabla v_n|^2 + |\nabla v_m|^2)^{\frac{p}{2} - 1}] |\nabla \zeta|^2 |v_n - v_m|^2\\
& - C_{\tau, H_1} \int_{\bB^+} [1 + \bep^{p-2}(r^2 + |\nabla v_n|^2 + |\nabla v_m|^2)^{\frac{p}{2} - 1}] \zeta^2 (|\nabla v_m|^2 + |\nabla v_n|^2) |v_n - v_m|^2.
\end{split}
\end{equation}
Since the $W^{1, p}$-norms of $\psi_{mn}^{(m)}$ and $\psi_{mn}^{(n)}$ are bounded uniformly in $m, n$, we deduce from the assumption~\eqref{eq:Palais-Smale} that the left hand side of~\eqref{eq:P-S-big} tends to zero as $m, n \to \infty$. Using the weak-$W^{1,p}$ and uniform convergence of $(u_m)$, we see that so do the second and third terms on the right hand side of~\eqref{eq:P-S-big}. Hence 
\[
\lim_{m, n \to \infty} \|\nabla v_m - \nabla v_n\|_{p; \bB^+_{\frac{1}{2}}} = 0,
\]
which shows that each boundary point $x_0 \in \partial \bB$ has a neighborhood on which $(u_m)$ converges strongly in $W^{1, p}$. A similar argument gives strong $W^{1, p}$-convergence on any compact subset of $\bB$, and we conclude strong $W^{1, p}$ convergence on all of $\overline{\bB}$. The remaining conclusions are obvious.
\end{proof}
\subsection{The second variation}\label{subsec:second-variation}
Introducing, for $(y, \xi) \in \RR^3 \times \RR^{3 \times 2}$, the functions
\[
\begin{split}
C_{i\alpha}(y, \xi) =\ & [1 + \ep^{p-2}(1 + 2F_\tau(y, \xi))^{\frac{p}{2} - 1}] ( \xi^i_\alpha - \tau \cdot \mu_{ijk}X^k(y) \xi^j_{\beta}\epsilon_{\alpha\beta}  )\\
C_i(z, \xi) = \ & h(y) \mu_{ijk} \xi^{j}_{\alpha} \xi^{k}_{\beta} \frac{\epsilon_{\alpha\beta}}{2} - \tau\cdot [1 + \ep^{p-2}(1 + 2F_\tau(y, \xi))^{\frac{p}{2} - 1}] \mu_{ljk}\pa{X^l}{y^i}(y) \xi^j_{\alpha}\xi^k_{\beta} \frac{\epsilon_{\alpha\beta}}{2},
\end{split}
\]
we may express the first variation formula~\eqref{eq:first-variation} as 
\[
(\delta E_{\ep, p, f, \tau})_u(\varphi) = \int_{\bB} C_{i\alpha}(u, \nabla u) \partial_\alpha\varphi^i + C_i(u, \nabla u) \varphi^i, \text{ for all }\varphi \in T_u \cM_p.
\]
Suppose now that $u$ is a smooth critical point of $E_{\ep, p, f, \tau}$. Then a direct computation shows that
\begin{equation}\label{eq:second-variation}
\begin{split}
(\delta^2 E_{\ep, p, f, \tau})_u(\varphi, \varphi) =\ & \int_{\bB}\big( \pa{C_{i\alpha}}{\xi^j_\beta}\partial_\beta\varphi^j + \pa{C_{i\alpha}}{z^m}\varphi^m \big)\partial_\alpha\varphi^i + \big( \pa{C_i}{\xi^j_\beta}\partial_\beta \varphi^j + \pa{C_i}{z^m}\varphi^m \big) \varphi^i\\
& + \int_{\partial\bB} [1 + \ep^{p-2}(1 + 2F_\tau)^{\frac{p}{2} - 1}] (u_r + \tau X(u) \times u_\theta) \cdot A^\Sigma_u(\varphi, \varphi),
\end{split}
\end{equation}
for all $\varphi \in T_u\cM_p$, where $A^\Sigma$ is the second fundamental form of $\Sigma$, given by
\[
A^\Sigma_y(v, w) = (\nabla_v w \cdot X(y))X(y) \text{ for all }y \in \Sigma,\ v, w \in T_y \Sigma.
\]
As in~\cite{Cheng22}, a simple but important observation about the second variation of $E_{\ep, p, f, \tau}$ is that it extends to a Hilbert space. Specifically we have the following lemma.
\begin{lemm}\label{lemm:second-variation-extend}
Let $u \in \cM_p$ be a smooth critical point of $E_{\ep, p, f, \tau}$. Then the second variation $(\delta^2 E_{\ep, p, f, \tau})_u$ extends to a bounded, symmetric bilinear form on the Hilbert space 
\[
\cH_u = \{\varphi \in W^{1, 2}(\bB; \RR^3)\ |\ (\varphi|_{\partial \bB})(x) \in T_{u(x)}\Sigma \text{ for a.e. }x \in \partial \bB\}.
\]
Moreover, $H_u$ possesses an $L^2$-orthonormal basis consisting of eigenfunctions of $(\delta^2 E_{\ep, p, f, \tau})_u$. Each eigenfunction is smooth on $\overline{\bB}$, and the corresponding eigenvalues form a sequence of real numbers tending to infinity.
\end{lemm}
\begin{proof}
Since $u$ is smooth, the coefficients in~\eqref{eq:second-variation} are all bounded on $\bB$, and hence 
\[
\big|(\delta^2 E_{\ep, p, f, \tau})_u (\varphi, \varphi)\big| \leq C\int_{\bB} |\nabla \varphi|^2 + |\varphi|^2  + C\int_{\partial \bB} |\varphi|^2  \leq C \|\varphi\|_{1, 2; \bB}^2,
\]
say for all $\varphi \in C^1(\overline{\bB}; \RR^3) \cap H_u$, where in getting the second inequality we used the fact that
\begin{equation}\label{eq:boundary-Sobolev}
\begin{split}
\int_{\partial \bB} |\varphi|^2 =\ & \int_{\partial \bB} |\varphi|^2 x \cdot \nu =  \int_{\bB} \Div(|\varphi|^2 x) \leq C\int_{\bB} |\varphi|^2 + |\varphi||\nabla \varphi|.
\end{split}
\end{equation}
As $C^1(\overline{\bB}; \RR^3) \cap H_u$ is dense in $H_u$, we infer that $(\delta^2 E_{\ep, p, f, \tau})_u$ extends to a bounded, symmetric bilinear form on $\cH_u$, and the formula~\eqref{eq:second-variation} continues to hold. To get the remaining conclusions, note that since
\begin{equation}
\label{eq:ellipticity-lead}
\pa{C_{i\alpha}}{\xi^{j}_{\beta}}(y, \xi)\eta^i_{\alpha}\eta^j_{\beta} \geq (1 - |\tau|)|\eta|^2 \text{ for all }y \in \RR^3 \text{ and }\xi, \eta \in \RR^{3 \times 2},
\end{equation}
we get the following estimate from~\eqref{eq:second-variation}:
\[
\begin{split}
(\delta^2 E_{\ep, p, f, \tau})_u(\varphi, \varphi) \geq\ & (1 - |\tau|)\int_{\bB} |\nabla \varphi|^2 - C\int_{\bB} |\nabla \varphi| |\varphi| + |\varphi|^2 - C\int_{\partial \bB}|\varphi|^2.
\end{split}
\]
Combining this with~\eqref{eq:boundary-Sobolev} and Young's inequality, we obtain a G$\mathring{a}$rding type inequality for $(\delta^2 E_{\ep, p, f, \tau})_u$ on $H_u$, and hence the desired basis of $\cH_u$ consisting of eigenfunctions of $(\delta^2 E_{\ep, p, f, \tau})_u$. The smoothness of the eigenfunctions then follows from the smoothness of the coefficients in~\eqref{eq:second-variation} and the ellipticity of the leading coefficients $\pa{C_{i\alpha}}{\xi^j_{\beta}}$ as expressed by~\eqref{eq:ellipticity-lead}.
\end{proof}
\begin{defi}\label{defi:index}
The number of negative eigenvalues of $(\delta^2 E_{\ep, p, f, \tau})_u$ counted with multiplicity is denoted $\Ind_{\ep, p, f, \tau}(u)$. Below we let $V_u^{-}$ be the sum of the negative eigenspaces of $(\delta^2 E_{\ep, p, f, \tau})_u$ and $V_u^{\perp}$ the $L^2$-orthogonal complement of $V_u^{-}$ in $T_u\cM_p$. Then both $V_u^{-}$ and $V_u^{\perp}$ are closed subspaces of $T_u\cM_p$, with $V_u^-$ being finite dimensional, and we have
\[
T_u\cM_p = V_u^{-} \oplus V_u^{\perp}.
\]
For each $\xi \in T_u\cM_p$, we write $\xi = \xi^{-} + \xi^{\perp}$ according to this decomposition. Since $V_{u}^{-}$ and $V_u^{\perp}$ are both closed subspaces, the norms $\xi \mapsto \|\xi^{-}\|_{1, p} + \|\xi^{\perp}\|_{1, p}$ and $\xi \mapsto \|\xi\|_{1, p}$ are equivalent on $T_u\cM_p$.
\end{defi}

The next lemma is the counterpart of~\cite[Propsition 4.3]{Cheng22}. The proof, which is not reproduced in full here to avoid repetition, involves just basic calculus, but the resulting estimates form the cornerstone for the deformation procedure in the proof of Proposition~\ref{prop:mountain-pass-upgraded} below. We continue to assume that $u$ is a smooth critical point of $E_{\ep, p, f, \tau}$, and in addition choose a simply-connected neighborhood $\cA$ of $u$ along with a local reduction $E^{\cA}_{\ep, p, f, \tau}$. Letting
\[
\cB_s = \{\varphi \in T_u \cM_p \ |\ \|\varphi\|_{1, p} < s\},
\]
for $s$ sufficiently small we may introduce a coordinate chart
\[
\Theta: (\cB_s, \{0\}) \to (\cA, \{u\}) 
\]
centered at $u$ such that $(d\Theta)_0$ is the identity map of $T_u\cM_p$. Letting $\widetilde{E} = E^{\cA}_{\ep, p, f, \tau}\circ \Theta$, then $\widetilde{E}$ is twice continuously differentiable on $\cB_s$. Moreover, since $u$ is a critical point, we have $(\delta \widetilde{E})_0 = 0$ and $(\delta^2 \widetilde{E})_0 = (\delta^2 E_{\ep, p, f, \tau})_u$.
\begin{lemm}\label{lemm:morse-nbd}
There exist $r \in (0, \frac{s}{3})$, $a > 0$ and $\theta \in (0, 1)$ depending on $u$ such that the following hold regardless of the choice of local reduction $E^{\cA}_{\ep, p, f, \tau}$.
\begin{enumerate}
\item[(a)] For all $\xi \in \cB_r$ such that $\|\xi^\perp\|_{1, p} \leq \theta \cdot \|\xi^-\|_{1, p}$ we have
\[
\widetilde{E}(0) - \widetilde{E}(\xi) \geq a \|\xi^-\|_{1, p}^2.
\]
\item[(b)] For all $\varphi \in \cB_r$ and $\xi \in V_u^-$ such that $\|\xi\|_{1, p} = 1$ and $(\delta \widetilde{E})_{\varphi}(\xi) \leq 0$, we have
\[
\widetilde{E}(\varphi) - \widetilde{E}(\varphi + t\xi) \geq at^2, \text{ for all }t \in [0, r].
\]
\end{enumerate}
\end{lemm}
\begin{proof}
The proof is essentially the same as that of~\cite[Proposition 4.3]{Cheng22}. Thus we only explain the beginning of the argument and refer the reader to~\cite{Cheng22} for further details. Since both (a) and (b) concern differences of the values of $\widetilde{E}$ at two points and since $\cB_s$ is a connected set, by Lemma~\ref{lemm:volume-properties} we see that the validity of the asserted estimates is independent of the choice of local reduction. Next, since $(\delta^2\widetilde{E})_0$ is a bounded bilinear form on $T_u\cM_p$, there exists $C_u > 0$ such that 
\[
(\delta^2\widetilde{E})_0(\xi, \xi) \leq C_u \|\xi\|_{1, p}^2, \text{ for all } \xi \in T_u\cM_p.
\]
On the other hand, since $V_u^{-}$ is finite dimensional, the restrictions of the $L^2$ and $W^{1, p}$-norms are equivalent, and hence there exists $c_u > 0$ such that
\[
(\delta^2\widetilde{E})_0(\xi, \xi) \leq -c_u \|\xi\|_{1, p}^2, \text{ for all } \xi  \in V_u^{-}.
\]
Also, we have $(\delta^2 \widetilde{E})_0(V_u^{-}, V_u^{\perp}) = 0$. Finally, since $\widetilde{E}$ is $C^2$, there exists $r \in (0, \frac{s}{3})$ such that 
\[
|(\delta^2\widetilde{E})_\psi(\xi, \xi) - (\delta^2\widetilde{E})_0(\xi, \xi)| \leq \frac{c_u}{16}\|\xi\|_{1, p}^2, \text{ for all }\xi \in T_u\cM_p \text{ and }\psi \in \cB_{2r}.
\]
We can then use the first-order Taylor expansion of $\widetilde{E}$ with the remainder in integral form, as in the proof of~\cite[Proposition 4.3]{Cheng22}, to get both conclusions of the lemma. 
\end{proof}

As in Section~\ref{subsec:smoothness}, we finish with a discussion of the case $\ep = 0$. Given a smooth solution $u: (\overline{\bB}, \partial \bB) \to (\RR^3, \Sigma)$ of~\eqref{eq:ep-0-classical}, define $(\delta^2 E_{f, \tau})_u$ to be the bilinear form obtained by setting $\ep = 0$ in~\eqref{eq:second-variation}. Also, we introduce the space
\[
\cV = \{\varphi \in C^\infty(\overline{\bB}; \RR^3) \ |\ \varphi(x) \in T_{u(x)}\Sigma, \text{ for all }x \in \partial \bB\},
\]
and use $\Ind^E_{f, \tau}(u)$ to denote the index of $(\delta^2 E_{f, \tau})_u$ on $\cV$. That is, $\Ind^E_{f, \tau}(u)$ is the supremum of $\dim \cL$ over all finite-dimensional subspaces $\cL$ of $\cV$ on which $(\delta^2 E_{f, \tau})_u$ is negative definite.
\begin{lemm}\label{lemm:second-variation-0}
For all $\varphi \in \cV$, we have the following expression for $(\delta^2 E_{f, \tau})_u(\varphi, \varphi)$:
\begin{equation}\label{eq:second-variation-0}
\begin{split}
(\delta^2 E_{f, \tau})_u(\varphi, \varphi) =\ & \int_{\bB}|\nabla \varphi|^2 + (\nabla f)_u(\varphi)\ \varphi \cdot u_{x^1} \times u_{x^2} + f(u)\ \varphi \cdot (\varphi_{x^1} \times u_{x^2} + u_{x^1} \times \varphi_{x^2})\\
& + \int_{\partial \bB} \big(u_r + \tau X(u) \times u_\theta\big) \cdot  A^{\Sigma}_u(\varphi, \varphi) + \tau \int_{\partial \bB} \varphi \cdot X(u) \times \varphi_\theta.
\end{split}
\end{equation}
\end{lemm}
\begin{proof}
Denoting $f + \tau\Div X$ by $h$ and setting $\ep = 0$ in~\eqref{eq:second-variation} gives
\begin{equation}\label{eq:second-variation-0-1}
\begin{split}
(\delta^2 E_{f, \tau})_u(\varphi, \varphi) =\ & \int_{\bB} |\nabla \varphi|^2 - 2\tau  X(u) \cdot \varphi_{x^1} \times \varphi_{x^2} -  \tau \int_{\bB} (\nabla X)_u(\varphi) \cdot (\varphi_{x^1} \times u_{x^2} + u_{x^1} \times \varphi_{x^2})  \\
& - \tau \int_{\bB} (\nabla X)_u(\varphi) \cdot (\varphi_{x^1} \times u_{x^2} + u_{x^1} \times \varphi_{x^2}) + (\nabla^2 X)_u(\varphi, \varphi) \cdot u_{x^1} \times u_{x^2}\\
& + \int_{\bB} (\nabla h)_u(\varphi)\  \varphi \cdot u_{x^1} \times u_{x^2} + h(u)\ \varphi \cdot (\varphi_{x^1} \times u_{x^2} + u_{x^1} \times \varphi_{x^2})\\
& + \int_{\partial \bB} \big(u_r + \tau X(u) \times u_\theta\big) \cdot  A^{\Sigma}_u(\varphi, \varphi).
\end{split}
\end{equation}
(We deliberately repeat $\int_{\bB} (\nabla X)_u(\varphi) \cdot (\varphi_{x^1} \times u_{x^2} + u_{x^1} \times \varphi_{x^2})$ instead of combining the two occurrences into one term.) Integrating by parts in the second term on the right hand side, we get
\begin{equation}\label{eq:nabla-X-by-parts}
\begin{split}
&-\tau\int_{\bB}(\nabla X)_u(\varphi) \cdot (\varphi_{x^1} \times u_{x^2} + u_{x^1} \times \varphi_{x^2})\\
 =\ & \tau \int_{\partial \bB} \varphi \cdot (\nabla X)_u(\varphi) \times u_{\theta} - \tau \int_{\bB}\varphi \cdot \big[ (\nabla X)_u(\varphi_{x^1}) + (\nabla^2X)_u(u_{x^1}, \varphi) \big] \times u_{x^2}\\
&-\tau \int_{\bB}\varphi \cdot u_{x^1} \times \big[ (\nabla X)_u(\varphi_{x^2}) + (\nabla^2X)_u(u_{x^2}, \varphi) \big].
\end{split}
\end{equation}
The boundary term on the right hand side vanishes because when $x \in \partial \bB$, the vectors $\varphi(x), (\nabla X)_{u(x)}(\varphi(x))$ and $u_\theta(x)$ all lie in $T_{u(x)}\Sigma$. Combining~\eqref{eq:nabla-X-by-parts} with~\eqref{eq:volume-form-Lie-2} gives
\[
\begin{split}
&-  \tau \int_{\bB} (\nabla X)_u(\varphi) \cdot (\varphi_{x^1} \times u_{x^2} + u_{x^1} \times \varphi_{x^2})  \\
& - \tau \int_{\bB} (\nabla X)_u(\varphi) \cdot (\varphi_{x^1} \times u_{x^2} + u_{x^1} \times \varphi_{x^2}) + (\nabla^2 X)_u(\varphi, \varphi) \cdot u_{x^1} \times u_{x^2}\\
=\ & \tau \int_{\bB} \varphi \cdot \big[  (\nabla X)_u(u_{x^1}) \times \varphi_{x^2} + \varphi_{x^1} \times (\nabla X)_u(u_{x^2})  \big]\\
& - \tau \int_{\bB}[\nabla (\Div X)]_u(\varphi)\ \varphi \cdot u_{x^1} \times u_{x^2} +  (\Div X)_u\ \varphi \cdot (\varphi_{x^1} \times u_{x^2} + u_{x^1} \times \varphi_{x^2}).
\end{split}
\]
Returning to~\eqref{eq:second-variation-0-1} and recalling that $h = f + \tau \Div X$, we deduce that 
\begin{equation}\label{eq:second-variation-0-2}
\begin{split}
(\delta^2 E_{f, \tau})_u(\varphi, \varphi) =\ & \int_{\bB} |\nabla \varphi|^2 - 2\tau X(u) \cdot \varphi_{x^1} \times \varphi_{x^2}\\
& + \tau \int_{\bB} \varphi \cdot \big[  (\nabla X)_u(u_{x^1}) \times \varphi_{x^2} + \varphi_{x^1} \times (\nabla X)_u(u_{x^2})  \big]\\
& + \int_{\bB} (\nabla f)_u(\varphi)\  \varphi \cdot u_{x^1} \times u_{x^2} + f(u)\ \varphi \cdot (\varphi_{x^1} \times u_{x^2} + u_{x^1} \times \varphi_{x^2})\\
& + \int_{\partial \bB} \big(u_r + \tau X(u) \times u_\theta\big) \cdot  A^{\Sigma}_u(\varphi, \varphi).
\end{split}
\end{equation}
To finish, we integrate by parts to get
\[
\begin{split}
-2\tau \int_{\bB}X(u) \cdot \varphi_{x^1} \times \varphi_{x^2} =\ & -\tau\int_{\bB} X(u)\cdot (\varphi_{x^1} \times \varphi_{x^2} + \varphi_{x^1} \times \varphi_{x^2})\\
=\ & \tau\int_{\partial \bB} \varphi \cdot X(u) \times \varphi_\theta\\
&- \tau \int_{\bB}\varphi \cdot \big[ (\nabla X)_u(u_{x^1}) \times \varphi_{x^2} + \varphi_{x^1} \times (\nabla X)_u(u_{x^2})  \big].
\end{split}
\]
Substituting this into~\eqref{eq:second-variation-0-2} gives~\eqref{eq:second-variation-0}.\end{proof}
\section{Main estimates}\label{sec:PDE}
\subsection{Bounds on critical points of the perturbed functional}\label{subsec:a-priori}
In this section we prove several estimates on critical points of $E_{\ep, p, f, \tau}$ by modifying the arguments in~\cite{Fraser2000}, replacing estimates for the Neumann problem by those collected in Appendix~\ref{sec:linear} for oblique derivative problems. At the end of this section we recall two consequences of the maximum principle from~\cite{Cheng22}, both still applicable in the present context. We continue to assume that $H_1$ is an upper bound for $|f|, |\nabla f|, |\nabla X|$ and $|\nabla^2X|$ on $\RR^3$.
\begin{prop}\label{prop:W14-W2q}
Let $p_1$ be as in Proposition~\ref{prop:boundary-regularity}. Given $q > 2$ and $L  > 0$, there exist $p_2 \in (2, p_1]$ (depending only on $q, \tau$) and $\theta_1 \in (0, \frac{1}{4})$ (depending only on $q,\tau, L$) with the following properties: Suppose $u \in \cM_p$ is a critical point of $E_{\ep, p, f, \tau}$ with $p \in (2, p_2]$, and that 
\begin{equation}\label{eq:W14-bound}
r^2\int_{\bB_r(x_0) \cap \bB} |\nabla u|^4 \leq L^4 \text{ for some }x_0 \in \partial\bB \text{ and }r \in (0, r_{\bB}].
\end{equation}
Then for all $\sigma < \rho \leq \theta_1$, there holds
\begin{equation}\label{eq:W2q-estimate}
r^{2 - \frac{2}{q}} \|\nabla^2 u\|_{q; \bB_{\frac{3}{4}\sigma r}(x_0) \cap \bB} \leq C(\rho, \rho - \sigma, q, H_1, L, \tau)\cdot r^{\frac{1}{2}} \|\nabla u\|_{4; \bB_{\frac{5}{4}\rho r}(x_0) \cap \bB}.
\end{equation}
\end{prop}
\begin{proof}
Without loss of generality we assume $x_0 = e_1$. By the same argument leading up to the choice of $\sigma_1$ in~\eqref{eq:sigma-1-definition} (with $4$ in place of $p$), we get $\theta_0 = \theta_0(L, \tau) \in (0, \frac{1}{4})$ such that
\[
\osc_{\bB_{2\theta_0 r}(e_1) \cap \bB} u < \rho_{\tau},
\]
which allows us to introduce $\Phi,\Psi, g, P, F ...$ as in Section~\ref{subsec:function-spaces}. Again letting
\[
\widehat{u}(x) = \Phi(u(x)) \text{ for }x \in \bB_{2\theta_0 r}(e_1) \cap \bB,
\]
\[
v(x) = \widehat{u}(F(rx)), \text{ for }x \in \bB^+_{\theta_0},
\]
the same reasoning leading to~\eqref{eq:W1p-bound-for-v} gives $\int_{\bB^+_{\theta_0}} |\nabla v|^4 \leq CL^4$, so that 
\begin{equation}\label{eq:v-osc-bound-2}
\osc_{\bB^+_{\theta}} v \leq CL \theta^{\frac{1}{2}}, \text{ for all }\theta \leq \theta_0.
\end{equation}
Moreover, $v$ satisfies~\eqref{eq:E-L-fermi-non-div} in $\bB^+_{\theta_0}$ and~\eqref{eq:E-L-fermi-bc} on $\bT_{\theta_0}$. Given $\sigma < \rho \leq \theta_0$, we define
\[
\rho' = \frac{\rho + \sigma}{2}
\]
and let $\zeta$ be a cutoff function so that $\zeta = 1$ on $\bB_{\sigma}$ and $\zeta = 0$ outside $\bB_{\rho'}$, and that 
\[
|\nabla^i \zeta| \leq C(\rho - \sigma)^{-i},
\]
with $C$ depending only on $i$. Letting $a^i = \fint_{\bB^+_{\rho'}} v^i $ for $i = 1, 2$, and taking $a^3 = 0$, we see from~\eqref{eq:E-L-fermi-non-div} that
\begin{equation}\label{eq:PDE-cutoff}
\begin{split}
&(\Delta-1) (\zeta (v^i - a^i)) + E_{\alpha\beta}^{ij}(x, v, \nabla v) \cdot  \partial_{\alpha\beta}(\zeta (v^j - a^j))\\
=\ & - \zeta(v^i - a^i) + \zeta b^i(x, v, \nabla v)  + 2\nabla \zeta \cdot \nabla v^i + (v^i - a^i) \cdot \Delta \zeta\\
& + E_{\alpha\beta}^{ij}(x, v, \nabla v) \cdot \big( \partial_\alpha\zeta \partial_\beta v^j + \partial_\beta\zeta \partial_\alpha v^j + (v^j - a^j) \partial_{\alpha\beta}\zeta \big) =: f_i.
\end{split}
\end{equation}
Also, again introducing
\[
e_{ij}(z) = \tau \big[J^i_j(z) - J^i_j(v(0))\big],\ \ \ (B_0 w)^i = w^i_{x^2} + \tau J^i_j(v(0))w^j_{x^1},
\] 
then on $\partial \RR^2_+$ we have 
\[
\zeta (v^3 - a^3) = \zeta v^3 = 0,
\]
and that, for $i = 1, 2$,
\begin{equation}\label{eq:bc-cutoff}
\begin{split}
&(B_0(\zeta (v - a)))^i + e_{ij}(v)  (\zeta (v^j - a^j))_{x^1}\\
= \ &  (v^i - a^i)\zeta_{x^2} + \tau J^i_{j}(v) (v^j - a^j)\zeta_{x^1} =: g_i.
\end{split}
\end{equation}
Recalling Lemma~\ref{lemm:coefficient-bounds} and using the Poincar\'e inequality, the latter applicable by our choice of $a$ and the fact that $v^3 = 0$ on $\bT_{\rho}$, we have
\begin{equation}\label{eq:rhs-bound-1}
\begin{split}
\| f \|_{q; \RR^2_+} \leq C_{H_1, \tau, q}\big( (\rho - \sigma)^{-2} \rho \|\nabla v\|_{q; \bB^+_{\rho'}} + \|\nabla v\|^2_{2q; \bB^+_{\rho'}}\big),
\end{split}
\end{equation}
\begin{equation}\label{eq:rhs-bound-2}
\begin{split}
\| g \|_{1, q; \RR^2_+} \leq C_{ \tau, q} \big( (\rho - \sigma)^{-2} \rho \|\nabla v\|_{q; \bB^+_{\rho'}} + (\rho - \sigma)^{-1} \rho \|\nabla v\|^2_{2q; \bB^+_{\rho'}}\big).
\end{split}
\end{equation}
By~\eqref{eq:PDE-cutoff} and the boundary conditions on $\zeta(v - a)$, as well as the bounds~\eqref{eq:rhs-bound-1} and~\eqref{eq:rhs-bound-2}, we deduce from Lemma~\ref{lemm:inhomogeneous-estimate} and standard $W^{2, p}$-estimates for the Dirichlet problem that
\begin{equation}\label{eq:W2q-estimate-cutoff-1}
\begin{split}
\|\nabla (\zeta (v - a))\|_{1,q; \RR^2_+} \leq\ & A_{q, \tau} \|f - E(x, v, \nabla v)\nabla^2(\zeta(v - a))\|_{q; \RR^2_+}\\
& + A_{q, \tau}\|g - e(v)(\zeta (v - a))_{x^1}\|_{1, q; \RR^2_+} \\
\leq\ & C_{H_1, \tau, q}\big( (\rho - \sigma)^{-2}\rho \|\nabla v\|_{q; \bB^+_{\rho'}} + (\rho - \sigma)^{-1}\rho\|\nabla v\|_{2q;\bB^+_{\rho'}}^2\big) \\
& + A_{q, \tau} \big(\| E(x, v, \nabla v) \nabla^2 (\zeta(v - a))\|_{q; \RR^2} + \| e(v) (\zeta(v - a))_{x^1}\|_{1, q; \RR^2_+} \big).
\end{split}
\end{equation}
Since $\zeta$ is supported in $\bB_{\rho}$, using the estimate~\eqref{eq:v-osc-bound-2} we have
\[
 \| e(v) \cdot (\zeta(v - a))_x\|_{1, q; \RR^2_+} \leq C_{\tau} L \rho^{\frac{1}{2}} \cdot \|\nabla (\zeta (v - a))\|_{1, q; \RR^2_+} + C_{\tau, q}(\rho - \sigma)^{-1}\rho\cdot \|\nabla v\|_{2q; \bB_{\rho'}^+}^2.
\]
Also, Lemma~\ref{lemm:coefficient-bounds}(d) gives
\[
\| E(x, v, \nabla v) \nabla^2 (\zeta(v - a))\|_{q; \RR^2_+} \leq C_{\tau}(p-2)\cdot \|\nabla^2 (\zeta (v - a))\|_{q; \RR^2_+},
\]
Combining the previous two estimates with~\eqref{eq:W2q-estimate-cutoff-1} yields
\begin{equation}\label{eq:W2q-estimate-cutoff-2}
\begin{split}
\|\nabla (\zeta (v - a))\|_{1, q; \RR^2_+} \leq\ & C_{q, \tau}(L \rho^{\frac{1}{2}} + p-2) \cdot \|\nabla (\zeta (v - a))\|_{1,q; \RR^2_+}  \\
& +C_{H_1, \tau, q}\big( (\rho - \sigma)^{-2}\rho \|\nabla v\|_{q; \bB^+_{\rho'}} + (\rho - \sigma)^{-1}\rho\|\nabla v\|_{2q;\bB^+_{\rho'}}^2\big).
\end{split}
\end{equation}
Choosing $p_2 \in (2, p_1]$ depending on $q, \tau$ and $\theta_1 \in (0, \theta_0]$ depending on $q, \tau, L$ so that
\[
(C_{2, \tau} + C_{q, \tau})(L \theta_1^{\frac{1}{2}} + p_2-2) < \frac{1}{2},
\]
we get from~\eqref{eq:W2q-estimate-cutoff-2} that
\begin{equation}\label{eq:W2q-estimate-cutoff-3}
\begin{split}
\|\nabla^2 v\|_{q; \bB^+_{\sigma}} \leq C_{H_1, \tau, q}\big( (\rho - \sigma)^{-2}\rho \|\nabla v\|_{q; \bB^+_{\rho'}} + (\rho - \sigma)^{-1}\rho\|\nabla v\|_{2q;\bB^+_{\rho'}}^2\big).
\end{split}
\end{equation}
Repeating the above argument with $2$ in place of $q$ and choosing $\zeta$ instead to equal $1$ on $\bB_{\rho'}$ and vanish outside of $\bB_{\rho}$, we get by the choice of $p_2$ that 
\begin{equation}\label{eq:W22-estimate-cutoff-4}
\begin{split}
\|\nabla^2 v\|_{2; \bB^+_{\rho'}} \leq C_{H_1, \tau}\big( (\rho - \sigma)^{-2}\rho \|\nabla v\|_{2; \bB^+_{\rho}} + (\rho - \sigma)^{-1}\rho\|\nabla v\|_{4;\bB^+_{\rho}}^2\big).
\end{split}
\end{equation}
The two above estimates can be stringed together via the following Sobolev inequality:
\[
\|\nabla v\|_{q; \bB_{\rho'}} + \|\nabla v\|_{2q; \bB_{\rho'}} \leq C_{q, \rho}( \|\nabla v\|_{2; \bB_{\rho'}} + \|\nabla^2 v\|_{2; \bB_{\rho'}}),
\]
and consequently we get
\[
\|\nabla v\|_{2q; \bB^+_{\sigma}}  + \|\nabla^2 v\|_{q; \bB^+_{\sigma}} \leq C(\rho, \rho - \sigma, q, H_1, \tau, L)\|\nabla v\|_{4; \bB_{\rho}^+},
\]
from which we deduce~\eqref{eq:W2q-estimate} upon recalling that $v(x) = (\Phi \circ u \circ F)(rx)$ and using~\eqref{eq:F-radius-comparable}.
\end{proof}
Equipped with Proposition~\ref{prop:W14-W2q}, we may follow the argument in~\cite{Fraser2000} to obtain the following result which yields estimates that are uniform in $\ep$ where energy does not concentrate. We sketch the proof for the reader's convenience.
\begin{prop}\label{prop:eta-regularity}
Let $p_2$ be given by Proposition~\ref{prop:W14-W2q} with $q = 4$. There exists $\eta_0 \in (0, \frac{1}{16})$ depending only on $\tau, H_1$ such that if $u \in \cM_p$ is a critical point of $E_{\ep, p, f, \tau}$ with $p \in (2, p_2]$ and
\begin{equation}\label{eq:small-energy-eta}
\int_{\bB_r(x_0) \cap \bB} |\nabla u|^2 < \eta_0, \text{ for some }x_0 \in \partial \bB,\ r \in (0, r_{\bB}],
\end{equation}
then 
\begin{equation}\label{eq:gradient-pointwise-bound}
r|\nabla u| \leq 8 \text{ on }\bB_{\frac{r}{2}}(x_0) \cap \bB.
\end{equation}
Moreover, with $\theta_1$ as given by Proposition~\ref{prop:W14-W2q} with $L = 1$ and $q = 4$, we have
\begin{equation}\label{eq:gradient-a-priori}
r^{\frac{1}{2}} \| \nabla u \|_{4; \bB_{\frac{1}{8}\theta_1 r}(x_0) \cap \bB} + r^{\frac{3}{2}} \| \nabla^2 u \|_{4; \bB_{\frac{1}{8}\theta_1 r}(x_0) \cap \bB} \leq C_{H_1, \tau} \|\nabla u\|_{2; \bB_{\frac{1}{2}\theta_1 r}(x_0) \cap \bB}.
\end{equation}
\end{prop}
\begin{proof}
Using Proposition~\ref{prop:W14-W2q} and following the same contradiction argument as in~\cite[Lemma 1.6]{Fraser2000} (see also~\cite[Proposition 3.6]{Cheng22}), we find $\eta_0 \in (0, \frac{1}{16})$ depending only on $\tau, H_1$ so that~\eqref{eq:small-energy-eta} implies
\[
\sup_{x \in \bB_{r}(x_0) \cap \bB}(r - |x - x_0|)|\nabla u(x)| \leq 4.
\]
From this we immediately deduce~\eqref{eq:gradient-pointwise-bound}, which together with~\eqref{eq:small-energy-eta} implies (recall that $\eta_0 < \frac{1}{16}$) that~\eqref{eq:W14-bound} in the hypothesis of Proposition~\ref{prop:W14-W2q} holds with $\frac{r}{2}$ in place of $r$ and with $L = 1$. Let $\theta_1$ be as given by Proposition~\ref{prop:W14-W2q} with $L = 1$ and $q = 4$. Then, in the notation of the proof of that proposition, but with $r$ replaced by $\frac{r}{2}$, we get from~\eqref{eq:gradient-pointwise-bound} and~\eqref{eq:lambda-bound},~\eqref{eq:isometry-deviation} some universal constant $C$ such that 
\begin{equation}\label{eq:gradient-bound-v}
|\nabla v(x)| \leq C, \text{ for all }x \in \bB^+_{\theta_1}.
\end{equation}
Next, repeating the derivation of~\eqref{eq:W2q-estimate-cutoff-3} and~\eqref{eq:W22-estimate-cutoff-4} with $q$ replaced by $4$ and with $r$ replaced by $\frac{r}{2}$, but at the end using~\eqref{eq:gradient-bound-v} to absorb the terms involving $\|\nabla v\|_{2q; \bB^+_{\rho'}}^2$ and $\|\nabla v\|_{4; \bB^+_{\rho}}^2$ into the term preceding them, we obtain
\[
\|\nabla^2 v\|_{4; \bB^+_{\sigma}} \leq C_{H_1, \tau} (\rho - \sigma)^{-2}\rho \|\nabla v\|_{4; \bB^+_{\rho'}},
\]
\[
\|\nabla^2 v\|_{2; \bB^+_{\rho'}} \leq C_{H_1, \tau} (\rho - \sigma)^{-2}\rho \|\nabla v\|_{2; \bB^+_{\rho}}.
\]
Applying this with $(\sigma, \rho) = (\frac{\theta_1}{3}, \frac{4\theta_1}{5})$, we obtain
\[
\|\nabla v\|_{4; \bB^+_{\frac{\theta_1}{3}}} + \|\nabla^2 v\|_{4; \bB^+_{\frac{\theta_1}{3}}} \leq C_{H_1, \tau} \|\nabla v\|_{2; \bB^+_{\frac{4\theta_1}{5}}},
\]
which together with~\eqref{eq:gradient-bound-v} gives~\eqref{eq:gradient-a-priori}.
\end{proof}
\begin{rmk}\label{rmk:ep-0}
In Proposition~\ref{prop:W14-W2q} and Proposition~\ref{prop:eta-regularity} we may allow $\ep = 0$. That is, the conclusions hold if we assume that $u: (\overline{\bB}, \partial \bB) \to (\RR^3, \Sigma)$ is a smooth map satisfying~\eqref{eq:ep-0-classical} in Remark~\ref{rmk:C0-regularity}. In this case the map $v$ would satisfy the equation~\eqref{eq:E-L-fermi-non-div} with $\ep = 0$, while the boundary condition remains~\eqref{eq:E-L-fermi-bc}.
\end{rmk}
The next estimate gives an energy lower bound for non-constant critical points of $E_{\ep, p, f, \tau}$ which is uniform in $\ep$ and is used in Section~\ref{sec:proof} to help guarantee that the solutions to~\eqref{eq:H-cmc-bvp} we obtain are not constants.
\begin{prop}\label{prop:lower-bound}
There exists $p_3 \in (2, p_2]$ depending only on $\tau$ and $\eta_1 \in (0, \eta_0)$ depending only on $\tau$ and $H_1$ such that for all $p \in (2, p_3]$, if $u \in \cM_p$ is a critical point of $E_{\ep, p, f, \tau}$ with 
\begin{equation}\label{eq:small-energy-global}
\int_{\bB} |\nabla u|^2 < \eta_1^2,
\end{equation}
then $u$ is constant.
\end{prop}
\begin{proof}
Since $\eta_1 < \eta_0$, under the assumption~\eqref{eq:small-energy-global} we can apply Proposition~\ref{prop:eta-regularity} (and its interior version) on balls of size $r_{\bB}$ to get from~\eqref{eq:gradient-a-priori} and the Sobolev embedding $W^{1, 4} \subset C^{0, \frac{1}{2}}$ that
\begin{equation}\label{eq:W14-global}
\sup_{\bB}|\nabla u| \leq C_{H_1,\tau} \cdot \eta_1,
\end{equation}
of course with possibly a different $C_{H_1, \tau}$ than the one in~\eqref{eq:gradient-a-priori}. In particular, provided $\eta_1$ is sufficiently small, we have $u(\bB) \subset U$ for some $U \in \cF$, and on $U$ we may introduce a chart $\Phi: U \to V$ of the kind described in Section~\ref{subsec:function-spaces}. Following calculations similar to the proof of Lemma~\ref{lemm:first-variation-simplified}, we find that, with $\widehat{u} = \Phi \circ u$, there holds
\[
\Delta \widehat{u}^i + \widehat{E}_{\alpha\beta}^{ij}(\widehat{u}, \nabla\widehat{u}) \cdot \partial_{\alpha\beta}\widehat{u}^j = \widehat{b}_i(\widehat{u}, \nabla \widehat{u}) \text{ in }\bB,
\]
along with the boundary conditions 
\[
\begin{split}
&\widehat{u}^3\big|_{\partial \bB} = 0,\\
&\Big(\widehat{u}^i_r - \tau J^i_{j}(\widehat{u}(0)) \widehat{u}^j_{\theta} + \tau \big[ J^i_{j}(\widehat{u}(0))  - J^{i}_{j}(\widehat{u}) \big]\widehat{u}^j_{\theta} \Big) \Big|_{\partial \bB}  = 0, \text{ for }i = 1, 2,
\end{split}
\]
where, for $(z, \xi) \in V \times \RR^{3 \times 2}$, the coefficients satisfy
\begin{equation}\label{eq:coefficient-bounds-global}
|\widehat{E}_{\alpha\beta}^{ij}(z, \xi)| \leq C_{\tau}(p-2),\ \ |\widehat{b}_i(z, \xi)| \leq C_{H_1,\tau}|\xi|^2.
\end{equation}
By Lemma~\ref{lemm:estimate-on-B} and global $W^{2,p}$-estimates for the Dirichlet problem, along with~\eqref{eq:coefficient-bounds-global} and~\eqref{eq:W14-global}, we have
\begin{equation}\label{eq:2-43-global}
\begin{split}
\|\nabla \widehat{u}\|_{1, 2; \bB} \leq\ & C_{\tau}( \|\widehat{b}(\widehat{u}, \nabla\widehat{u}) - E(\widehat{u}, \nabla\widehat{u})\nabla^2\widehat{u}\|_{2; \bB} + \|[J(\widehat{u}(0)) 
 - J(\widehat{u})]\widehat{u}_\theta \|_{1, 2; \bB})\\
\leq\ & C_{H_1, \tau}\||\nabla\widehat{u}|^2\|_{2; \bB} + [C_{H_1, \tau}\cdot\eta_1 + C_{\tau}(p-2)] \|\nabla\widehat{u}\|_{1, 2; \bB}.
\end{split}
\end{equation}
For the $L^2$-norm of $|\nabla u|^2$, by~\eqref{eq:W14-global} we have
\[
\||\nabla\widehat{u}|^2\|_{2; \bB} \leq \|\nabla \widehat{u}\|_{\infty; \bB}\|\nabla \widehat{u}\|_{2; \bB} \leq C_{H_1, \tau}\eta_1 \|\nabla\widehat{u}\|_{2; \bB}.
\]
Hence we deduce from~\eqref{eq:2-43-global} that
\[
\|\nabla \widehat{u}\|_{1, 2; \bB} \leq [C_{H_1, \tau} \eta_1 + C_{\tau}(p - 2)] \cdot \|\nabla \widehat{u}\|_{1, 2; \bB}.
\]
The desired conclusion follows easily.
\end{proof}
To finish this section we recall from~\cite{Cheng22} two consequences of the maximum principle and the convexity of $\Sigma'$. For these we impose extra conditions on $X$ and $f$. Specifically, we define, for $t > 0$,
\[
\Omega'_{t} = \{ y \in \RR^3\ |\ \dist(y, \Omega') < t \}\ \ \text{ and }\ \  \Sigma'_t = \partial \Omega'_t.
\]
With $\underline{H'}$ as in~\eqref{eq:H-requirement} and with $H \in [0, \underline{H'})$, we choose $t_0$ such that
\begin{equation}\label{eq:t0-choice}
H_{\Sigma'_{t}}(y) \geq \frac{1}{2}\big( \underline{H'} + H \big), \text{ for all }t \in [0, t_0] \text{ and }y \in \Sigma'_t.
\end{equation}
We then require additionally that the supports of both $X$ and $f$ are contained in $\Omega'_{\frac{3t_0}{4}}$, and that $f$ takes values in $[0, H]$.
\begin{prop}\label{prop:maximum-principle}
Suppose $X$ and $f$ satisfy the above additional requirements.
\vskip 1mm
\begin{enumerate}
\item[(a)] If $u \in \cM_p$ is a smooth critical point of $E_{\ep, p, f, \tau}$, then $u(\overline{\bB}) \subset \overline{\Omega'_{t_0}}$.
\vskip 1mm
\item[(b)] Suppose $u:(\overline{\bB}, \partial \bB) \to (\overline{\Omega'_{t_0}}, \Sigma)$ is a smooth solution of
\begin{equation}\label{eq:cmc-bvp}
\left\{
\begin{array}{ll}
\Delta u = f(u) u_{x^1} \times u_{x^2},& \text{ in } \bB,\\
u_r + \tau X(u) \times u_\theta \perp T_u\Sigma, & \text{ on }\partial \bB.
\end{array}
\right.
\end{equation}
Then $u$ is weakly conformal, and $u(\overline{\bB})\subset \overline{\Omega'}$.
\end{enumerate}
\end{prop}
\begin{proof}
Below we abbreviate the distance function $\dist(\cdot, \Omega')$ as $d$. Note that it is a convex function since $\Sigma'$ is a convex surface.

For part (a), since $u$ is smooth and maps $\partial \bB$ into $\Sigma$, the set
\[
C_+ = \{x \in \overline{\bB}\ |\ d(u(x)) > t_0\},
\]
is an open subset of $\bB$ and
\[
d(u(x)) = t_0 \text{ on }\partial C_+,
\]
As both $f$ and $X$ vanish outside of $\Omega'_{\frac{3t_0}{4}}$ by assumption, we see by~\eqref{eq:first-variation} and the convexity of $d$ that
\[
\Div\big( [1 + \ep^{p-2}(1 + 2F_{\tau}(u, \nabla u))^{\frac{p}{2} - 1}]\nabla (d\circ u) \big) \geq 0 \text{ on }C_+.
\]
The maximum principle then forces $C_+$ to be empty. That is, $d(u(x)) \leq t_0$ for all $x \in \overline{\bB}$.

For part (b), that $u$ is weakly conformal follows from Lemma~\ref{lemm:weakly-conformal}. We can then follow the proof of~\cite[Proposition 3.10(b)]{Cheng22}. More precisely, with $a > 0$ to be determined, we define $F = e^{ad}$. By our choice of $t_0$, a direct computation shows that if $a$ is sufficiently large, then at points where $d \in (0, t_0]$ we have
\begin{equation}\label{eq:two-eigenvalue}
\lambda_1 + \lambda_2 \geq \frac{ae^{ad}}{2} \big(\underline{H'} + H\big),
\end{equation}
where $\lambda_1 \leq \lambda_2 \leq \lambda_3$ denote the eigenvalues of $\nabla^2 F$. Next we define
\[
C'_{+} := \{x \in \overline{\bB}\ |\ d(u(x)) > 0\}
\]
and note that this is an open subset of $\bB$, with $d(u) = 0$ on $\partial C_{+}'$. Moreover, by assumption we have $0 < d(u) \leq t_0$ on $C_{+}'$, so that the lower bound~\eqref{eq:two-eigenvalue} is in effect. Recalling that $u$ is weakly conformal and that $0 \leq f \leq H$ everywhere, we compute
\[
\begin{split}
\Delta(F(u)) =\ & (\nabla^2 F)_{u}(u_{x^1}, u_{x^1}) + (\nabla^2 F)_{u}(u_{x^2}, u_{x^2}) + (\nabla F)_{u} \cdot (\Delta u)\\
\geq \ & \frac{|\nabla u|^2}{2}(\lambda_{1} + \lambda_2) - ae^{ad(u)}H \frac{|\nabla u|^2}{2}\\
\geq\ & \frac{ae^{ad(u)}}{2}\frac{|\nabla u|^2}{2} \big( \underline{H'} - H \big) \geq 0.
\end{split}
\]
The maximum principle then yields $\sup_{C_+'}F(u) = \sup_{\partial C_{+}'}F(u) = 1$, which means $C_{+}'$ must be empty, so that $u(\overline{\bB}) \subset \overline{\Omega'}$.
\end{proof}

\subsection{Removal of isolated singularities}\label{subsec:remove}
Let $f, X$ be as in Definition~\ref{defi:perturbed-definition}, and suppose we have a smooth solution $u:(\overline{\RR^2_+}, \partial\RR^2_+) \to (\RR^3, \Sigma)$ of
\begin{equation}\label{eq:half-space-bvp}
\left\{
\begin{array}{l}
\Delta u = f(u)u_{x^1} \times u_{x^2} \text{ in }\RR^2_+,\\
u_{x^2} - \tau X(u) \times u_{x^1} \perp T_u\Sigma \text{ on }\partial \RR^2_+,
\end{array}
\right.
\end{equation}
along with
\begin{equation}\label{eq:finiteness}
\int_{\RR^2_+} |\nabla u|^2 < \infty.
\end{equation}
Formulated this way, the isolated singularity is at infinity, and its removal boils down to showing that the limit of $u(x)$ exists as $x \to \infty$. In fact, we establish a decay estimate on $|\nabla u|$ in Proposition~\ref{prop:limit-at-infinity} below, which is the main result of the section and whose proof is based on the differential inequality derived in Proposition~\ref{prop:decay-ODE}. Crucial to the proof of the latter is Lemma~\ref{lemm:poincare-q}. This approach is inspired by earlier instances of its use in the bubbling analysis of harmonic maps, such as~\cite[Lemma 3.2]{Parker1996},~\cite[Lemma 2.1]{Lin-Wang1998},~\cite[Lemma B.2]{ColdingMinicozzi2008} and~\cite[Lemma 7.4]{Lin-Sun-Zhou2020}. In these works, the counterpart of Lemma~\ref{lemm:poincare-q} is either provided by the classical Wirtinger inequality or derived from it by reflection. Here we prove Lemma~\ref{lemm:poincare-q} without reflecting the solution. 

We start with a slightly stronger version of Lemma~\ref{lemm:weakly-conformal}.
\begin{lemm}\label{lemm:half-space-conformal}
Let $u$ be as above. Then it is weakly conformal.
\end{lemm}
\begin{proof}
By~\eqref{eq:half-space-bvp}, the Hopf differential $\varphi := \bangle{u_z, u_z}$ is holomorphic on $\RR^2_+$ and real on $\partial\RR^2_+$. By~\eqref{eq:finiteness}, $\varphi$ is integrable on $\RR^2_+$. Hence, by performing a reflection and then using Liouville's theorem, we see that $\varphi$ vanishes identically, so that $u$ is weakly conformal.
\end{proof}
Next we consider the conformal transformation $(t, \theta) \to (e^{t}\cos\theta, e^{t}\sin\theta)$ from $\RR \times [0, \pi]$ to $\overline{\RR^2_+} \setminus \{0\}$ and write $u(t, \theta)$ for $u(e^t\cos\theta, e^t \sin\theta)$. In these new independent variables, $u$ maps $\RR \times \{0, \pi\}$ into $\Sigma$ and satisfies
\begin{equation}\label{eq:pmc-t-theta}
\left\{
\begin{array}{l}
\Delta u = f(u)u_{t} \times u_{\theta} \text{ on }\RR \times [0, \pi],\\
u_\theta - \tau X(u) \times u_t \perp T_u\Sigma \text{ on } \RR \times \{0, \pi\}.
\end{array}
\right.
\end{equation}
The assumption~\eqref{eq:finiteness} implies that for all $\eta > 0$ there exists $a  > 0$ such that 
\begin{equation}\label{eq:small-energy-cylinder}
\int_{[a, \infty) \times [0, \pi]} |u_\theta|^2 + |u_t|^2 < \eta^2.
\end{equation}
Repeating the proofs of Proposition~\ref{prop:W14-W2q} and Proposition~\ref{prop:eta-regularity} in the present context and using the Sobolev embedding $W^{1, 4} \to C^{0, \frac{1}{2}}$, we get some $\widetilde{\eta}_0 \in (0, 1)$ and $C > 0$ such that if $\eta \in (0, \widetilde{\eta}_0)$ and if $a > 0$ is chosen so that~\eqref{eq:small-energy-cylinder} holds, then
\begin{equation}\label{eq:small-gradient-cylinder}
|\nabla u|^2(t, \theta) \leq C\int_{[t-1, t + 1] \times [0, \pi]} |\nabla u|^2 < C\eta^2, \text{ for all }(t, \theta) \in (a + 1, \infty) \times [0, \pi].
\end{equation}
Since $u(t, 0), u(t, \pi)$ lie in $\Sigma$, we deduce from~\eqref{eq:small-gradient-cylinder} that
\begin{equation}\label{eq:small-distance-to-sigma}
\dist(u(t, \theta), \Sigma) < C\eta, \text{ for all }(t, \theta) \in (a + 1, \infty) \times [0, \pi].
\end{equation}
Thus, since $X$ is assumed to agree with $\bN\circ\Pi$, and hence with $X \circ \Pi$, near $\Sigma$, decreasing $\widetilde{\eta}_0$ if necessary, we may assume that
\begin{equation}\label{eq:X-unit-length}
X(u(t, \theta)) = X(\Pi(u(t, \theta))), \text{ for all }(t, \theta) \in (a+1, \infty) \times [0, \pi].
\end{equation}
Below,~\eqref{eq:small-gradient-cylinder},~\eqref{eq:small-distance-to-sigma} and~\eqref{eq:X-unit-length} are used repeatedly, often without further comment. Also, to save space, we write $X_y$ for $X(y)$ and introduce
\begin{equation}\label{eq:p-q-defi}
p = u_t + \tau X_u \times u_\theta,\ \ q = u_\theta - \tau X_u \times u_t.
\end{equation}
\begin{lemm}\label{lemm:pq-comparison}
For all $(t, \theta) \in \RR \times [0, \pi]$ we have
\begin{equation}\label{eq:partial-comparable}
(1 - |\tau|)|u_\theta| \leq |q| \leq (1 + |\tau|)|u_\theta|,
\end{equation}
\begin{equation}\label{eq:2nd-partial-comparable}
|q_{t}| + |q_{\theta}| \geq (1 - |\tau|)\cdot \big( |u_{\theta\theta}| + |u_{\theta t}| \big) - C|\tau| \cdot |\nabla u|^2.
\end{equation}
\end{lemm}
\begin{proof}
Since $u$ is weakly conformal and $|X| \leq 1$, we see at once that~\eqref{eq:partial-comparable} holds. Next, differentiating $q$ and replacing $u_{tt}$ by $\Delta u - u_{\theta\theta}$, and again using $|X| \leq 1$, we obtain
\[
\begin{split}
|q_{t}| + |q_{\theta}| =\ & |u_{\theta t} + \tau X_u \times u_{\theta\theta} - \tau X_u \times \Delta u - \tau (\nabla X)_u(u_t) \times u_t|\\
& + |u_{\theta\theta} - \tau X_u \times u_{t\theta} - \tau (\nabla X)_u(u_\theta) \times u_t|\\
\geq\ & |u_{\theta t}| - |\tau| |u_{\theta\theta}| - |\tau| \cdot (|\Delta u| + C|\nabla u|^2) + |u_{\theta\theta}| - |\tau| |u_{t\theta}| - C|\tau| |\nabla u|^2.
\end{split}
\]
Grouping the terms appropriately and noting that $|\Delta u| \leq C|\nabla u|^2$ yields~\eqref{eq:2nd-partial-comparable}.
\end{proof}
\begin{lemm}\label{lemm:poincare-q}
Given $\eta \in (0, \widetilde{\eta}_0)$, whenever $a > 0$ is such that~\eqref{eq:small-energy-cylinder} holds, we have for all $t > a + 1$ that
\[
\Big( 1 - \frac{C\eta^2}{(1 - |\tau|)^2} \Big) \cdot \int_{0}^{\pi} |q(t, \theta)|^2 d\theta \leq 2\int_{0}^{\pi} |q_\theta(t, \theta)|^2 d\theta.
\]
\end{lemm}
\begin{proof}
We first split $q$ into $q^T + q^\perp$, where
\[
q^T := q - \bangle{q, X_u}X_u\ \  \text{ and }\ \ q^\perp := \bangle{q, X_u}X_u.
\]
Noting that $|X(u(t, \theta))| = 1$ on $(a + 1, \infty) \times [0, \pi]$, we have
\begin{equation}\label{eq:q-pythagoras}
|q|^2 = |q^T|^2 + |q^\perp|^2.
\end{equation}
By~\eqref{eq:pmc-t-theta}, $q^T$ vanishes at $\theta = 0$ and $\theta = \pi$, and thus the Poincar\'e inequality gives
\begin{equation}\label{eq:T-poincare}
\int_{0}^{\pi} |q^T|^2 d\theta \leq \int_{0}^{\pi} |(q^T)_\theta|^2 d\theta.
\end{equation}
By a direct computation, we find that
\[
\begin{split}
(q^T)_{\theta} =\ & (q_\theta)^T - \bangle{q, (\nabla X)_u(u_\theta)}X_u - \bangle{q, X_u}(\nabla X)_u(u_\theta).
\end{split}
\]
Combining this with~\eqref{eq:T-poincare} gives
\begin{equation}\label{eq:poincare-t}
\int_{0}^{\pi}\big|q^T\big|^2 d\theta \leq 2\int_{0}^{\pi}|(q_\theta)^T|^2 d\theta + C\eta^2 \int_{0}^{\pi} |q|^2 d\theta.
\end{equation}
Turning to $q^\perp$, since $\bangle{X_u \times u_t, X_u} = 0$, we have
\[
q^\perp = \bangle{u_{\theta}, X_u}X_u.
\]
By the Poincar\'e inequality applied to $\bangle{u_\theta, X_u}$ on $[0, \pi]$ we have
\begin{equation}\label{eq:poincare-n-1}
\begin{split}
\int_{0}^{\pi} \big|q^\perp\big|^2 d\theta =\ & \int_{0}^{\pi} \bangle{u_\theta, X_u}^2 d\theta\\
\leq\ & \int_{0}^{\pi} \big| \bangle{u_\theta, X_u}_\theta \big|^2 d\theta + \frac{1}{\pi}\Big( \int_{0}^{\pi} \bangle{u_\theta, X_u} d\theta \Big)^2.
\end{split}
\end{equation}
For the very last term, we use  $u_\theta(t, \theta) = \big[ u(t, \theta) - u(t, 0) \big]_{\theta}$ and integrate by parts to get
\[
\begin{split}
\int_{0}^{\pi} \bangle{u_\theta, X_u} d\theta=\ & \bangle{u(t, \pi) - u(t, 0), X_{u(t, \pi)}} - \int_{0}^{\pi}\bangle{u(t, \theta) - u(t, 0), (\nabla X)_u(u_\theta)} d\theta\\
\leq\ & C|u(t, \pi) - u(t, 0)|^2 + C\eta\int_{0}^{\pi}|u(t, \theta) - u(t, 0)| d\theta,
\end{split}
\]
where the squared term in the last line comes from the fact that $u(t, 0), u(t, \pi) \in \Sigma$ and that $X_{u(t, \pi)}$ is a unit normal to $T_{u(t, \pi)}\Sigma$. Squaring the above estimate leads to
\begin{equation}\label{eq:poincare-n-2}
\begin{split}
\Big( \int_{0}^{\pi} \bangle{u_\theta, X_u} d\theta \Big)^2 \leq\ & C\eta^2 |u(t, \pi) - u(t, 0)|^2 + C\eta^2 \int_{0}^{\pi} |u(t, \theta) - u(t, 0)|^2 d\theta\\
\leq\ & C\eta^2 \int_{0}^{\pi} |u_\theta|^2 d\theta \leq \frac{C\eta^2}{(1 - |\tau|)^2}\int_{0}^{\pi} |q|^2 d\theta,
\end{split}
\end{equation}
where in getting that last inequality we used Lemma~\ref{lemm:pq-comparison}. On the other hand, for the first term on the last line of~\eqref{eq:poincare-n-1}, we have
\[
\begin{split}
\bangle{u_\theta, X_u}_\theta =\ & \bangle{q, X_u}_\theta = \bangle{q_\theta, X_u} + \bangle{q, (\nabla X)_u(u_\theta)}.
\end{split}
\]
Therefore
\[
\int_{0}^{\pi} |\bangle{u_\theta, X_u}_\theta|^2 d\theta \leq 2\int_{0}^{\pi} \big| (q_\theta)^\perp \big|^2 d\theta + C\eta^2 \int_{0}^{\pi}|q|^2 d\theta.
\]
Upon substituting this and~\eqref{eq:poincare-n-2} into~\eqref{eq:poincare-n-1}, adding the resulting estimate to~\eqref{eq:poincare-t}, and then recalling~\eqref{eq:q-pythagoras}, we get that
\[
\begin{split}
\int_{0}^{\pi}|q|^2 d\theta =\ & \int_{0}^{\pi} |q^T|^2 d\theta + \int_{0}^{\pi}|q^\perp|^2 d\theta\leq 2\int_{0}^{\pi}|q_\theta|^2 d\theta + \frac{C\eta^2}{(1 - |\tau|)^2} \int_{0}^{\pi}|q|^2 d\theta.
\end{split}
\]
Rearranging gives the asserted inequality.
\end{proof}
\begin{lemm}\label{lemm:replace-tt}
We have the following pointwise bounds on $\RR \times [0, \pi]$.
\vskip 1mm
\begin{enumerate}
\item[(a)] $|(X_u \times u_t)_{tt} - X_u \times u_{ttt}| \leq C(|\nabla u|^3 + |\nabla u||u_{\theta\theta}|)$.
\vskip 1mm
\item[(b)] $|(\Delta u)_t| \leq C(|\nabla u|^3 + |\nabla u| |u_{\theta\theta}| + |\nabla u||u_{\theta t}|)$.
\vskip 1mm
\item[(c)] $|p_t - (u_{tt} + \tau X_u \times u_{\theta t})| \leq C|\tau||\nabla u|^2$.
\vskip 1mm
\item[(d)] $|q_{\theta} + p_t| \leq  C|\nabla u|^2$.
\end{enumerate}
\end{lemm}
\begin{proof}
By direct computation we have
\[
(X_u \times u_t)_{tt} - X_u \times u_{ttt} =  [(\nabla X)_u(u_{tt}) + (\nabla^2 X)_u(u_t, u_t)] \times u_t +  2(\nabla X)_u(u_t) \times u_{tt},
\]
\[
(\Delta u)_t = (\nabla f)_u(u_t) u_t \times u_\theta + f(u) \cdot u_{tt} \times u_\theta + f(u) u_t \times u_{\theta t}.
\]
Parts (a) and (b) then follow easily upon replacing $u_{tt}$ by $-u_{\theta\theta} + f(u)u_t \times u_\theta$. Part (c) and part (d) are straightforward from the definition of $p$ and $q$ in~\eqref{eq:p-q-defi}.
\end{proof}

\begin{prop}\label{prop:decay-ODE}
There exists $a > 0$ such that 
\[
\frac{d^2}{dt^2}\int_{0}^{\pi} |q(t, \theta)|^2 d\theta \geq \frac{1}{2}\int_{0}^{\pi}|q(t, \theta)|^2 d\theta, \text{ for all $t > a + 1$.}
\]
\end{prop}
\begin{proof}
With $\eta$ to be determined, we choose $a$ so large that~\eqref{eq:small-energy-cylinder} holds. Twice differentiating $\int_{0}^{\pi}\frac{|q(t, \theta)|^2}{2} d\theta$ with respect to $t$, integrating by parts and using Lemma~\ref{lemm:replace-tt}(a), we get
\begin{equation}\label{eq:energy-tt-derivative}
\begin{split}
\frac{d^2}{dt^2}\int_{0}^{\pi} \frac{|q|^2}{2} d\theta =\ & \int_{0}^{\pi} |q_{t}|^2 d\theta +  \int_{0}^{\pi} \bangle{q, u_{\theta tt}} d\theta  - \tau \int_{0}^{\pi}\bangle{q, (X_u \times u_t)_{tt}} d\theta \\
\geq\ & \bangle{q, u_{tt}}\big|_{\theta = 0}^{\theta = \pi} 
 + \int_{0}^{\pi}|q_t|^2 - \bangle{q_\theta, u_{tt}}d\theta  - \tau\int_{0}^{\pi} \bangle{q, X_u \times u_{ttt}} d\theta\\
& - C|\tau|\int_{0}^{\pi} |q|(|\nabla u|^3 + |\nabla u||u_{\theta\theta}|)d\theta.
\end{split}
\end{equation}
For the term involving $u_{ttt}$, we get after writing $u_{ttt}$ as $(\Delta u)_t - u_{\theta\theta t}$, using Lemma~\ref{lemm:replace-tt}(b), and again integrating by parts that
\[
\begin{split}
-\tau\int_{0}^{\pi}\bangle{q, X_u \times u_{ttt}} d\theta =\ & -\tau\int_{0}^{\pi} \bangle{q, X_u \times (\Delta u)_t} d\theta + \tau \int_{0}^{\pi} \bangle{q, X_u \times u_{\theta\theta t}} d\theta\\
=\ & -\tau\int_{0}^{\pi} \bangle{q, X_u \times (\Delta u)_t} d\theta + \tau \bangle{q, X_u \times u_{\theta t}}\big|_{\theta = 0}^{\theta = \pi}\\
&- \tau \int_{0}^{\pi} \bangle{q_\theta, X_u \times u_{\theta t}}d\theta  - \tau \int_{0}^{\pi} \bangle{q,(\nabla X)_u(u_\theta) \times u_{\theta t}} d\theta\\
\geq\ & -C|\tau| \int_{0}^{\pi} |q|(|\nabla u|^3 + |\nabla u||u_{\theta\theta}| + |\nabla u||u_{\theta t}|) d\theta\\
& + \tau \bangle{q, X_u \times u_{\theta t}}\big|_{\theta = 0}^{\theta = \pi} - \tau \int_{0}^{\pi} \bangle{q_\theta, X_u \times u_{\theta t}}d\theta.
\end{split}
\]
Note that the boundary term vanishes since $q$ is parallel to $X_u$ and hence orthogonal to $X_u \times u_{\theta t}$ when $\theta = 0, \pi$. Substituting the above into~\eqref{eq:energy-tt-derivative} and using Lemma~\ref{lemm:replace-tt}(c) gives
\begin{equation}\label{eq:energy-tt-derivative-2}
\begin{split}
\frac{d^2}{dt^2}\int_{0}^{\pi} \frac{|q|^2}{2} d\theta\geq\ & \bangle{q, u_{tt}} \big|_{\theta = 0}^{\theta = \pi} +\int_{0}^{\pi} |q_{t}|^2 -  \bangle{q_{\theta}, p_t}d\theta\\
 & - C|\tau| \int_{0}^{\pi} |q| \cdot \big(|\nabla u|^3 + |\nabla u||u_{\theta\theta}| + |\nabla u||u_{\theta t}| \big) d\theta\\
 & - C|\tau| \int_{0}^{\pi} |q_{\theta}| |\nabla u|^2 d\theta.
\end{split}
\end{equation}
To continue, we use Lemma~\ref{lemm:replace-tt}(d) to see that
\begin{equation}\label{eq:replace-tt-d}
-\int_{0}^{\pi}\bangle{q_\theta, p_t} d\theta \geq \int_{0}^{\pi} |q_\theta|^2 d\theta - C\int_{0}^{\pi}|q_\theta||\nabla u|^2 d\theta.
\end{equation}
On the other hand, for the boundary term in~\eqref{eq:energy-tt-derivative-2}, observe that when $\theta = 0, \pi$, we have 
\[
q = \bangle{q, X_u}X_u,\ \ \bangle{u_{tt}, X_u} = -\bangle{u_t, (\nabla X)_u(u_t)},
\]
and hence by the fundamental theorem of calculus and some straightforward computation we deduce that
\begin{equation}\label{eq:boundary-term}
\begin{split}
\bangle{q, u_{tt}}\big|_{\theta = 0}^{\theta = \pi} =\ & - \bangle{q, X_u}\bangle{u_{t}, (\nabla X)_u(u_t)}\big|_{\theta = 0}^{\theta = \pi}\\
=\ & -\int_{0}^{\pi} \big( \bangle{q, X_u}\bangle{u_{t}, (\nabla X)_u(u_t)}\big)_{\theta}\  d\theta\\
\geq\ & -C\int_{0}^{\pi} |q_\theta||\nabla u|^2 + |q|(|\nabla u|^3 + |\nabla u||u_{t\theta}|) d\theta.
\end{split}
\end{equation}
Combining~\eqref{eq:replace-tt-d} and~\eqref{eq:boundary-term} with~\eqref{eq:energy-tt-derivative-2}, and also using~\eqref{eq:small-gradient-cylinder}, we obtain
\begin{equation}\label{eq:energy-tt-derivative-3}
\begin{split}
\frac{d^2}{dt^2}\int_{0}^{\pi} \frac{|q|^2}{2} d\theta\geq\ & \int_{0}^{\pi} |q_t|^2 + |q_\theta|^2 d\theta - C\eta\int_{0}^{\pi}|q| (|\nabla u|^2 + |u_{t\theta}| + |u_{\theta\theta}|) d\theta\\
& - C\eta \int_{0}^{\pi}|q_\theta||\nabla u| d\theta.
\end{split}
\end{equation}
We then apply Lemma~\ref{lemm:pq-comparison} to the second and third terms on the right hand side to get
\begin{equation}\label{eq:2nd-line-2}
\int_{0}^{\pi} |q| \cdot \big( |\nabla u|^2 + |u_{\theta \theta}| + |u_{\theta t}| \big) d\theta \leq C(1 - |\tau|)^{-3}\int_{0}^{\pi} |q| \cdot \big( |q|^2 + |q_{t}| + |q_{\theta}| \big) d\theta,
\end{equation}
and 
\begin{equation}\label{eq:3rd-line-2}
 \int_{0}^{\pi} |q_{\theta}| |\nabla u| d\theta \leq C(1 - |\tau|)^{-1}  \int_{0}^{\pi}  |q_{\theta}| |q| d\theta.
\end{equation}
Putting~\eqref{eq:2nd-line-2} and~\eqref{eq:3rd-line-2} back into~\eqref{eq:energy-tt-derivative-3} and using Young's inequality gives
\begin{equation}\label{eq:energy-tt-derivative-4}
\begin{split}
\frac{d^2}{dt^2}\int_{0}^{\pi} \frac{|q|^2}{2} d\theta \geq\ & \big(1 - C (1 - |\tau|)^{-3}\eta\big)\cdot \int_{0}^{\pi}  |q_{t}|^2 +  |q_{\theta}|^2 d\theta - C(1 - |\tau|)^{-3}\eta \cdot \int_{0}^{\pi} |q|^2 d\theta.
\end{split}
\end{equation}
Applying Lemma~\ref{lemm:poincare-q}, we finally get
\[
\frac{d^2}{dt^2}\int_{0}^{\pi} \frac{|q|^2}{2} d\theta \geq \frac{1}{2}\Big( 1 - C(1 - |\tau|)^{-3}\eta \Big) \cdot \int_{0}^{\pi} |q|^2 d\theta.
\]
The asserted estimate follows upon choosing $\eta$ small enough so that
\[
C(1 - |\tau|)^{-3}\eta < \frac{1}{2}.
\]
\end{proof}
\begin{prop}\label{prop:limit-at-infinity}
Suppose $u:(\overline{\RR^2_+}, \partial\RR^2_+) \to (\RR^3, \Sigma)$ is a smooth solution of~\eqref{eq:half-space-bvp} satisfying~\eqref{eq:finiteness}. With $a > 0$ as given by Proposition~\ref{prop:decay-ODE} and $\kappa = \frac{1}{\sqrt{2}}$, we have
\begin{equation}\label{eq:grad-exp-decay}
\sup_{|x| \geq e^{a + 2}} |x|^{\kappa + 2}|\nabla u(x)|^2 < \infty.
\end{equation}
In particular, $\lim_{x \to \infty}u(x)$ exists.
\end{prop}
\begin{proof}
It suffices to establish~\eqref{eq:grad-exp-decay} as the second conclusion follows easily. Following our convention throughout this section, we write $u(t, \theta)$ for $u(e^t\cos\theta, e^t\sin\theta)$. Then whenever $t_2 > t > t_1 \geq a + 1$, we have from Proposition~\ref{prop:decay-ODE} that
\begin{equation}\label{eq:ODE-comparison}
\begin{split}
\int_{0}^{\pi} |q(t, \theta)|^2 d\theta \leq \frac{\sinh \kappa(t - t_1)}{\sinh \kappa(t_2 - t_1)} \cdot \int_{0}^{\pi}|q(t_2, \theta)|^2 d\theta + \frac{\sinh \kappa(t_2 - t)}{\sinh \kappa(t_2 - t_1)} \cdot \int_{0}^{\pi}|q(t_1, \theta)|^2 d\theta.
\end{split}
\end{equation}
Since the function $t \mapsto \int_{0}^\pi |q(t, \theta)|^2 d\theta$ is bounded on $(a + 1, \infty)$ by~\eqref{eq:small-gradient-cylinder}, upon letting $t_2$ tend to infinity in~\eqref{eq:ODE-comparison} while keeping $t, t_1$ fixed, we obtain 
\[
\int_{0}^{\pi}|q(t, \theta)|^2 d\theta \leq e^{-\kappa(t - t_1)}\int_{0}^{\pi} |q(t_1, \theta)|^2 d\theta, \text{ for all }t > t_1 \geq a+1.
\]
Taking $t_1 = a+1$ and integrating from $t$ to infinity gives
\[
\int_{t}^{\infty} \int_{0}^{\pi} |q(s, \theta)|^2 d\theta ds \leq \frac{e^{-\kappa(t - a - 1)}}{\kappa}\int_{0}^{\pi}|q(a+1, \theta)|^2 d\theta, \text{ for all }t > a+ 1.
\]
From this, we infer with the help of Lemma~\ref{lemm:pq-comparison} and the first inequality in~\eqref{eq:small-gradient-cylinder} that
\[
\sup_{(t, \theta) \in [a+2, \infty) \times [0, \pi]} e^{\kappa t}|\nabla u(t, \theta)|^2 < \infty, 
\]
which gives~\eqref{eq:grad-exp-decay}. 
\end{proof}
\begin{coro}\label{coro:removal-singularity}
Suppose $u:(\overline{\RR^2_+}, \partial\RR^2_+) \to (\RR^3, \Sigma)$ is a smooth solution of~\eqref{eq:half-space-bvp} satisfying~\eqref{eq:finiteness}. Given $x_0 \in \partial \bB$ and a conformal map $G$ from $(\bB, \partial \bB \setminus \{x_0\})$ onto $(\RR^2_+, \partial \RR^2_+)$, the composition $v: = u \circ G$ extends to a smooth map on $\overline{\bB}$.
\end{coro}
\begin{proof}
The map $v$ initially is a smooth solution of~\eqref{eq:cmc-bvp} on $\overline{\bB} \setminus \{x_0\}$. Proposition~\ref{prop:limit-at-infinity} implies that $v$ extends continuously to $\overline{\bB}$ and that there exists some $q > 2$ such that $\nabla v \in L^q(\bB)$. Smoothness then follows from Remark~\ref{rmk:C0-regularity}.
\end{proof}

\subsection{Branch points}\label{subsec:branch-points}

Suppose we have a smooth, weakly conformal and non-constant solution $u:(\bB^+ \cup \bT, \bT) \to (\RR^3, \Sigma)$ of
\begin{equation}\label{eq:half-ball-bvp}
\left\{
\begin{array}{l}
\Delta u = f(u)u_{x^1} \times u_{x^2} \text{ in }\bB^+,\\
u_{x^2} - \tau X(u) \times u_{x^1} \perp T_u\Sigma \text{ on }\bT,
\end{array}
\right.
\end{equation}
and that $\nabla u(0) = 0$. Our goal in this section is to show that the origin is an isolated zero of $\nabla u$, and that the tangent plane to the image of $u$, namely $\Span\{u_{x^1}(x), u_{x^2}(x)\}$, has a limit as $x \to 0$. The methods are inspired by the work of Hildebrandt and Nitsche~\cite{Hildebrandt-Nitsche1979} on free boundary minimal surfaces. The main results are Proposition~\ref{prop:smooth-extension} and Corollary~\ref{coro:branch}. 

Since $u$ is smooth and $u(0) \in \Sigma$, and since~\eqref{eq:half-ball-bvp} is conformally invariant, dilating the domain if necessary, we may assume that the derivatives of $u$ of all orders are bounded on $\bB^+ \cup \bT$, that $X = X \circ \Pi$ on $u(\bB^+ \cup \bT)$, and that the latter is contained in $U$ for some $U \in \cF$. We then let
\[
\Phi: U \to V
\]
be the chart described in Section~\ref{subsec:function-spaces}. With $\widehat{u} = \Phi \circ u$, and adopting the notation of Section~\ref{subsec:function-spaces}, we compute based on~\eqref{eq:half-ball-bvp} that 
\begin{equation}\label{eq:half-ball-fermi}
\left\{
\begin{array}{l}
\Delta \widehat{u}^i = -\Gamma_{jk}^i(\widehat{u}) \bangle{\nabla\widehat{u}^j, \nabla \widehat{u}^k} + \widehat{f}(\widehat{u}) P^i_{jk}(\widehat{u})\widehat{u}^j_{x^1} \widehat{u}^k_{x^2} \text{ in }\bB^+,\\
\widehat{u}_{x^2}^i + \tau J_{j}^i(\widehat{u})\widehat{u}^j_{x^1} = 0 \text{ on }\bT \text{ for }i = 1, 2,\\
\widehat{u}^3 = 0 \text{ on }\bT.
\end{array}
\right.
\end{equation}
Also, the bounds in~\eqref{eq:metric-close} and the fact that $u$ is weakly conformal implies that 
\begin{equation}\label{eq:conformal-hat}
\frac{2 + 2|\tau|}{3 + |\tau|} |\widehat{u}_{x^1}| \leq |\widehat{u}_{x^2}| \leq \frac{3 + |\tau|}{2 + 2|\tau|}|\widehat{u}_{x^1}|.
\end{equation}
Next, following the approach of the previous section, we define, 
\[
p = \widehat{u}_{x^1} - \tau J(\widehat{u})  \widehat{u}_{x^2},\ \ q = \widehat{u}_{x^2} + \tau J(\widehat{u})\widehat{u}_{x^1},
\]
where we think of each $J(y)$ as a linear map from $\RR^3$ to itself. Then we have
\begin{equation}\label{eq:pq-boundary}
p^3 = 0,\ \ q^{1} = q^{2} = 0 \text{ on }\bT.
\end{equation}
We find it most convenient to carry out some of the subsequent calculations in terms of complex variables. Thus, below we introduce
\begin{equation}\label{eq:F-pq-defi}
\begin{split}
F =\ & p - \sqrt{-1}q\\
=\ & 2\widehat{u}_z - 2\sqrt{-1}\tau J(\widehat{u})\widehat{u}_z,
\end{split}
\end{equation}
where in the second line we have extended the maps $\{J(y)\}_{y \in V}$ by complex linearity to $\CC^3$. Note that the estimate~\eqref{eq:J-close} continues to hold.
\begin{lemm}\label{lemm:almost-CR-complex}
There exist families $\{T_y\}_{y \in V}$ and $\{\overline{T}_y\}_{y \in V}$ of linear maps from $\CC^3$ to itself, depending smoothly on $y$, such that 
\begin{equation}\label{eq:Ty-two-side-bound}
\frac{|\xi|}{4} \leq |T_y(\xi)|, |\overline{T}_y(\xi)| \leq \frac{|\xi|}{1 - |\tau|}, \text{ for all } \xi \in \CC^3,
\end{equation}
and
\begin{equation}\label{eq:u-z-F}
\widehat{u}_z =  T_{\widehat{u}}(F),\ \ \widehat{u}_{\overline{z}} =  \overline{T}_{\widehat{u}}(\overline{F}).
\end{equation}
Moreover, there exist $\Lambda_{jk}^i \in C^{\infty}(\bB^+ \cup \bT; \CC)$ for $i, j, k \in \{1, 2, 3\}$ so that 
\begin{equation}\label{eq:Fz-uz}
F_{\overline{z}}^i = \Lambda_{jk}^i \overline{F^j}F^k.
\end{equation}
\end{lemm}
\begin{proof}
Define $S_y: \CC^3 \to \CC^3$ by $S_y(\xi) = 2\xi - 2\sqrt{-1}\tau J(y)\xi$. By~\eqref{eq:J-close} we have
\[
(1 - |\tau|)|\xi| \leq |S_y(\xi)| \leq 4|\xi|,
\]
which shows that each $S_y$ is invertible, and that the inverse, denoted $T_y$, satisfies~\eqref{eq:Ty-two-side-bound}. The smooth dependence of $T_y$ on $y$ follows since $S_y$ depends smoothly on $y$. Also, the first equation in~\eqref{eq:u-z-F} is an immediate consequence of the fact that $F = S_{\widehat{u}}(\widehat{u}_z)$. Next, letting 
\[
\overline{T}_y(\xi) = \overline{T_y(\overline{\xi})},
\]
we get the second equation in~\eqref{eq:u-z-F} and that each $\overline{T}_y$ satisfies~\eqref{eq:Ty-two-side-bound} as well.

To obtain~\eqref{eq:Fz-uz}, we apply $\partial_{\overline{z}}$ to~\eqref{eq:F-pq-defi}, use the PDE in~\eqref{eq:half-ball-fermi} along with the following standard relations
\[
\widehat{u}_{z\overline{z}} = \frac{1}{4}\Delta\widehat{u},\ \ \widehat{u}_{x^1} = \widehat{u}_z + \widehat{u}_{\overline{z}},\ \ \widehat{u}_{x^2} = \sqrt{-1}(\widehat{u}_z - \widehat{u}_{\overline{z}}),
\]
and finally use~\eqref{eq:u-z-F} to express $\widehat{u}_z$ and $\widehat{u}_{\overline{z}}$ in terms of $F$ and $\overline{F}$.
\end{proof}

\begin{lemm}\label{lemm:reflection}
$F$ and $\Lambda_{jk}^i$ admit extensions to $\bB$, denoted with the same letters and lying respectively in $C^{0, 1}(\bB;\CC^3)$ and $L^{\infty}(\bB;\CC)$, so that
\begin{equation}\label{eq:almost-CR-eq}
F_{\overline{z}}^i = \Lambda^{i}_{jk}\overline{F^j}F^k \text{ almost everywhere on }\bB.
\end{equation}
\end{lemm}
\begin{proof}
Let $R: \RR^3 \to \RR^3$ denote reflection across $\{y^3 = 0\}$. We extend $p, q$ to $\bB$ by requiring that
\[
p(x^1, x^2) = R(p(x^1, -x^2)),\ \ q(x^1, x^2) = -R(q(x^1, -x^2)), \text{ for }x^2 < 0.
\]
By~\eqref{eq:pq-boundary} and the smoothness of $\widehat{u}$ on $\bB^+ \cup \bT$, we see that the extended $p$ and $q$ are both Lipschitz on $\bB$. Still letting
\[
F = p - \sqrt{-1} q,
\]
we have that
\begin{equation}\label{eq:F-reflection}
F(x^1, x^2) = R(\overline{F(x^1, -x^2)}) \text{ for }x^2 < 0.
\end{equation}
Here we are extending $R$ to $\CC^3$ by complex linearity. By~\eqref{eq:F-reflection} and a simple calculation we get
\begin{equation}\label{eq:F-z-lower}
F_{\overline{z}}(x^1, x^2) = R(\overline{F_{\overline{z}}(x^1, -x^2)}), \text{ for }x^2 < 0.
\end{equation}
Replacing $F_{\overline{z}}(x^1, -x^2)$ using~\eqref{eq:Fz-uz} and then applying~\eqref{eq:F-reflection} to express $F(x^1, -x^2)$ and $\overline{F}(x^1, -x^2)$ in terms of $F(x^1, x^2)$ and $\overline{F}(x^1, x^2)$ shows how to define $\Lambda^i_{jk}$ on the lower half disk so that~\eqref{eq:almost-CR-eq} holds.
\end{proof}
We may now state the main results of this section.
\begin{prop}\label{prop:smooth-extension}
There exist $k \in \NN$ and a smooth function $\widehat{F}: \bB^+ \cup \bT \to \CC^3$ such that $\widehat{F}^3(0) \neq 0$ and 
\[
\widehat{u}_z(z) = z^k \widehat{F}(z) \text{ on $\bB^+ \cup \bT$ near }0.
\]
\end{prop}
\begin{proof}
Let $(\Lambda^i_{jk}) \subset L^{\infty}(\bB)$ be as given by Lemma~\ref{lemm:reflection}, and define $L^i_{j} = \Lambda^i_{kj}\overline{F}^k$, so that each $L^i_{j}$ is smooth on $\bB^+ \cup \bT$ and bounded on $\bB$. Then we may find $\Upsilon : \bB \to \CC^{3 \times 3}$ lying in $W^{1, p}(\bB)$ for all $1 < p < \infty$ so that $\Upsilon_{\overline{z}} = L$ on $\bB$. Also,~\eqref{eq:almost-CR-eq} becomes 
\begin{equation}\label{eq:F-L}
F_{\overline{z}}^i = L^{i}_j F^j,
\end{equation}
from which we see that $e^{-\Upsilon}F$ is holomorphic in the distributional sense, and hence smooth, on $\bB$. Moreover, by assumption it vanishes at the origin, and the order of vanishing must be finite for otherwise $F$, and hence $\widehat{u}_z$, would vanish on $\bB^+ \cup \bT$, making $\widehat{u}$ constant. Thus, we get some $k \in \NN$ and a holomorphic function $G:\bB \to \CC^3$ so that $G(0) \neq 0$ and that 
\begin{equation}\label{eq:uz-factorization}
e^{-\Upsilon(z)}F(z) = z^k G(z) \text{ for }z \in \bB.
\end{equation}
Consequently, if we let 
\[
\widehat{F} =  T_{\widehat{u}}(e^{\Upsilon}G),
\]
then $\widehat{F}$ is continuous on $\bB^+ \cup \bT$ and agrees with $z^{-k}\widehat{u}_z$ on $(\bB^+ \cup \bT) \setminus \{0\}$. Moreover, $\widehat{F}(0) \neq 0$ since $G(0) \neq 0$.

To show that $\widehat{F}^3 \neq 0$, note that by~\eqref{eq:pq-boundary} and the fact that $J^3_{j}(y) = 0$ for all $y \in V$, we have 
\[
|\widehat{u}^3_{x^2}| = |q^3 | = |q| \text{ on } \bT.
\]
On the other hand, by~\eqref{eq:J-close} and~\eqref{eq:conformal-hat}, we have on $\bT$ that
\[
|q| \geq \frac{1 - |\tau|}{4} |\widehat{u}_{x^2}|.
\]
Combining these two estimates and using~\eqref{eq:conformal-hat} again, we deduce the existence of a positive constant $c_\tau$ such that 
\[
|\widehat{u}_z^3| \geq c_\tau |\widehat{u}_z| \text{ on }\bT.
\] 
Dividing this by $z^{k}$ on $\bT \setminus \{0\}$ and letting $z$ tend to zero gives 
\[
|\widehat{F}^3(0)| \geq c_\tau |\widehat{F}(0)| > 0.
\]

It remains to prove that $\widehat{F}$ is smooth on $\bB^+ \cup \bT$, and for that it suffices to prove that $e^{\Upsilon}G$ is. To that end, we observe the following.
\vskip 1mm
\begin{enumerate}
\item[(i)] $e^{\Upsilon}G$ lies in $W^{1, p}(\bB)$ for all $1 < p < \infty$. 
\vskip 1mm
\item[(ii)] On $(\bB^+ \cup \bT) \setminus \{0\}$, $e^{\Upsilon}G$ agrees with $z^{-k}F$ and hence satisfies, by~\eqref{eq:F-L}, that 
\[
\big(e^{\Upsilon}G\big)_{\overline{z}} = L\cdot e^{\Upsilon}G.
\] 
In other words, splitting $e^{\Upsilon}G$ and $L$ into their real and imaginary parts as
\[
e^{\Upsilon}G = \widetilde{p} - \sqrt{-1}\widetilde{q},\ \ L = M + \sqrt{-1}N,
\]
we have that $M, N$ are smooth on $\bB^+ \cup \bT$, and that
\begin{equation}\label{eq:tilde-pq-PDE}
\left\{
\begin{array}{l}
\widetilde{p}_{x^1} + \widetilde{q}_{x^2} = 2(M\widetilde{p} + N\widetilde{q}),\\
\widetilde{p}_{x^2} - \widetilde{q}_{x^1} = 2(N\widetilde{p} - M\widetilde{q} ),
\end{array}
\right.
\text{ on }(\bB^+ \cup \bT)\setminus \{0\}.
\end{equation}
\vskip 1mm
\item[(iii)] By~\eqref{eq:pq-boundary}, on $\bT \setminus \{0\}$ there holds
\[
\widetilde{p}^3 = \frac{p^3}{x^k} = 0 \text{ and } \widetilde{q}^{i} = \frac{q^{i}}{x^k} = 0 \text{ for }i = 1, 2.
\]
\end{enumerate}
In addition, letting
\[
\cH_T = \{\xi \in W^{1, 2}(\bB^+; \RR^3)\ |\ \xi^{3}|_{\bT} = 0\},\ \ \cH_{\perp} = \{\eta \in W^{1, 2}(\bB^+; \RR^3)\ |\ \eta^{1}|_{\bT} = \eta^{2}|_{\bT} = 0\},
\]
we see by a simple approximation argument that
\[
\int_{\bB^+}\xi_{x^1} \cdot \eta_{x^2}  - \xi_{x^2} \cdot \eta_{x^1} = 0,
\]
whenever $\xi \in \cH_T$, $\eta \in \cH_\perp$ and at least one of the two has vanishing trace on $\partial \bB \cap \RR^2_+$. Combining the above observations, we see that $\widetilde{p} \in \cH_T$, $\widetilde{q} \in \cH_\perp$, and that
\[
\int_{\bB^+}\bangle{\nabla \widetilde{p}, \nabla \xi} + \bangle{\nabla\widetilde{q}, \nabla \eta} = 2\int_{\bB^+} (M\widetilde{p} + N\widetilde{q})(\xi_{x^1} + \eta_{x^2}) + (N\widetilde{p} - M\widetilde{q} )(\xi_{x^2} - \eta_{x^1}),
\]
for all $\xi \in \cH_T$ and $\eta \in \cH_{\perp}$ that vanish on $\partial \bB \cap \RR^2_+$. Since $M, N$ are smooth on $\bB^+ \cup \bT$, by difference quotient arguments in the $x^1$-direction applied to the above equation along with differentiation of~\eqref{eq:tilde-pq-PDE}, we conclude that $\widetilde{p}$ and $\widetilde{q}$, and hence $e^{\Upsilon}G$, are smooth on $\bB^+ \cup \bT$. 
\end{proof}

Proposition~\ref{prop:smooth-extension} has the following global consequence.
\begin{coro}\label{coro:branch}
Let $u:(\overline{\bB}, \partial \bB) \to (\RR^3, \Sigma)$ be a smooth, non-constant solution of
\begin{equation}\label{eq:pmc-bvp}
\left\{
\begin{array}{l}
\Delta u = f(u) u_{x^1} \times u_{x^2} \text{ on }\bB,\\
u_{r} + \tau X(u) \times u_\theta \perp T_u\Sigma \text{ on }\partial\bB,
\end{array}
\right.
\end{equation}
and denote by $\cS$ the set of zeros of $\nabla u$ on $\overline{\bB}$. Then we have the following.
\begin{enumerate}
\item[(a)] $\cS$ is a finite set. 
\vskip 1mm
\item[(b)] The vector bundle over $\overline{\bB} \setminus \cS$ whose fiber at $x$ is $\Span\{u_{x^1}(x), u_{x^2}(x)\}$ extends smoothly across $\cS$ to an oriented rank-two subbundle, denoted $E$, of $\overline{\bB} \times \RR^3$.
\vskip 1mm
\item[(c)] $E_x \neq T_{u(x)}\Sigma$ for all $x \in \partial \bB$.
\end{enumerate}
In particular, $u$ is a branched immersion in the sense defined below Remark~\ref{rmk:tau}.
\end{coro}
\begin{proof}
By Lemma~\ref{lemm:weakly-conformal}, $u$ is weakly conformal. Statements (a) and (b) are then immediate consequences of Proposition~\ref{prop:smooth-extension} and its interior version, the latter being classical. To prove statement (c), note by~\eqref{eq:pmc-bvp} that on $\partial \bB$ and away from $\cS$, 
\[
\big| \bangle{u_r, X(u)} \big| = \big| \bangle{u_r + \tau X(u) \times u_{\theta}, X(u)} \big| =  \big| u_r + \tau X(u) \times u_{\theta} \big| \geq (1 - |\tau|) |u_r| > 0,
\]
which implies that $E \neq T_u\Sigma$. That the same holds at the branch points follows from Proposition~\ref{prop:smooth-extension}, specifically the conclusion that $\widehat{F}^3(0) \neq 0$.
\end{proof}
\section{Index comparison}\label{sec:index-comparison}
In this section we establish index comparison results which are needed at the end of Section~\ref{sec:proof} to establish the last claim of Theorem~\ref{thm:main-1}. The bilinear forms of interest are introduced in Section~\ref{subsec:index-notation}, and the two sections that follow compare their indices.
\subsection{Notation}\label{subsec:index-notation}
Let $u: (\overline{\bB}, \partial \bB) \to (\RR^3, \Sigma)$ be a smooth, non-constant solution of~\eqref{eq:H-cmc-bvp}. Then $u$ is a weakly conformal by Lemma~\ref{lemm:weakly-conformal} and a branched immersion by Corollary~\ref{coro:branch}. Below we define $\lambda$ by the relation
\[
|\nabla u|^2 = 2\lambda^2,
\]
and let $\cS$ denote the set of branch points of $u$ and $E$ the oriented rank-two subbundle of $\overline{\bB} \times \RR^3$ given by Corollary~\ref{coro:branch}. The orthogonal complement of $E$ in $\overline{\bB} \times \RR^3$ is denoted by $E^\perp$. Next, recall that the induced metric and the orientation on $E$ turns it into a complex line bundle. Accordingly, its complexification, $E_{\CC}$, splits into $E^{1, 0} \oplus E^{0, 1}$, which away from $\cS$ are spanned over $\CC$ by $u_z$ and $u_{\overline{z}}$, respectively. We denote orthogonal projection onto $E^{1, 0}$ and $E^{0, 1}$ by $(\ \cdot \ )^{1, 0}$ and $(\ \cdot \ )^{0, 1}$. Similarly, orthogonal projection onto $E$ and $E^{\perp}$ are denoted $(\ \cdot \ )^T$ and $(\ \cdot\ )^\perp$. The usual directional derivative on $\RR^3$ induces connections $\nabla^T$ on $E$ and $\nabla^\perp$ on $E^\perp$. Note that $\nabla^T$ preserves both $E^{1, 0}$ and $E^{0, 1}$. Note also that we may choose a unit-length global section $\nu$ of $E^{\perp}$ that agrees with $\lambda^{-2}u_{x^1} \times u_{x^2}$ away from $\cS$. 

For all $x \in \partial \bB$, we see from Corollary~\ref{coro:branch}(c) that the orthogonal projection of $\nu(x)$ onto $T_{u(x)}\Sigma$ and that of $X(u(x))$ onto $E_x$ are both non-zero. Thus it makes sense to define 
\begin{equation}\label{eq:X-nu-tilde-def}
\widetilde{\nu} = \frac{\proj_{T_u\Sigma}(\nu)}{|\proj_{T_u\Sigma}(\nu)|},\ \ \widetilde{X} = \frac{\proj_{E}(X_u)}{|\proj_{E}(X_u)|},
\end{equation}
where again we write $X_y$ to mean $X(y)$. Next, still on $\partial \bB$, but away from $\cS$, we introduce the vector fields
\[
\bn = \lambda^{-1}u_r,\ \ \bt = \lambda^{-1}u_\theta.
\]
Then $\nu = \bn \times \bt$. Moreover, from the boundary condition in~\eqref{eq:H-cmc-bvp}, we have
\begin{equation}\label{eq:capillary-t-n}
\bn = -\tau X_u \times \bt + \sigma X_u,
\end{equation}
where $\sigma = \pm \sqrt{1 - \tau^2}$ and is constant on each connected component of $\partial \bB \setminus \cS$. Taking the cross product of~\eqref{eq:capillary-t-n} with $\bt$ on the right yields
\begin{equation}\label{eq:nu-t-n}
\nu = \tau X_u + \sigma X_u \times \bt,
\end{equation}
and from the two above identities we get
\[
X_u = \sigma \bn + \tau \nu.
\]
This together with~\eqref{eq:nu-t-n} implies that on $\partial \bB\setminus \cS$, the space $E \cap T_{u}\Sigma$ is spanned by $\bt$, and that
\begin{equation}\label{eq:tilde-t-n}
\widetilde{\nu} = \frac{\sigma}{|\sigma|} X_u \times \bt,\ \ \widetilde{X} = \frac{\sigma}{|\sigma|} \bn, \text{ on }\partial \bB \setminus \cS.
\end{equation}
Next we define the relevant bilinear forms. Recall from Section~\ref{subsec:second-variation} that
\[
\cV = \{\psi \in C^\infty(\overline{\bB}; \RR^3) \ |\ \psi(x) \in T_{u(x)}\Sigma, \text{ for all }x \in \partial \bB\}.
\]
In addition, we define
\[
\cE = \{\varphi \in \cV \ |\ \varphi(x) \in E_x \text{ for all }x \in \overline{\bB}\},
\]
and let $\cV_\cS$ and $\cE_\cS$ denote, respectively, those elements of $\cV$ and $\cE$ which are supported away from $\cS$. We then consider the following three bilinear forms:
\vskip 1mm
\begin{enumerate} 
\item
Given $\varphi \in \cV$, we define
\begin{equation}\label{eq:E-second-variation-H}
\begin{split}
(\delta^2 E_{H, \tau})_{u}(\varphi, \varphi) =\ & \int_{\bB} |\nabla \varphi|^2 + H \varphi \cdot \big( \varphi_{x^1} \times u_{x^2} + u_{x^1} \times \varphi_{x^2} \big)\ dx\\
& + \int_{\partial \bB} (u_r + \tau X_u \times u_\theta)\cdot A^\Sigma_u(\varphi, \varphi)\ d\theta + \tau \int_{\partial \bB} \varphi \cdot X_u \times \varphi_\theta\ d\theta.
\end{split}
\end{equation}
Note that this is just~\eqref{eq:second-variation-0} with the constant function $H$ in place of $f$.
\vskip 1mm
\item
For $\varphi \in \cV_{\cS}$, we define
\begin{equation}\label{eq:A-second-variation}
\begin{split}
(\delta^2 A_{H, \tau})_u(\varphi, \varphi) =\ & \int_{\bB} \Big[\big(\sum_{i = 1}^2|\varphi^{\perp}_{x^i}|^2\big) + \lambda^{-2}\big(\sum_{i = 1}^2\varphi_{x^i} \cdot u_{x^i}\big)^2 - \lambda^{-2}\sum_{i, j = 1}^2 (\varphi_{x^i}\cdot u_{x^j})(\varphi_{x^j}\cdot u_{x^i})\Big] dx\\
& + H\int_{\bB} \varphi \cdot \big( \varphi_{x^1} \times u_{x^2} + u_{x^1} \times \varphi_{x^2} \big)dx\\
& + \int_{\partial \bB} (u_r + \tau X_u \times u_\theta)\cdot A^\Sigma_u(\varphi, \varphi)  \  d\theta + \tau \int_{\partial \bB} \varphi \cdot X_u \times \varphi_\theta\  d\theta.
\end{split}
\end{equation}
This is essentially the second variation formula for the functional obtained from $E_{H, \tau}$ by replacing the term $\int_{\bB}\frac{|\nabla u|^2}{2} dx$ in~\eqref{eq:D-tau-definition} with the mapping area of $u$. We stated~\eqref{eq:A-second-variation} in the above form so that the formula is more easily compared with~\eqref{eq:E-second-variation-H} (see Lemma~\ref{lemm:index-comparison} below). For~\eqref{eq:A-second-variation} stated in a way that is perhaps more familiar, see~\eqref{eq:A-second-var-re} in Appendix~\ref{sec:second-variation}.
\vskip 1mm
\item For $f \in C^\infty(\overline{\bB})$, we define
\begin{equation}\label{eq:I-bilinear-form}
\begin{split}
(I_{H, \tau})_u(f, f) =\ & \int_{\bB} \big( |\nabla f|^2 -f^2 |\nabla \nu|^2 \big)\ dx\\
& + \int_{\partial \bB} \big[ \frac{\tau}{\sqrt{1- \tau^2}} (\nu_r \cdot \widetilde{X}) + \frac{u_r \cdot X_u}{1 - \tau^2}h^\Sigma(\widetilde{\nu}, \widetilde{\nu}) \big]  f^2  d\theta,
 \end{split}
\end{equation}
where our convention for the scalar-valued second fundamental form $h^{\Sigma}$ is 
\[
h_y^\Sigma(v, w) = \nabla_{v}w \cdot X_y = -w \cdot (\nabla X)_{y}(v) \text{ for }v, w \in T_y\Sigma.
\]
For the boundary terms in~\eqref{eq:I-bilinear-form} rewritten in a more standard form when $f$ is supported away from $\cS$, see~\eqref{eq:I-second-var-re} in Appendix~\ref{sec:second-variation}.
\end{enumerate}
To finish the list of notation, recall that the index of $(\delta^2 E_{H,\tau})_u$ on $\cV$ is denoted $\Ind^E_{H, \tau}(u)$. In addition, we denote the index of $(\delta^2 A_{H, \tau})_u$ on $\cV_{\cS}$ by $\Ind^A_{H, \tau}(u)$ and that of $(I_{H, \tau})_u$ on $C^\infty(\overline{\bB})$ by $\Ind^I_{H, \tau}(u)$. Again, by index we mean the supremum of $\dim\cL$ over all finite-dimensional subspaces $\cL$ of its domain on which the bilinear form restricts to be negative definite.
\begin{rmk}\label{rmk:index-statement}
The index in the statement of Theorem~\ref{thm:main-1} refers to $\Ind^I_{H, \tau}(u)$.
\end{rmk}
\subsection{Index comparison I}\label{subsec:index-comparison-I}
In this section we compare $\Ind^E_{H, \tau}$ with $\Ind^A_{H, \tau}$, the main result being Proposition~\ref{prop:A-E-index}, which rests upon the formula~\eqref{eq:index-comparison} established in Lemma~\ref{lemm:index-comparison} for the difference between $(\delta^2 E_{H, \tau})_u$ and $(\delta^2 A_{H, \tau})_u$. The identity~\eqref{eq:index-comparison} and its application to index comparison is inspired by the work of Ejiri-Micallef~\cite{Ejiri-Micallef2008} on the energy index and area index of closed minimal surfaces. Similar results for free boundary minimal surfaces and for closed CMC surfaces can be found respectively in~\cite{Lima2022} and~\cite{cz-cmc}. 

We begin with a preliminary lemma that allows us to reduce the proof of Lemma~\ref{lemm:index-comparison} to a simple pointwise calculation.
\begin{lemm}\label{lemm:area-kernel}
For all $\psi, \varphi \in C^\infty(\overline{\bB}\setminus \cS; \RR^3)$ such that $\varphi(x) \in E_x$ for all $x$, we have 
\begin{equation}\label{eq:area-kernel-formula}
\begin{split}
(\delta^2 A_{H, \tau})_{u}(\varphi, \psi) = \ & \int_{\partial \bB}(\varphi \cdot u_r) \Div_{E}\psi d\theta -\frac{H}{2}\int_{\partial \bB} \varphi \times \psi \cdot u_\theta d\theta\\
&  - \frac{\tau}{2}\int_{\partial\bB} (\nabla X)_u(u_\theta) \cdot \varphi \times \psi d\theta\\
& + \tau\int_{\partial \bB}(\varphi \cdot \bn)(X_u \times \bt \cdot \psi_r - X_u \times \bn \cdot \psi_\theta) d\theta.
\end{split}
\end{equation}
In particular, when $\psi \in \cV_{\cS}$ and $\varphi \in \cE_{\cS}$, we have $(\delta^2 A_{H, \tau})_{u}(\varphi, \psi) = 0$.
\end{lemm}
\begin{proof}
This is a standard result. For a proof in our notation, see Appendix~\ref{sec:second-variation}.
\end{proof}
\begin{lemm}\label{lemm:index-comparison}
Let $u: (\overline{\bB}, \partial \bB) \to (\RR^3, \Sigma)$ be a smooth, non-constant solution of~\eqref{eq:H-cmc-bvp}. For all $\varphi \in \cE$ and $\psi \in \cV_{\cS}$, we have 
\begin{equation}\label{eq:index-comparison}
(\delta^2 A_{H, \tau})_u (\psi, \psi) = (\delta^2 E_{H, \tau})_u(\psi + \varphi, \psi + \varphi) - 8\int_{\bB}|\nabla_z^T \varphi^{0, 1} + (\nabla_z\psi)^{0, 1}|^2\ dx.
\end{equation}
Note that since $\nabla^T$ preserves $E^{1, 0}$ and $E^{0, 1}$, we have $\nabla_{z}^T(\varphi^{0, 1}) = (\nabla_{z}^T\varphi)^{0, 1}$, and hence no ambiguity arises from writing $\nabla_z^T \varphi^{0, 1}$ in~\eqref{eq:index-comparison}.
\end{lemm}
\begin{proof}
By an approximation argument involving logarithmic cutoff functions, we need only prove~\eqref{eq:index-comparison} when $\varphi \in \cE_{\cS}$, in which case, by Lemma~\ref{lemm:area-kernel}, the left hand side of~\eqref{eq:index-comparison} does not change if we replace $\psi$ by $\psi + \varphi$. Also, since $|u_z|^2 = \frac{\lambda^2}{2}$, we have
\[
|\nabla_z^T \varphi^{0, 1} + (\nabla_z\psi)^{0, 1}|^2 = 2\lambda^{-2}|(\psi_z + \varphi_z) \cdot u_z|^2.
\]
To sum up, we have reduced our task to proving that 
\[
(\delta^2 E_{H, \tau})_u(\varphi, \varphi) = (\delta^2 A_{H, \tau})_u (\varphi, \varphi) + 16\int_{\bB}\lambda^{-2} |\varphi_z \cdot u_z|^2 dx, \text{ for all }\varphi \in \cV_{\cS}.
\]
Inspection of~\eqref{eq:E-second-variation-H} and~\eqref{eq:A-second-variation} shows that to prove this, it suffices to establish
\begin{equation}\label{eq:grad-relation}
\begin{split}
16\lambda^{-2}|\varphi_z \cdot u_z|^2
=\ &  \sum_{i = 1}^2 \big(|\varphi_{x^i}|^2 -|\varphi^{\perp}_{x^i}|^2\big)\\
&- \lambda^{-2}\sum_{i, j = 1}^2 (\varphi_{x^i} \cdot u_{x^i})(\varphi_{x^j} \cdot u_{x^j}) + \lambda^{-2}\sum_{i, j = 1}^2 (\varphi_{x^i}\cdot u_{x^j})(\varphi_{x^j}\cdot u_{x^i}).
\end{split}
\end{equation}
To that end we note that
\[
|\varphi_{x^i}|^2 - |\varphi_{x^i}^\perp|^2 = \lambda^{-2}\sum_{j = 1}^2 (\varphi_{x^i} \cdot u_{x^j})^2.
\]
Using this to replace the first term on the right hand side of~\eqref{eq:grad-relation}, we end up with three summations over $i, j$. Splitting each into terms with $i = j$ and those with $i \neq j$, we get that the right hand side of~\eqref{eq:grad-relation} is equal to
\[
\begin{split}
&\sum_{i}\lambda^{-2}(\varphi_{x^i} \cdot u_{x^i})^2 + \sum_{i \neq j} \lambda^{-2}(\varphi_{x^i}\cdot u_{x^j})^2\\
& - \sum_{i \neq j}\lambda^{-2} (\varphi_{x^i} \cdot u_{x^i})(\varphi_{x^j}\cdot u_{x^j}) + \sum_{i \neq j} \lambda^{-2}(\varphi_{x^i} \cdot u_{x^j})(\varphi_{x^j} \cdot u_{x^i})\\
=\ & \lambda^{-2}(\varphi_{x^1} \cdot u_{x^1} - \varphi_{x^2} \cdot u_{x^2})^2 + \lambda^{-2}(\varphi_{x^1} \cdot u_{x^2} + \varphi_{x^2} \cdot u_{x^1})^2,
\end{split}
\]
which we see is equal to the left hand side of~\eqref{eq:grad-relation} upon noticing that 
\[
\begin{split}
4(\varphi_z \cdot u_z) = (\varphi_{x^1} \cdot u_{x^1} - \varphi_{x^2} \cdot u_{x^2}) - \sqrt{-1}(\varphi_{x^1} \cdot u_{x^2} + \varphi_{x^2} \cdot u_{x^1}).
\end{split}
\]
\end{proof}
\begin{prop}\label{prop:A-E-index}
Let $u: (\overline{\bB}, \partial \bB) \to (\RR^3, \Sigma)$ be a smooth, non-constant solution of~\eqref{eq:H-cmc-bvp}. Then $\Ind^E_{H, \tau}(u) = \Ind^A_{H, \tau}(u)$.
\end{prop}
\begin{proof}
By the logarithmic cutoff trick, $\Ind^E_{H, \tau}(u)$ equals the index of $(\delta^2 E_{H, \tau})_u$ restricted to $\cV_{\cS}$. Applying Lemma~\ref{lemm:index-comparison} with $\varphi = 0$, we see that if $(\delta^2 E_{H, \tau})_u$ is negative definite on some finite-dimensional subspace of $\cV_{\cS}$, then so is $(\delta^2 A_{H, \tau})_u$. Thus 
\[
\Ind^E_{H, \tau}(u) \leq \Ind^A_{H, \tau}(u).
\] 
Conversely, suppose $\cL$ is a finite-dimensional subspace of $\cV_{\cS}$ on which $(\delta^2 A_{H, \tau})_u$ is negative definite, and let $\psi_1, \cdots, \psi_d$ be a basis of $\cL$. By the proof of~\cite[Proposition 7.1]{Cheng22}, in which the exact same boundary value problem as the one below is considered, we obtain for each $i \in \{1, \cdots, d\}$  a smooth section $\sigma_i$ of $E^{0, 1}$ so that
\[
\left\{
\begin{array}{l}
\nabla_{z}^T\sigma_i = - (\nabla_z \psi_i)^{0, 1} \text{ on }\bB,\\
\re (\sigma_i \cdot (X_u)^{1, 0}) = 0 \text{ on }\partial \bB.
\end{array}
\right.
\]
Letting $\varphi_i = \overline{\sigma_i} + \sigma_i$, we see that $\varphi_i \in \cE$ and $\sigma_i = \varphi_i^{0, 1}$. Lemma~\ref{lemm:index-comparison} then implies that $(\delta^2 E_{H, \tau})_u$ is negative definite on $\cL' = \Span\{\varphi_i + \psi_i\}_{1 \leq i \leq d}$, so that 
\begin{equation}\label{eq:negative-for-E}
\dim \cL' \leq \Ind_{H, \tau}^E(u).
\end{equation}
To see that $\dim \cL' = \dim \cL$, suppose there exist constants $a_1, \cdots, a_d \in \RR$ such that $\sum_{i = 1}^d a_i(\varphi_i + \psi_i) = 0$. Then we have
\[
\psi: = \sum_{i = 1}^d a_i \psi_i = -\sum_{i = 1}^da_i \varphi_i \in \cL \cap \cE = \cL \cap \cE_{\cS},
\]
where the last equality follows from $\cL \subset \cV_{\cS}$. But then by Lemma~\ref{lemm:area-kernel},
\[
(\delta^2 A_{H, \tau})_u(\psi, \psi) = 0,
\]
which forces $\psi$, and hence each $a_i$, to vanish since $(\delta^2 A_{H, \tau})_u$ is negative definite on $\cL$. Thus, $\psi_1 +\varphi_1, \cdots, \psi_d + \varphi_d$ are linearly independent, that is, $\dim \cL' = \dim\cL$. Recalling~\eqref{eq:negative-for-E}, we conclude that
\[
\Ind^A_{H, \tau}(u) \leq \Ind^E_{H, \tau}(u),
\] 
since $\cL$ is arbitrary.
\end{proof}
\subsection{Index comparison II}\label{subsec:index-comparison-II}
In this section we compare $\Ind^A_{H, \tau}(u)$ with $\Ind^I_{H, \tau}(u)$. The next two lemmas are due respectively to Ros-Souam~\cite{RosSouam1997} and Ainouz-Souam~\cite{Ainouz-Souam2016}. 
\begin{lemm}\label{lemm:A-second-variation-function}
Given $\varphi \in \cV_{\cS}$, define $f = \varphi \cdot \nu$. Then
\[
(\delta^2 A_{H, \tau})_u(\varphi, \varphi)  = (I_{H, \tau})_u(f, f).
\]
\end{lemm}
\begin{proof}
See~\cite[Section 4]{RosSouam1997}. The proof is also reproduced in our notation in Appendix~\ref{sec:second-variation} for the reader's convenience.
\end{proof}
\begin{lemm}\label{lemm:complement}
There exists a smooth section $\xi$ of $E$ on $\overline{\bB}$ such that $(\nu + \xi)|_x \in T_{u(x)}\Sigma$ for all $x \in \partial \bB$.
\end{lemm}
\begin{proof}
See~\cite[Proposition 4.1]{Ainouz-Souam2016}. The argument consists of taking advantage of the vector field $\widetilde{\nu}$ on $\partial \bB$ defined in~\eqref{eq:X-nu-tilde-def}, and goes essentially as follows. Recalling that $\nu(x)$ spans $E_x^\perp$, we have
\[
\big(\widetilde{\nu} - (\widetilde{\nu} \cdot \nu)\nu \big)\big|_{x} \in E_x, \text{ for all }x \in \partial \bB.
\]
Since $\nu \cdot \widetilde{\nu} \neq 0$ as can be seen from~\eqref{eq:X-nu-tilde-def}, we may define $\xi = \frac{\widetilde{\nu}}{\nu \cdot \widetilde{\nu}} - \nu$ on $\partial \bB$. The proof is complete upon extending $\xi$ to a section of $E$ on $\overline{\bB}$.
\end{proof}
\begin{prop}\label{prop:index-comparison-2}
Let $u: (\overline{\bB}, \partial \bB) \to (\RR^3, \Sigma)$ be a smooth, non-constant solution of~\eqref{eq:H-cmc-bvp}. Then $\Ind^A_{H, \tau}(u) = \Ind^I_{H, \tau}(u)$.
\end{prop}
\begin{proof}
Again, by the logarithmic cutoff trick, $\Ind^I_{H, \tau}(u)$ coincides with index of $(I_{H, \tau})_u$ over the smaller domain $C^{\infty}_c(\overline{\bB} \setminus \cS)$. Suppose $\cL$ is a finite-dimensional subspace of $\cV_{\cS}$ on which $(\delta^2 A_{H, \tau})_u$ is negative definite. Let $\psi_1, \cdots, \psi_d$ be a basis for $\cL$ and define 
\[
f_i = \psi_i \cdot \nu, \text{ for }i = 1, \cdots, d.
\]
Then each $f_i$ lies in $C^{\infty}_c(\overline{\bB} \setminus \cS)$, and $(I_{H, \tau})_u$ is negative definite on $\cL' = \Span\{f_1, \cdots, f_d\}$ by Lemma~\ref{lemm:A-second-variation-function}, which implies that 
\begin{equation}\label{eq:negative-for-I}
\dim \cL' \leq \Ind_{H, \tau}^I(u).
\end{equation}
If $a_1, \cdots, a_d \in \RR$ are such that $a_1f_1 + \cdots + a_d f_d =0$ identically, then the combination $\psi: = a_1 \psi_1 + \cdots + a_d \psi_d$ lies in $\cE_{\cS}$ in addition to being in $\cL$. Lemma~\ref{lemm:area-kernel} then gives
\[
(\delta^2 A_{H, \tau})_u(\psi, \psi) = 0,
\]
and that all $a_i$ consequently vanish. Thus $\dim\cL' = \dim\cL$, and we get from~\eqref{eq:negative-for-I} and the arbitrariness of $\cL$ that 
\[
\Ind^A_{H, \tau}(u) \leq \Ind^I_{H, \tau}(u).
\]
Conversely, suppose $f_1, \cdots, f_d$ span a $d$-dimensional subspace of $C^{\infty}_c(\overline{\bB} \setminus \cS)$ on which $(I_{H, \tau})_u$ is negative definite, and let 
\[
\psi_i = f_i (\nu + \xi),
\]
where $\xi$ is the section of $E$ given by Lemma~\ref{lemm:complement}. Then $\psi_i \in \cV_{\cS}$ and $\psi_i \cdot \nu = f_i$, so we have from Lemma~\ref{lemm:A-second-variation-function} that $(\delta^2 A_{H, \tau})_u$ is negative definite on $\Span\{\psi_1, \cdots, \psi_d\}$. This latter subspace of $\cV_{\cS}$ is $d$-dimensional since $\nu + \xi$ is nowhere vanishing on $\overline{\bB}$, and we deduce that 
\[
\Ind^I_{H, \tau}(u) \leq \Ind^A_{H, \tau}(u).
\]
 \end{proof}

\section{Existence of critical points}\label{sec:perturbed-existence}
In this section we return to the study of the perturbed functional $E_{\ep, p, f, \tau}$. The main result,  Proposition~\ref{prop:perturbed-existence}, produces non-constant critical points and is the starting point of the proof of Theorem~\ref{thm:main-1} in Section~\ref{sec:proof}. In Section~\ref{subsec:crit-theorey} we introduce a special class of paths in $\cM_p$ (Definition~\ref{defi:P-a-C}) and develop the tools (mainly Proposition~\ref{prop:mountain-pass-upgraded}) needed to extract non-trivial critical points with index bounds out of sequences of such paths. These sequences of paths are obtained in Section~\ref{subsec:extract} based on Lemma~\ref{lemm:D-difference}(a) and Struwe's monotonicity trick. Proposition~\ref{prop:perturbed-existence} appears at the end of Section~\ref{subsec:extract}.

Most of the proofs in this section are omitted or merely sketched, since the details are similar to the free boundary case addressed in~\cite{Cheng22}. 
\subsection{Some critical point theory}\label{subsec:crit-theorey}
Let $\cC$ denote the collection of constant maps from $\bB$ to $\Sigma$. Each continuous path $\gamma: ([0, 1], \{0, 1\}) \to (\cM_p, \cC)$ induces a continuous map $h_{\gamma}: S^2 \to \Sigma$, which has a well-defined topological degree. With this in mind, we define
\[
\cP = \{\gamma \in C^0([0, 1], \{0, 1\}; \cM_p, \cC)\ |\ \deg(h_\gamma) = 1\}.
\]
Since given $\gamma \in \cP$ we have $\gamma|_{[0, t]} \in \cE(\gamma(t))$ for all $t \in [0, 1]$, it makes sense to define
\[
\omega_{\ep, p, f, \tau} = \inf_{\gamma \in \cP} \Big[ \sup_{t \in [0, 1]} E_{\ep, p, f, \tau}(\gamma(t), \gamma|_{[0, t]}) \Big].
\]
\begin{defi}\label{defi:P-a-C}
For $\alpha, C_0 > 0$, we let $\cP_{\ep, p, f, \tau, \alpha, C_0}$ be the subset of $\cP$ consisting of paths $\gamma$ such that
\[
\sup_{t \in [0, 1]} E_{\ep, p, f, \tau}(\gamma(t), \gamma|_{[0, t]}) \leq \omega_{\ep, p, f, \tau} + \alpha,
\]
and that 
\[
D_{\ep, p}(\gamma(t)) \leq C_0, \text{ whenever $E_{\ep, p, f, \tau}(\gamma(t), \gamma|_{[0, t]}) \geq \omega_{\ep, p, f, \tau} - \alpha$.}
\]
\end{defi}
\begin{defi}\label{defi:K-C}
For $C_0 > 0$, we let $\cK_{\ep, p, f, \tau, C_0}$ denote the set of critical points $u \in \cM_p$ of $E_{\ep, p, f, \tau}$ such that $D_{\ep, p}(u) \leq C_0$ and that 
\[
E_{\ep, p, f, \tau}(u, \gamma ) = \omega_{\ep, p, f, \tau} \text{, for some $\gamma \in \cE(u)$.}
\]
By Proposition~\ref{prop:Palais-Smale} and Remark~\ref{rmk:volume-l}, the set $\cK_{\ep, p, f, \tau, C_0}$ is compact in $\cM_p$.
\end{defi}
The next lemma prevents critical points of $E_{\ep, p, f, \tau}$ at the level $\omega_{\ep, p, f, \tau}$ from being constant. Recall that $H_1$ denotes an upper bound for the $C^1$ norms of $f$ and $\nabla X$ on $\RR^3$.
\begin{lemm}\label{lemm:lower-bound}
There exist $\eta_2, \eta_3 \in (0, 1)$ depending only on $\ep, p, H_1, \tau$ such that for all $\gamma \in \cP$, if $t_0 \in [0, 1]$ is such that 
\[
D_{\ep, p}(\gamma(t_0)) < \eta_2,
\]
then 
\[
E_{\ep, p, f, \tau}(\gamma(t_0), \gamma|_{[0, t_0]}) < \sup_{t \in [0, 1]} E_{\ep, p, f, \tau}(\gamma(t), \gamma|_{[0, t]}) - \eta_3.
\]
\end{lemm}
\begin{proof}
By Sobolev embedding, we can find $K > 0$ depending only on $p, \ep$ such that 
\begin{equation}\label{eq:oscillation-e1}
\osc_{\bB}u \leq K \cdot (D_{\ep, p}(u))^{\frac{1}{p}}, \text{ for all }u \in \cM_p.
\end{equation}
Thus, if we choose $\alpha_0 = \alpha_0(\ep, p)$ so that 
\[
K\alpha_0^{\frac{1}{p}} < \frac{\delta_0}{2},
\] 
where $\delta_0$ is the constant introduced in Remark~\ref{rmk:h-geodesics}, then whenever $u \in \cM_p$ satisfies $D_{\ep, p}(u) < \alpha_0$, we may consider, in the notation of Remark~\ref{rmk:h-geodesics}, the path $l_u := l_{u(e_1), u}$, which satisfies, by~\eqref{eq:gradient-pointwise},
\begin{equation}\label{eq:contract-expansion}
D_{\ep, p}(l_u(t)) \leq C_{p}\cdot D_{\ep, p}(u) \text{ for all }t \in [0, 1],
\end{equation}
and, by Remark~\ref{rmk:volume-l},~\eqref{eq:oscillation-e1} and Sobolev embedding,
\begin{equation}\label{eq:weak-isoperimetric}
|V_{f}(l_u)| + |V_{\Div X}(l_u)| \leq C_{H_1, p, \ep}\cdot (D_{\ep, p}(u))^{1 + \frac{1}{p}}.
\end{equation}
Now, given $\gamma \in \cP$ we must have
\[
\sup_{t \in [0, 1]} D_{\ep, p}(\gamma(t)) \geq \alpha_0,
\]
for otherwise $\gamma$ is homotopic to a path through constant maps, violating the requirement that $\deg(h_\gamma) = 1$. Hence, with $\eta_2 < \alpha < \alpha_0$, where $\eta_2$ and $\alpha$ are to be determined, if $t_0$ is as in the statement of the lemma, then there exists $t_1 \neq t_0$ such that 
\[
D_{\ep, p}(\gamma(t_1)) = \alpha.
\]
We assume without loss of generality that $t_1 > t_0$, and that 
\begin{equation}\label{eq:small-D-path}
D_{\ep, p}(\gamma(t)) \leq \alpha \text{ for }t \in [t_0, t_1].
\end{equation}
Now consider the path $\widetilde{\gamma} = l_{\gamma(t_0)} + \gamma|_{[t_0, t_1]} + \big( -l_{\gamma(t_1)} \big)$. Then by~\eqref{eq:contract-expansion} there holds
\[
\sup_{t \in [0, 1]} D_{\ep, p}(\widetilde{\gamma}(t)) \leq (C_{p}+ 1) \cdot \alpha,
\]
so that if $(C_{p} + 1)\alpha < \alpha_0$, then $\widetilde{\gamma}$ is homotopic to a path through constant maps, implying that $V_{f}(\widetilde{\gamma}) = 0 = V_{\Div X}(\widetilde{\gamma})$. Consequently, by~\eqref{eq:weak-isoperimetric} and~\eqref{eq:small-D-path} we have
\[
\big| V_{f}(\gamma|_{[t_0, t_1]}) \big| + \big|  V_{\Div X}(\gamma|_{[t_0, t_1]}) \big| \leq C_{H_1, p, \ep}\alpha^{\frac{1}{p}} \cdot \big( D_{\ep, p}(\gamma(t_0)) + D_{\ep, p}(\gamma(t_1))  \big).
\]
Combining this with~\eqref{eq:D-p-estimates}, we have for some $C$ depending only on $\ep, p,  H_1$ that
\[
\begin{split}
&E_{\ep, p, f, \tau}(\gamma(t_1), \gamma|_{[0, t_1]}) - E_{\ep, p, f, \tau}(\gamma(t_0), \gamma|_{[0, t_0]})\\
\geq\ & [ (1 - |\tau|)^2 - C \alpha^{\frac{1}{p}}]\cdot D_{\ep, p}(\gamma(t_1)) - [4  + C \alpha^{\frac{1}{p}}]\cdot D_{\ep, p}(\gamma(t_0))\\
\geq\ &  [ (1 - |\tau|)^2 - C \alpha^{\frac{1}{p}}] \alpha - [ 4  + C \alpha^{\frac{1}{p}}]\eta_2,
\end{split}
\]
and we get the desired conclusion by first choosing $\alpha$ and then $\eta_2$ to be small enough.
\end{proof}
\begin{defi}\label{defi:target-reduction}
With $\eta_3$ as given by Lemma~\ref{lemm:lower-bound}, we define
\begin{equation}\label{eq:beta-0-defi}
\beta_0 = \left\{
\begin{array}{ll}
\min\big\{\frac{\eta_3}{4},  \frac{\big| \int_{\Omega} f + \tau\int_{\Omega}\Div X\big|}{4} \big\} & \text{ if } \int_{\Omega} f + \tau\int_{\Omega}\Div X \neq 0,\\
\\
\frac{\eta_3}{4} &\text{ otherwise. }
\end{array}
\right.
\end{equation}
Then we let $\cN$ be the set of maps $u \in \cM_p$ such that 
\[
|E_{\ep, p, f, \tau}(u, \gamma) -\omega_{\ep, p, f, \tau} | < \beta_0,\ \ \text{ for some }\gamma \in \cE(u).
\]
For $u \in \cN$, we choose $\gamma \in \cE(u)$ satisfying the above and define
\[
E^{\cN}_{\ep, p, f, \tau}(u) = E_{\ep, p, f, \tau}(u, \gamma).
\]
Note that, by Lemma~\ref{lemm:volume-properties}(b), for $\gamma, \widetilde{\gamma} \in \cE(u)$ we have
\[
E_{\ep, p, f, \tau}(u, \gamma) - E_{\ep, p, f, \tau}(u, \widetilde{\gamma}) = k \cdot\big( \int_{\Omega} f + \tau\int_{\Omega}\Div X \big), \text{ for some }k \in \ZZ.
\]
Hence $E^{\cN}_{\ep, p, f, \tau}$ is well-defined by our choice of $\beta_0$.
\end{defi}

\begin{lemm}\label{lemm:N-reduction}
The set $\cN$ is open in $\cM_p$. Moreover, $E^{\cN}_{\ep, p, f, \tau}$ is a $C^2$-functional on $\cN$.
\end{lemm}
\begin{proof}
Given $u_0 \in \cN$, choose $\gamma_0 \in \cE(u_0)$ so that 
\[
|E_{\ep, p, f, \tau}(u_0, \gamma_0) - \omega_{\ep, p, f, \tau}| < \beta_0,
\]
and let $\cA$ be a simply-connected neighborhood of $u_0$ in $\cM_p$. With $E^{\cA}_{\ep, p, f, \tau}$ being the local reduction induced by $(u_0, \gamma_0)$, we have
\[
E^{\cA}_{\ep, p, f, \tau}(u_0) = E_{\ep, p, f, \tau}(u_0, \gamma_0) = E^{\cN}_{\ep, p, f, \tau}(u_0).
\] 
Since $E^{\cA}_{\ep, p, f, \tau}$ is continuous, there exists a neighborhood $\cA'$ of $u_0$ in $\cA$ such that 
\[
|E^{\cA}_{\ep, p, f, \tau}(u) - \omega_{\ep, p, f, \tau}| < \beta_0, \text{ for all }u \in \cA'.
\]
This shows that $\cA'$ is contained in $\cN$ and that $E^{\cN}_{\ep, p, f, \tau}$ agrees with $E^{\cA}_{\ep, p, f, \tau}$ on $\cA'$, from which the assertions of the lemma follow.
\end{proof}
\begin{prop}\label{prop:mountain-pass}
Assume that for some sequence $(\alpha_k)$ converging to zero and some constant $C_0 > 0$, we have a sequence of paths $(\gamma_k)$ in $\cP$ such that 
\[
\gamma_k \in \cP_{\ep, p, f, \tau, \alpha_k, C_0}, \text{ for all } k.
\]
Then, passing to a subsequence of $(\gamma_k)$ if necessary, there exists a sequence $(t_k)$ in $[0, 1]$ such that 
\[
|E_{\ep, p, f, \tau}(\gamma_k(t_k), (\gamma_k)|_{[0, t_k]}) - \omega_{\ep, p, f, \tau}| \leq \alpha_k \text{ for all }k,
\]
and that $\gamma_{k}(t_k)$ converges strongly in $W^{1, p}$ to some $u \in \cK_{\ep, p, f, \tau, C_0}$ satisfying $D_{\ep, p}(u) \geq \eta_2$, where $\eta_2$ is given by Lemma~\ref{lemm:lower-bound}. 
\end{prop}
\begin{proof}
The proof is essentially the same as that of~\cite[Proposition 3.6]{cz-cgc}. We merely remark the following:
\begin{enumerate}
\item[(i)] Parts (a) and (b) of the conclusion of Proposition 3.3 in~\cite{cz-cgc}, which form part of the assumptions of Proposition 3.6 there, correspond to the two defining properties of the class $\cP_{\ep, p, f, \tau, \alpha_k, C_0}$ of this paper. 
\vskip 1mm
\item[(ii)] The role of the number $r$ appearing in condition $(\ast)$ in the proof of~\cite[Proposition 3.6]{cz-cgc} is now played by $\beta_0$ defined in~\eqref{eq:beta-0-defi} above.
\vskip 1mm
\item[(iii)] The set $\cK$ defined in the proof of~\cite[Propposition 3.6]{cz-cgc} corresponds to the set $\cK_{\ep, p, f, \tau, C_0}$ in Definition~\ref{defi:K-C} above.
\vskip 1mm
\item[(iv)] The Lemma 3.5, Proposition 2.13, Lemma 2.4(c), Lemma 2.1(a) and Lemma 3.4 invoked in the proof of~\cite[Proposition 3.6]{cz-cgc} correspond respectively to Lemma~\ref{lemm:N-reduction}, Proposition~\ref{prop:Palais-Smale}, Remark~\ref{rmk:volume-l}, Lemma~\ref{lemm:volume-properties}(b) and Lemma~\ref{lemm:lower-bound} of the present paper. Also, Lemma 2.4(a) of~\cite{cz-cgc} corresponds to the straightforward estimate $\|\nabla u\|_{2}^2 \leq 2D_{\ep, p}(u)$, and the estimate in Lemma 2.4(b), though not explicitly mentioned in the proof of~\cite[Proposition 3.6]{cz-cgc}, corresponds to~\eqref{eq:D-difference} above with $\tau = 0$.
\end{enumerate}
\end{proof}
The following proposition, basically an upgraded version of Proposition~\ref{prop:mountain-pass}, is the main result of this section.

\begin{prop}\label{prop:mountain-pass-upgraded}
Under the same assumptions as Proposition~\ref{prop:mountain-pass}, the perturbed functional $E_{\ep, p, f, \tau}$ has a critical point $u$ such that
\vskip 1mm
\begin{enumerate}
\item[(a)] $\eta_2 \leq D_{\ep, p}(u) \leq C_0 + 1$.
\vskip 1mm
\item[(b)] $E_{\ep, p, f, \tau}(u, \gamma) = \omega_{\ep, p, f, \tau}$ for some $\gamma \in \cE(u)$.
\vskip 1mm
\item[(c)] $\Ind_{\ep, p, f, \tau}(u) \leq 1$. (Recall Definition~\ref{defi:index}.)
\end{enumerate}
\end{prop}
\begin{proof}[Proof]
Define 
\[
\cK' = \{v \in \cK_{\ep, p, f, \tau, C_0+1}\ |\  \eta_2 \leq D_{\ep, p}(v) \},
\]
which is a compact subset of $\cK_{\ep, p, f, \tau, C_0 + 1}$ and is non-empty by Proposition~\ref{prop:mountain-pass} applied to the sequence $(\gamma_k)$. To prove the proposition, it suffices to produce an element of $\cK'$ with index at most $1$. Assume by contradiction that 
\[
\Ind_{\ep, p, f, \tau}(v) \geq 2, \text{ for all }v \in \cK'.
\]
We then follow the proof of~\cite[Proposition 5.5]{Cheng22} to obtain a new sequence $(\widetilde{\gamma}_k) \subset \cP$ out of $(\gamma_k)$ such that 
\begin{enumerate}
\item[(s1)] $\widetilde{\gamma}_k \in \cP_{\ep, p, f, \tau, \alpha_k, C_0 + 1}$ for all $k$.
\vskip 1mm
\item[(s2)] For all $v \in \cK'$, there exist $d_0 > 0$ and $k_0 \in \NN$ so that 
\[
\| \widetilde{\gamma}_k(t) - v \|_{1, p} \geq d_0,
\]
whenever $k \geq k_0$ and $E_{\ep, p, f, \tau}(\widetilde{\gamma}_k(t), \widetilde{\gamma}_k|_{[0, t]}) \geq \omega_{\ep, p, f, \tau} - \alpha_k$. 
\end{enumerate}
We refer the reader to~\cite{Cheng22} for details on the construction of $(\widetilde{\gamma}_k)$, and only note here that 
\vskip 1mm
\begin{enumerate}
\item[(i)] The constants $\eta_2$ and $\delta$ in Step 3 of the proof of~\cite[Proposition 5.5]{Cheng22} correspond respectively to $\eta_3$ from Lemma~\ref{lemm:lower-bound} and $\beta_0$ from~\eqref{eq:beta-0-defi}.
\vskip 1mm
\item[(ii)] The Proposition 4.3, Definition 4.1, Lemma 5.2, Lemma 2.3 and Corollary 2.2 invoked in the proof of~\cite[Proposition 5.5]{Cheng22} should be replaced, respectively, by Lemma~\ref{lemm:morse-nbd}, Definition~\ref{defi:index}, Lemma~\ref{lemm:lower-bound}, Lemma~\ref{lemm:volume-properties}(b) and Lemma~\ref{lemm:volume-properties}(a) of the present paper. 
\end{enumerate}
To derive a contradiction and finish the proof of the proposition, note that Proposition~\ref{prop:mountain-pass} applied to $(\widetilde{\gamma}_k)$ gives $(t_k)$ in $[0, 1]$ and $u \in \cK'$ such that 
\[
|E_{\ep, p, f, \tau}(\widetilde{\gamma}_k(t_k), \widetilde{\gamma}_k|_{[0, t_k]}) - \omega_{\ep, p, f, \tau}| \leq \alpha_k,
\]
and that, up to taking a subsequence, $\widetilde{\gamma}_k(t_k)$ converges strongly in $W^{1, p}$ to $u$. This is in contradiction with property (s2) above. Hence $\cK'$ must contain an element with index at most $1$. 
\end{proof}
\subsection{Existence of critical points with index bounds}\label{subsec:extract}
The purpose of this section is to produce sequences of sweepouts to which the results of the previous section can be applied. The key observation is that for all $\tau \in (-1, 1)\setminus \{0\}$, $u \in \cM_p$ and $\gamma \in \cE(u)$, and for all $0 < r < s < 1$, we have by Lemma~\ref{lemm:D-difference}(a) that
\begin{equation}\label{eq:monotonicity}
\frac{1}{s - r} \cdot\Big( \frac{E_{\ep, p, r  f, r  \tau}(u, \gamma)}{r} - \frac{E_{\ep, p, s  f, s \tau}(u, \gamma)}{s}  \Big) 
\geq  b(|\tau|) \cdot D_{\ep, p}(u) \geq 0,
\end{equation}
whether or not $\tau$ is positive.
\begin{lemm}\label{lemm:derivative-to-sweepout}
Suppose $\tau \in (-1, 1) \setminus \{0\}$ and that for some $r_0 \in (0, 1)$ and $c_0 > 0$ the derivative $\frac{d}{dr}\Big( -\frac{\omega_{\ep, p, r f, r\tau}}{r} \Big)\Big|_{r = r_0}$ exists and satisfies
\begin{equation}\label{eq:derivative-bound}
0 \leq \frac{d}{dr}\Big( -\frac{\omega_{\ep, p, r f, r\tau}}{r} \Big)\Big|_{r = r_0} \leq c_0.
\end{equation}
Then, letting $\alpha_k = \frac{r_0}{k}$ and $C_0 = \frac{8 c_0}{b(|\tau|)}$, the collection $\cP_{\ep, p, r_0 f, r_0  \tau, \alpha_k, C_0}$ is non-empty for sufficiently large $k$.
\end{lemm}
\begin{proof}
Let $r_k = r_0 - \frac{1}{4c_0 \cdot k}$. Then for sufficiently large $k$ there holds
\begin{equation}\label{eq:difference-quotient}
\frac{1}{r_0 - r_k}\Big(\frac{\omega_{\ep, p, r_k f, r_k \tau}}{r_k} - \frac{\omega_{\ep, p, r_0 f, r_0 \tau}}{r_0}\Big) \leq 2c_0.
\end{equation}
For all such $k$, choose a path $\gamma_k \in \cP$ such that 
\[
\sup_{t \in [0, 1]} E_{\ep, p, r_k f, r_k \tau}(\gamma_k(t), \gamma_k|_{[0, t]}) < \omega_{\ep, p, r_k f, r_k \tau} + \frac{r_k}{2k}.
\]
Then by~\eqref{eq:difference-quotient} and~\eqref{eq:monotonicity} we have
\begin{equation}\label{eq:mono-condition-1}
\begin{split}
\frac{E_{\ep, p, r_0f, r_0\tau}(\gamma_k(t), \gamma_k|_{[0, t]})}{r_0} \leq\ & \frac{E_{\ep, p, r_k f, r_k  \tau}(\gamma_k(t), \gamma_k|_{[0, t]})}{r_k}\\
 \leq\ & \frac{\omega_{\ep, p, r_k f, r_k \tau}}{r_k} + \frac{1}{2k} \leq  \frac{\omega_{\ep, p, r_0 f, r_0 \tau}}{r_0} + \frac{1}{k},
\end{split}
\end{equation}
for all $t \in [0, 1]$. This shows that $\gamma_k$ satisfies the first requirement in Definition~\ref{defi:P-a-C}. Next, suppose $t \in [0, 1]$ is such that 
\[
E_{\ep, p, r_0 f, r_0  \tau}(\gamma_k(t), \gamma_k|_{[0, t]}) \geq \omega_{\ep, p, r_0 f, r_0  \tau} - \frac{r_0}{k}.
\]
Then by~\eqref{eq:monotonicity} and the part of~\eqref{eq:mono-condition-1} after the first inequality, we have
\[
\begin{split}
b(|\tau|) \cdot D_{\ep, p}(\gamma_k(t)) =\ &\frac{1}{r_0 - r_k} \cdot \Big( \frac{E_{\ep, p, r_k f, r_k \tau}(\gamma_k(t), \gamma_k|_{[0, t]})}{r_k} - \frac{E_{\ep, p, r_0 f, r_0  \tau}(\gamma_k(t), \gamma_k|_{[0, t]})}{r_0}  \Big) \\
\leq\ & 8c_0.
\end{split}
\]
Therefore $\gamma_k \in \cP_{\ep, p, r_0 f, r_0 \tau, \frac{r_0}{k}, \frac{8c_0}{b(|\tau|)}}$, and the proof is complete.
\end{proof}

\begin{lemm}\label{lemm:monotonicity}
Given $\tau \in (-1, 1) \setminus \{0\}$, for almost every $r_0 \in (0, 1)$ there exists a sequence $(\ep_j)$ converging to zero and a constant $c_0 > 0$ such that 
\[
0 \leq \frac{d}{dr}\Big( -\frac{\omega_{\ep_j, p, r f, r\tau}}{r} \Big)\Big|_{r = r_0} \leq c_0 \text{ for all }j.
\]
\end{lemm}
\begin{proof}
Start with an arbitrary sequence $(\ep_j)$ converging to $0$. For all $j$, we have from~\eqref{eq:monotonicity} that the function
\[
r \mapsto  -\frac{\omega_{\ep_j, p, r f, r\tau}}{r}
\]
is non-decreasing on $(0, 1)$. The remainder of the proof is essentially the same as that of~\cite[Proposition 3.2(a)(c)]{cz-cgc}.
\end{proof}
Combining Lemma~\ref{lemm:monotonicity}, Lemma~\ref{lemm:derivative-to-sweepout} and Proposition~\ref{prop:mountain-pass-upgraded}, we immediately obtain the following existence result. The proof is straightforward and we omit it.
\begin{prop}\label{prop:perturbed-existence}
Given $\tau \in (-1, 1) \setminus \{0\}$, for almost every $r_0 \in (0, 1)$, there exist a sequence $(\ep_j)$ converging to zero and a constant $C_0 > 0$ so that for all $j$, the perturbed functional $E_{\ep_j, p, r_0 f, r_0  \tau}$ has a critical point $u_j$ with the following properties:
\begin{enumerate}
\item[(a)] $u_j$ is non-constant, and $D_{\ep_j, p}(u_j) \leq C_0$.
\vskip 1mm
\item[(b)] $\Ind_{\ep_j, p, r_0f, r_0\tau}(u_j) \leq 1$.
\end{enumerate}
\end{prop}

\section{Proof of the main theorem}\label{sec:proof}
From this point on in the paper we fix $p \in (2, \min\{p_0, p_1, p_2, p_3\}]$, with the $p_i$ coming respectively from Propositions~\ref{prop:interior-regularity},~\ref{prop:boundary-regularity},~\ref{prop:eta-regularity} and~\ref{prop:lower-bound}. Note that the threshold for $p$ depends only on $\tau$. 

We first reiterate Theorem~\ref{thm:main-1} with the meaning of the Morse index bound clarified in terms of the notation in Section~\ref{sec:index-comparison}. 
\begin{thm*}
Given $H \in [0, \underline{H'})$ and $\tau \in (-1, 1)$, for almost every $r \in (0, 1)$ there exists a smooth, non-constant, weakly conformal branched immersion $u: (\overline{\bB}, \partial \bB) \to (\overline{\Omega'}, \Sigma)$ solving~\eqref{eq:H-cmc-bvp} with $rH$ and $r\tau$ instead of $H$ and $\tau$. Moreover, in the notation of Section~\ref{subsec:index-notation}, we have $\Ind^I_{rH, r\tau}(u) \leq 1$.
\end{thm*}
This section is devoted to proving Theorem~\ref{thm:main-1}. As mentioned in Remark~\ref{rmk:tau}, we need only consider the case $\tau \neq 0$. We begin by fixing our choice of $X$ and $f$. With $H \in [0, \underline{H'})$ given, we choose $t_0$ to satisfy~\eqref{eq:t0-choice} and let $f: \RR^3 \to [0, H]$ be a compactly supported smooth function so that 
\begin{equation}\label{eq:f-choice}
f(y) = \left\{
\begin{array}{l}
H, \text{ if }y \in \overline{\Omega'_{\frac{t_0}{4}}},\\
0, \text{ if }y \not\in \Omega'_{\frac{3t_0}{4}}.
\end{array}
\right.
\end{equation}
Next, choose $d <\frac{t_0}{2}$ such that 
\[
\cV_{2d} := \{y \in \RR^3\ |\ \dist(y, \Sigma) < 2d\}
\]
is contained in the tubular neighborhood $\cW$ of $\Sigma$ fixed at the very beginning of Section~\ref{subsec:function-spaces}. With $\zeta$ being a cutoff function that equals $1$ on $\overline{\cV_{d}}$ and vanishes outside of $\cV_{\frac{3d}{2}}$, define $X: \RR^3 \to \RR^3$ by
\begin{equation}\label{eq:X-extension}
X(y) = \zeta(y)\bN(\Pi(y)).
\end{equation}
Note that $f$ and $X$ satisfy the requirements in Definition~\ref{defi:perturbed-definition}. Moreover, as both are supported in $\Omega'_{\frac{3t_0}{4}}$ and the range of $f$ is contained in $[0, H]$, Proposition~\ref{prop:maximum-principle} applies. 
 
Now let $r_0 \in (0, 1)$ be one of the almost every values given by Proposition~\ref{prop:perturbed-existence}. Then we get $C_0 > 0$ and sequences $(\ep_j)$ and $(u_j)$ satisfying conclusions (a) and (b) there. Propositions~\ref{prop:interior-regularity} and~\ref{prop:boundary-regularity} implies that each $u_j$ is smooth on $\overline{\bB}$. Furthermore, since $u_j$ is a non-constant critical point, we have by Proposition~\ref{prop:lower-bound} that
\begin{equation}\label{eq:uniform-lowerbound}
\int_{\bB} |\nabla u_j|^2 \geq \eta_1 \text{ for all }j,
\end{equation}
where we emphasize that $\eta_1$ is independent of $j$. Also, Proposition~\ref{prop:maximum-principle}(a) gives
\begin{equation}\label{eq:L-infinity-bound}
u_j(\overline{\bB})\subset \overline{\Omega'_{t_0}} \text{ for all }j.
\end{equation}
Next, by the energy bound in Proposition~\ref{prop:perturbed-existence}(a), the gradient estimates in Proposition~\ref{prop:eta-regularity}, and the $L^{\infty}$ bound provided by~\eqref{eq:L-infinity-bound}, we see via a standard argument that there exist a finite set $\cS \subset \overline{\bB}$ and a $C^1$-map $u: \overline{\bB}\setminus \cS \to \overline{\Omega'_{t_0}}$ so that, up to taking a subsequence of $(u_j)$, the following hold.
\begin{enumerate}
\item[(i)] $u_j$ converges in $C^1_{\loc}(\overline{\bB}\setminus \cS)$ to $u$.
\vskip 1mm
\item[(ii)] Letting $d = \frac{1}{2}\min_{x, x' \in \cS, x \neq x'} |x - x'|$, to each $x_0 \in \cS$ we may associate a sequence $(x_j)$ in $\overline{\bB}$ converging to $x_0$, and a sequence $(t_j)$ converging to zero, such that 
\begin{equation}\label{eq:x-t-choice}
\int_{\bB_{t_j}(x) \cap \bB}|\nabla u_j|^2 \leq \int_{\bB_{t_j}(x_j)\cap \bB} |\nabla u_j|^2 = \frac{\eta_0}{3}, \text{ for all $x \in \overline{\bB_d(x_0)} \cap \overline{\bB}$,}
\end{equation}
where $\eta_0$ is given by Proposition~\ref{prop:eta-regularity}.
\end{enumerate}
\begin{prop}\label{prop:scale-comparison}
Suppose $\cS \neq \emptyset$. For all $x_0 \in \cS$, with $x_j, t_j$ as in (ii) above, we have
\begin{enumerate}
\item[(a)] $\lim_{j \to \infty} \frac{\ep_j}{t_j} = 0$.
\vskip 1mm
\item[(b)] $\limsup_{j \to \infty} \frac{1 - |x_j|}{t_j} < \infty$. In particular, $\cS \subset \partial \bB$.
\end{enumerate}
\end{prop}
\begin{proof}[Proof of Proposition~\ref{prop:scale-comparison}(a)]
We argue by contradiction. Negating part (a) yields some $\alpha > 0$ and a subsequence of $j$'s, which we do not relabel, so that 
\[
\frac{\ep_j}{t_j} \geq \alpha \text{ for all }j.
\]
Next, we let
\[
v_j(y) = u_j(x_j + \ep_j y), \text{ for all }y \in \bB_j' := \bB_{\frac{1}{\ep_j}}(-\frac{x_j}{\ep_j}),
\]
and assume that 
\begin{equation}\label{eq:not-blown-away}
\limsup_{j \to \infty} \frac{1 - |x_j|}{\ep_j} < \infty,
\end{equation}
since the case where $\limsup_{j \to \infty} \frac{1 - |x_j|}{\ep_j}  = \infty$ can be treated similarly, and is in fact simpler. Then, as in the proof of~\cite[Proposition 6.3]{Cheng22}, for each $j$ we can find:
\begin{enumerate}
\item[(i)] an isometry $I_j$ of $\RR^2$ that takes $\frac{1}{\ep_j}\big( \frac{x_j}{|x_j|} - x_j \big)$ to the origin and maps $\bB_j'$ to $\bB_j'' := \bB_{\frac{1}{\ep_j}}(\frac{1}{\ep_j}e_2)$;
\vskip 1mm
\item[(ii)] a conformal map $F_j: (\RR^2_+, \partial \RR^2_+) \to (\bB_j'', \partial\bB_j''\setminus \{\frac{2}{\ep_j}e_2\})$ that converges to the identity map in $C^2$ on compact subsets of $\overline{\RR^2_+}$.
\end{enumerate}
By~\eqref{eq:not-blown-away}, the sequence of points $(I_j(0))$ is bounded. Also, letting $\lambda_j = |(F_j)_z|$, we have by (ii) that $(\lambda_j)$ converges to $1$ in $C^{1}_{\loc}(\overline{\RR^2_+})$. We then consider
\[
w_j := v_j \circ I_j^{-1} \circ F_j.
\]
To obtain estimates to pass to the limit as $j \to \infty$, we observe first that Proposition~\ref{prop:perturbed-existence}(a) implies~\eqref{eq:W1p-scale-invariant-bound} with $r = \ep_j$ and $L = L(p, C_0)$. Combining this with Remark~\ref{rmk:quantitative-W22} and Proposition~\ref{prop:W14-W2q} (with $q = 4$), as well as the fact that each $w_j$ maps into $\overline{\Omega'_{t_0}}$, we see that a subsequence of $(w_j)$ converges in $C^{1}_{\loc}(\overline{\RR^2_+})$ to a limit $w$. Since $(I_j(0))$ is bounded and $(F_j)$ converges locally to the identity map, we see with the help of the equality in~\eqref{eq:x-t-choice} and the lower bound $\frac{\ep_j}{\alpha} \geq t_j$ that $w$ is non-constant. Also, by Proposition~\ref{prop:perturbed-existence}(a) we have
\begin{equation}\label{eq:w-p-finite}
\int_{\RR^2_+} |\nabla w|^2 + |\nabla w|^p < \infty.
\end{equation}
Following the proof of~\cite[Proposition 6.3]{Cheng22}, which is ultimately inspired by the work of Lamm~\cite{Lamm2006}, we next derive a Pohozaev type identity for $w$ that forces it to be constant, thereby getting a contradiction. 

The computation is performed on $w_j$ and then passed to the limit. To save space, let us fix $j$ and write $w$, $\lambda$, $\ep$ for $w_j$, $\lambda_j$, $\ep_j$, respectively, and also omit the parameter $r_0$. Then we have $w(\partial\RR^2_+) \subset \Sigma$ and, from~\eqref{eq:first-variation}, that
\begin{equation}\label{eq:w-j-PDE}
\begin{split}
0=\ & \int_{\RR^2_+}[1 + (\ep^2 + 2\lambda^{-2}F_\tau(w, \nabla w))^{\frac{p}{2} - 1}] \big( \bangle{\nabla w, \nabla \psi} - \tau X(w)\cdot (\psi_{x^1} \times w_{x^2} + w_{x^1} \times \psi_{x^2})  \big)\\
& - \tau\int_{\RR^2_+} [1 + (\ep^2 + 2\lambda^{-2}F_\tau(w, \nabla w))^{\frac{p}{2} - 1}] (\nabla X)_{w}(\psi) \cdot w_{x^1} \times w_{x^2} \\
&+ \int_{\RR^2_+}  h(w) \psi \cdot w_{x^1} \times w_{x^2},
\end{split}
\end{equation}
whenever $\psi \in C^{1}_{c}(\overline{\RR^2_+})$ is such that $\psi(x) \in T_{w(x)}\Sigma$ for all $x \in \partial \RR^2_+$. Next, given $R > 0$, we substitute 
\[
\psi(x) =\zeta^2 x^k w_{x^k}
\]
into~\eqref{eq:w-j-PDE}, where $\zeta$ is a cutoff function that equals $1$ on $\bB_R$, vanishes outside of $\bB_{2R}$, and satisfies that $|\nabla \zeta| \leq CR^{-1}$ with $C$ independent of $R$. Note that this choice of $\psi$ is admissible since $w$ is smooth and $w_{x^1} \in T_{w}\Sigma$ on $\partial\RR^2_+$. Using the following identities:
\vskip 1mm
\begin{enumerate}
\item $(x^k w_{x^k}) \cdot w_{x^1} \times w_{x^2} = 0$,
\vskip 1mm
\item $\bangle{\nabla w, \nabla (x^k w_{x^k})} = |\nabla w|^2 + x^k\partial_k\big( \frac{|\nabla w|^2}{2} \big)$,
\vskip 1mm
\item $ (x^kw_{x^k})_{x^1} \times w_{x^2} + w_{x^1} \times (x^k w_{x^k})_{x^2} =  2 w_{x^1} \times w_{x^2} + x^k \partial_k(w_{x^1} \times w_{x^2})$,
\vskip 1mm
\item $(\nabla X)_w (x^k  w_{x^k}) = x^k\partial_k (X(w))$,
\vskip 1mm
\item $(\ep^2 + 2\lambda^{-2}F_\tau)^{\frac{p}{2} - 1} \partial_k F_\tau = \lambda^{2}\partial_k \big[ \frac{1}{p}(\ep^2 + 2\lambda^{-2}F_\tau)^{\frac{p}{2}} \big] + 2 \frac{\partial_k \lambda}{\lambda} (\ep^2 + 2\lambda^{-2}F_\tau)^{\frac{p}{2} - 1}  F_\tau$,
\end{enumerate}
we obtain
\begin{equation}\label{eq:w-j-pohozaev}
\begin{split}
0 =\ & \int_{\RR^2_+} [1 + (\ep^2 + 2\lambda^{-2}F_{\tau})^{\frac{p}{2} - 1}] (2\zeta^2 F_{\tau}) + \int_{\RR^2_+}\zeta^2 x^k \partial_k F_{\tau} +  \zeta^2 \lambda^2 x^k \partial_k \big[ \frac{1}{p}(\ep^2 + 2\lambda^{-2}F_{\tau})^{\frac{p}{2}} \big]\\
& + \int_{\RR^2_+} [1 + (\ep^2 + 2\lambda^{-2}F_{\tau})^{\frac{p}{2} - 1}] \bangle{\nabla (\zeta^2), \nabla w} \cdot (x^k  w_{x^k}) \\
& - \tau \int_{\RR^2_+}[1 + (\ep^2 + 2\lambda^{-2}F_{\tau})^{\frac{p}{2} - 1}] x^k(\zeta^2)_{x^k}   X(w) \cdot w_{x^1} \times w_{x^2}\\
& + 2\int_{\RR^2_+} \frac{\zeta^2 x^k\partial_k\lambda}{\lambda} \cdot (\ep^2 + 2\lambda^{-2}F_{\tau})^{\frac{p}{2}- 1}F_\tau,
\end{split}
\end{equation}
Integrating by parts in the second term on the right hand side gives
\[
\begin{split}
& \int_{\RR^2_+} \zeta^2 x^k \partial_k F_{\tau} +  \zeta^2 \lambda^2 x^k \partial_k \big[ \frac{1}{p}(\ep^2 + 2\lambda^{-2}F_{\tau})^{\frac{p}{2}} \big]\\
 =\ &-2 \int_{\RR^2_+} \zeta^2 \big(  F_{\tau} +  \frac{\lambda^2}{p}(\ep^2 + 2\lambda^{-2}F_\tau)^{\frac{p}{2}} \big)\\
 & - \int_{\RR^2_+} x^k \partial_k (\zeta^2) F_\tau + x^k \partial_k (\zeta^2 \lambda^2)\frac{1}{p}(\ep^2 + 2\lambda^{-2}F_\tau)^{\frac{p}{2}}
 \end{split}
\]
Substituting the above into~\eqref{eq:w-j-pohozaev} and letting $j$ tend to infinity, while recalling that $\lambda_j$ converges to the constant function $1$ in $C^1_{\loc}(\overline{\RR^2_+})$, we get
\[
\begin{split}
\frac{p-2}{p}\int_{\bB^+_R}  (2F_\tau(w, \nabla w))^{\frac{p}{2}} \leq\ &  C\int_{\RR^2_+} |x||\nabla \zeta | \cdot \big[ (1 + (2F_\tau)^{\frac{p}{2} - 1})|\nabla w|^2   +  (F_\tau + (2F_\tau)^{\frac{p}{2}})\big]\\
\leq\  & C\int_{\bB^+_{2R} \setminus \bB^+_{R}} \big[ (1 + (2F_\tau)^{\frac{p}{2} - 1})|\nabla w|^2   +  (F_\tau + (2F_\tau)^{\frac{p}{2}})\big],
\end{split}
\]
with $C$ independent of $R$. Letting $R$ tend to infinity and using~\eqref{eq:w-p-finite}, we get
\[
\frac{p-2}{p}\int_{\RR^2_+}  (2F_\tau(w, \nabla w))^{\frac{p}{2}} = 0,
\]
which forces $w$ to be constant by~\eqref{eq:F-estimates}, a contradiction. 
\end{proof}
\begin{proof}[Proof of Proposition~\ref{prop:scale-comparison}(b)]
We again argue by contradiction. Suppose that, after passing to a subsequence, we have 
\begin{equation}\label{eq:scale-b-negation}
\lim_{j \to \infty} \frac{1 - |x_j|}{t_j} = \infty.
\end{equation}
Define 
\[
v_j(y) = u_j(x_j + t_j y).
\]
Then by~\eqref{eq:scale-b-negation},~\eqref{eq:x-t-choice}, Proposition~\ref{prop:eta-regularity}, and the fact that each $v_j$ maps into $\overline{\Omega'_{t_0}}$, we see that up to taking a subsequence, $(v_j)$ converges in $C^{1}_{\loc}(\RR^2)$ to a limit $v$ which is non-constant by the equality in~\eqref{eq:x-t-choice}. Since $\frac{\ep_j}{t_j} \to 0$ by part (a), we see by rescaling in~\eqref{eq:first-variation} and using the $C^1$ convergence of $(v_j)$ that 
\[
\begin{split}
0 =\ & \int_{\RR^2}  \bangle{\nabla v, \nabla \varphi} - r_0\tau  X(v) \cdot (\varphi_{x^1} \times v_{x^2} + v_{x^1} \times \varphi_{x^2}) \\
& - r_0\tau \int_{\RR^2} (\nabla X)_v(\varphi) \cdot v_{x^1} \times v_{x^2}  + \int_{\RR^2} r_0 h(u)\varphi \cdot v_{x^1} \times v_{x^2},
\end{split}
\]
for all $\varphi \in C^1_c(\RR^2; \RR^3)$. Remark~\ref{rmk:C0-regularity} gives that $v$ is smooth and satisfies
\[
\Delta v = r_0 f(v) v_{x^1} \times v_{x^2} \text{ on }\RR^2.
\]
Also, by Proposition~\ref{prop:perturbed-existence}(a) we have
\[
\int_{\RR^2} |\nabla v|^2 < \infty.
\]
When $H = 0$, in which case $f \equiv 0$, Liouville's theorem forces $v$ to be constant, a contradiction. When $H \neq 0$, we pull $v$ back to $S^2$ via the stereographic projection and obtain a contradiction using a classification result of Brezis-Coron~\cite[Lemma A.1]{Brezis-Coron1984} and the assumption that $H < \underline{H'}$, as done in the proof of~\cite[Proposition 6.4]{Cheng22}.
\end{proof}
\noindent\textbf{Conclusion of the proof of Theorem~\ref{thm:main-1}:}
\vskip 2mm
\noindent\textit{Case 1: $\cS = \emptyset$.}
\vskip 1mm
In this case the limiting map $u$ satisfies the condition in Remark~\ref{rmk:C0-regularity}. Indeed, since~\eqref{eq:EL-bc} holds with $u_j$ in place of $u$, the $C^1$-convergence of $(u_j)$ gives
\begin{equation}\label{eq:limit-perp}
u_r + r_0\tau X(u) \times u_\theta \perp T_u \Sigma \text{ on }\partial \bB.
\end{equation}
Now, given a test function $\varphi$ as in Remark~\ref{rmk:C0-regularity}, we recall that~\eqref{eq:EL} also holds with $u_j$ in place of $u$, and hence by an integration by parts we get
\[
\begin{split}
&\text{(Right hand side of~\eqref{eq:first-variation} with $u_j$ in place of $u$) }\\
= \ & \int_{\partial \bB} [1 + \ep_j^{p-2}(1 + 2F_{r_0\tau}(u_j, \nabla u_j))^{\frac{p}{2} - 1}] (u_{j, r} 
 + r_0\tau X(u_j) \times u_{j, \theta}) \cdot \varphi d\theta,
\end{split}
\]
and~\eqref{eq:weak-EL} follows upon passing to the limit and using~\eqref{eq:limit-perp}. We then conclude from Remark~\ref{rmk:C0-regularity} that $u$ is a smooth solution to
\[
\left\{
\begin{array}{l}
\Delta u = r_0f(u)u_{x^1} \times u_{x^2} \text{ in }\bB,\\
u_r + r_0\tau X(u) \times u_{\theta} \perp T_u\Sigma \text{ on }\partial \bB.
\end{array}
\right.
\]
Lemma~\ref{lemm:weakly-conformal} implies that $u$ is weakly conformal, while the lower bound~\eqref{eq:uniform-lowerbound} passes to the limit to give that $u$ is non-constant. We then get from Corollary~\ref{coro:branch} that $u$ is a branched immersion. On the other hand, Proposition~\ref{prop:maximum-principle}(b) shows that $u$ maps into $\overline{\Omega'}$, so that, by our choice of $f$, the map $u$ in fact is a solution of~\eqref{eq:H-cmc-bvp} with $r_0 H$ and $r_0 \tau$ in place of $H$ and $\tau$, respectively. 

To obtain the index bound, note that since $u_j \to u$ in $C^1(\overline{\bB})$ and since $\ep_j \to 0$, it is straightforward to show, with the help of~\eqref{eq:second-variation}, that the index bound in Proposition~\ref{prop:perturbed-existence}(b) passes to the limit to give $\Ind^E_{r_0f, r_0\tau}(u) \leq 1$, where $\Ind^E_{r_0 f, r_0 \tau}(u)$ is defined just before Lemma~\ref{lemm:second-variation-0}. (We refer to Case 1 in the proof of~\cite[Theorem 1.1]{Cheng22}, found in Section 6 of that paper, for details on this passage.) Since the image of $u$ is contained in $\overline{\Omega'}$ and since $f = H$ on a neighborhood of the latter, we infer upon recalling Lemma~\ref{lemm:second-variation-0} that $(\delta^2 E_{r_0f, r_0\tau})_u$ coincides with $(\delta^2 E_{r_0H, r_0\tau})_u$, the latter given by~\eqref{eq:E-second-variation-H}. We then use Proposition~\ref{prop:A-E-index} and Proposition~\ref{prop:index-comparison-2} to transfer the index bound to $(I_{r_0H, r_0\tau})_u$, thereby finishing the proof in the case $\cS = \emptyset$.
\vskip 1em
\noindent\textit{Case 2: $\cS \neq \emptyset$.}
\vskip 1mm
In this case, we take $x_0 \in \cS$ and let $(x_j), (t_j)$ be as in item (ii) above Proposition~\ref{prop:scale-comparison}. Also, we define
\[
v_j(y) = u_j(x_j + t_j y), \text{ for all }y \in \bB_j' : = \bB_{\frac{1}{t_j}}(-\frac{x_j}{t_j}),
\]
and let $I_j, F_j, \lambda_j$ and $w_j$ be as in the proof of Proposition~\ref{prop:scale-comparison}(a), with $t_j$ in place of $\ep_j$. Then we can argue as in Case 2 in the proof of~\cite[Theorem 1.1]{Cheng22}, replacing the Proposition 3.6, Proposition 6.3 and Proposition 6.4 there by Proposition~\ref{prop:eta-regularity}, Proposition~\ref{prop:scale-comparison}(a) and Proposition~\ref{prop:scale-comparison}(b) from this paper, to obtain a subsequence of $(w_j)$ that converges in $C^{1}_{\loc}(\overline{\RR^2_+})$ to a non-constant limiting map $w:(\overline{\RR^2_+}, \partial\RR^2_{+}) \to (\overline{\Omega'_{t_0}}, \Sigma)$ which by Remark~\ref{rmk:C0-regularity} is smooth and satisfies
\[
\left\{
\begin{array}{l}
\int_{\RR^2_+}|\nabla w|^2 < \infty,\\
\Delta w = r_0 f(w) w_{x^1} \times w_{x^2} \text{ on }\RR^2_+,\\
w_{x^2} - r_0\tau X(w) \times w_{x^1} \perp T_{w}\Sigma \text{ on }\partial \RR^2_{+}.
\end{array}
\right.
\]
Letting $G$ be a conformal map from $(\bB, \partial\bB\setminus \{e_1\})$ to $(\RR^2_+, \partial \RR^2_+)$, by Corollary~\ref{coro:removal-singularity} we see that $\widetilde{w} : = w \circ G$ extends smoothly to $\overline{\bB}$. Lemma~\ref{lemm:weakly-conformal} and Corollary~\ref{coro:branch} then apply to show that $u$ is a weakly conformal branched immersion. Also, by Proposition~\ref{prop:maximum-principle}(b), $\widetilde{w}$ maps into $\overline{\Omega'}$, and hence is a solution of~\eqref{eq:H-cmc-bvp} with $r_0H$ and $r_0\tau$ replacing $H$ and $\tau$. Finally, following the proof of Case 2 of~\cite[Theorem 1.1]{Cheng22}, and again using Lemma~\ref{lemm:second-variation-0} along with the fact that $f = H$ on a neighborhood of $\widetilde{w}(\overline{\bB})$, we obtain the index bound $\Ind^E_{r_0H, r_0\tau}(\widetilde{w}) \leq 1$ from Proposition~\ref{prop:perturbed-existence}(b). We then conclude the proof using Proposition~\ref{prop:A-E-index} and Proposition~\ref{prop:index-comparison-2} as in the previous case.
\appendix
\section{Properties of the perturbed functional}\label{sec:differentiability}
\begin{lemm}\label{lemm:F-tau-derivative}
There exists a constant $b = b(\tau_0) > 0$ depending only on $\tau_0 \in (0, 1)$ such that for all $\tau \in [-\tau_0, \tau_0] \setminus \{0\}$, we have
\[
- \paop{\tau} \Big( \frac{F_{\tau}(y, \xi) + \frac{\ep^{p-2}}{p}G_\tau(y, \xi)}{\tau}\Big) \geq \frac{b(\tau_0)}{\tau^2} \cdot \big[ \frac{|\xi|^2}{2} + \frac{\ep^{p-2}}{p}(1 + |\xi|^2)^{\frac{p}{2}} \big].
\]
\end{lemm}
\begin{proof}
Differentiating with respect to $\tau$, we have by direct computation that 
\[
\begin{split}
&-\tau^2 \cdot \paop{\tau} \Big( \frac{F_{\tau}(y, \xi) + \frac{\ep^{p-2}}{p}G_\tau(y, \xi)}{\tau}\Big)\\
 = \ & \frac{|\xi|^2}{2} + \frac{\ep^{p-2}}{p}(1 + 2F_\tau(y, \xi))^{\frac{p}{2} - 1} \Big( 1 + |\xi|^2 + \tau \cdot \big( \frac{p}{2} - 1 \big) \mu_{ijk}X^k(y) \xi^i_\alpha \xi^j_\beta\Big)\\
\geq\ &  \frac{|\xi|^2}{2} + \frac{\ep^{p-2}}{p}(1 + 2F_\tau(y, \xi))^{\frac{p}{2} - 1} \Big( 1 + \big( 1 - \big|\frac{p}{2} - 1\big| \big) |\xi|^2 \Big).
\end{split}
\]
We then use the lower bound in~\eqref{eq:F-estimates}, recall that $p \in (2, 3]$, and divide both sides by the positive number $\tau^2$ to obtain the desired inequality with $b(\tau_0) = \frac{1 - \tau_0}{2}$.
\end{proof}
\subsection{Proof of Lemma~\ref{lemm:D-difference}(a)}
By Lemma~\ref{lemm:F-tau-derivative} we have
\[
- \paop{\tau}\Big(\frac{D_{\ep, p, \tau}(u)}{\tau}\Big) \geq \frac{b(\tau_0)}{\tau^2}\cdot D_{\ep, p}(u),
\]
for all $\tau \in [-\tau_0, \tau_0]\setminus \{0\}$. Suppose $\tau_1, \tau_2 \in [-\tau_0,\tau_0]$ are such that $\tau_1 < \tau_2$  and $\tau_1 \cdot \tau_2 > 0$, and integrating the above inequality from $\tau = \tau_1$ to $\tau = \tau_2$, we obtain
\[
\frac{1}{\tau_2 - \tau_1}\Big(\frac{D_{\ep, p, \tau_1}(u)}{\tau_1} - \frac{D_{\ep, p, \tau_2}(u)}{\tau_2}\Big) \geq \frac{b(\tau_0)}{\tau_1 \cdot \tau_2} \cdot D_{\ep, p}(u).
\]
Noting that $\tau_1 \cdot \tau_2 \leq \tau_0^2$, we obtain~\eqref{eq:D-tau-difference}.
\subsection{Proof of Lemma~\ref{lemm:D-difference}(b)}
Using the relation
\[
d\Big( F + \frac{\ep^{p-2}}{p}G \Big) = [1 + \ep^{p-2}(1 + 2F)^{\frac{p}{2} - 1}]d F
\]
along with~\eqref{eq:F-estimates}, we see that 
\[
\begin{split}
&\big| (F + \frac{\ep^{p-2}}{p}G)(z, \xi) - (F + \frac{\ep^{p-2}}{p}G)(w, \eta) \big| \\
\leq\ & C[(1 + |\xi|^2 + |\eta|^2)^{\frac{1}{2}} + \ep^{p-2}(1 + |\xi|^2 + |\eta|^2)^{\frac{p}{2} - \frac{1}{2}}] |\xi - \eta|\\
& + C[(1 + |\xi|^2 + |\eta|^2) + \ep^{p-2}(1 + |\xi|^2 + |\eta|^2)^{\frac{p}{2}}] |z - w|,
\end{split}
\]
where both constants $C$ depend only on the $C^1$ norm of $X$. H\"older's inequality then gives
\[
\begin{split}
\big| D_{\ep, p, \tau}(u) - D_{\ep, p, \tau}(v) \big| \leq\  & C\Big( \int_{\bB} 1 + |\nabla u|^2 + |\nabla v|^2 \Big)^{\frac{1}{2}} \|\nabla u - \nabla v\|_{2; \bB} \\
& + C\ep^{p-2} \cdot \Big( \int_{\bB} (1 + |\nabla u|^2 + |\nabla v|^2)^{\frac{p}{2}} \Big)^{\frac{p-1}{p}} \|\nabla u - \nabla v\|_{p; \bB}\\
& + C \Big( \int_{\bB} 1 + |\nabla u|^2 + |\nabla v|^2 \Big) \|u - v\|_{\infty; \bB}\\
& + C\ep^{p-2} \cdot \Big(\int_{\bB} (1 + |\nabla u|^2 + |\nabla v|^2 )^{\frac{p}{2}} \Big) \|u - v\|_{\infty; \bB},
\end{split}
\]
from which we deduce~\eqref{eq:D-difference} with the help of Sobolev embedding.
\section{Structure of the Euler-Lagrange equation}\label{sec:structure}
\subsection{Proof of Lemma~\ref{lemm:first-variation-simplified}}
Given $\psi \in W^{1, p}(\bB_{\sigma_1}^+;\RR^3)$ such that $\psi = 0$ on $\partial \bB_{\sigma_1} \cap \RR^2_+$ and $\psi^3 = 0$ on $\bT$, we first define 
\[
\widehat{\varphi}(x) = \left\{
\begin{array}{ll}
\psi(\frac{1}{r}F^{-1}(x)), & \text{ if }x \in F(\bB^+_{\sigma_1r}),\\
0, & \text{ if }x \in \bB \setminus F(\bB^+_{\sigma_1r}),
\end{array}
\right.
\]
and let 
\[
\varphi(x) = \left\{
\begin{array}{ll}
(d\Psi)_{\widehat{u}(x)}(\widehat{\varphi}(x)), & \text{ if }x \in F(\bB^+_{\sigma_1r}),\\
0, & \text{ if }x \in \bB \setminus F(\bB^+_{\sigma_1r}),
\end{array}
\right.
\]
where we recall that $\widehat{u} = \Phi \circ u$. Then $\varphi \in T_u\cM_p$ and, by a straightforward calculation, we see that the various terms in~\eqref{eq:first-variation} can be expressed using $\widehat{u}$ and $\widehat{\varphi}$ as
\begin{equation}\label{eq:hat-expressions}
\begin{split}
2F_{\tau}(u, \nabla u) =\ & g_{ij}(\widehat{u}) \bangle{\nabla \widehat{u}^i, \nabla \widehat{u}^j} - \tau  P_{ijk}(\widehat{u})\widehat{X}^k(\widehat{u}) \widehat{u}^i_\alpha \widehat{u}^j_\beta \epsilon_{\alpha\beta},\\
\bangle{\nabla u, \nabla \varphi} =\ & g_{ij}(\widehat{u}) \bangle{\nabla \widehat{u}^j, \nabla \widehat{\varphi}^i} + \frac{1}{2}g_{jk,i}(\widehat{u}) \widehat{\varphi}^i \bangle{\nabla \widehat{u}^j, \nabla \widehat{u}^k},\\
X(u) \cdot (\varphi_{x^1} \times u_{x^2} + u_{x^1} \times \varphi_{x^2}) =\ & P_{ijk}(\widehat{u}) \widehat{X}^k(\widehat{u}) \widehat{\varphi}^i_\alpha\widehat{u}^j_\beta\epsilon_{\alpha\beta}\\
& + P_{mjk}(\widehat{u})\widehat{X}^k(\widehat{u}) C^m_{il}(\widehat{u}) \widehat{u}^l_\alpha \widehat{u}^j_\beta \widehat{\varphi}^i \epsilon_{\alpha\beta},\\
(\nabla X)_u(\varphi) \cdot u_{x^1} \times u_{x^2} =\ & \frac{1}{2}P_{ljk}(\widehat{u}) B_{li}(\widehat{u}) \widehat{u}^j_\alpha \widehat{u}^k_\beta \widehat{\varphi}^i \epsilon_{\alpha\beta},\\
h(u) \varphi \cdot u_{x^1} \times u_{x^2} =\ & \frac{1}{2}\widehat{h}(\widehat{u}) P_{ijk}(\widehat{u}) \widehat{u}^j_\alpha \widehat{u}^k_{\beta}\widehat{\varphi}^i\epsilon_{\alpha\beta}.
\end{split}
\end{equation}
Setting 
\begin{equation}\label{eq:hat-coefficients}
\begin{split}
\widehat{N}(z, \xi) =\ & \big(g_{ij}(z)\delta_{\alpha\beta} - \tau P_{ijk}(z) \widehat{X}^k(z) \epsilon_{\alpha\beta}\big)\xi^i_{\alpha} \xi^j_{\beta},\\
\widehat{A}_{i\alpha}(z, \xi) =\ & [1 + \ep^{p-2}(1 + \widehat{N})^{\frac{p}{2} - 1}]\big(g_{ij}(z)\xi^j_\alpha - \tau  P_{ijk}(z) \widehat{X}^k(z) \xi^j_\beta \epsilon_{\alpha\beta}\big),\\
\widehat{A}_i(z, \xi) =\ & \frac{1}{2}\widehat{h}(z)P_{ijk}(z)\xi^j_\alpha\xi^k_\beta \epsilon_{\alpha\beta} + [1 + \ep^{p-2}(1 + \widehat{N})^{\frac{p}{2} - 1}] \frac{1}{2}g_{jk, i}(z) \xi^j_{\alpha} \xi^k_{\beta}\delta_{\alpha\beta}\\
& - \tau  [1 + \ep^{p-2}(1 + \widehat{N})^{\frac{p}{2} - 1}] \big( P_{mjk}(z) \widehat{X}^k(z)C^m_{il}(z) \xi^l_\alpha \xi^j_\beta \epsilon_{\alpha\beta} + \frac{1}{2}P_{ljk}(z) B_{il}(z) \xi^j_\alpha \xi^k_\beta \epsilon_{\alpha\beta} \big),
\end{split}
\end{equation}
we see upon substituting the relations in~\eqref{eq:hat-expressions} into~\eqref{eq:first-variation} that 
\[
\begin{split}
(\delta E_{\ep, p, f, \tau})_u(\varphi) = \int_{\bB} \widehat{A}_{i\alpha}(\widehat{u}, \nabla\widehat{u}) \widehat{\varphi}^i_\alpha + \widehat{A}_i(\widehat{u}, \nabla\widehat{u}) \widehat{\varphi}^i.
\end{split}
\]
Pulling everything back by the conformal map $F(r\ \cdot)$ and noting that
\[
v(x) = \widehat{u}(F(rx))\ \  \text{ and }\ \  \psi(x) = \widehat{\varphi}(F(rx)),
\]
we obtain~\eqref{eq:E-L-fermi-2} as asserted. Noting that $v$ is smooth in $\bB^+_{\sigma_1}$ by Proposition~\ref{prop:interior-regularity} and our choice of $p$, we get~\eqref{eq:E-L-fermi-non-div} by using test functions supported away from $\bT$ in~\eqref{eq:E-L-fermi-2} and observing that $\epsilon_{\alpha\beta} \cdot \partial_{\alpha\beta}v = 0$.

\subsection{Proof of Lemma~\ref{lemm:coefficient-bounds}}
Part (a) is an immediate consequence of~\eqref{eq:norm-close},~\eqref{eq:lambda-bound} and~\eqref{eq:X-hat-length}. For part (b), we first observe that, by~\eqref{eq:N-norm} and the fact that $0 < \frac{p}{2} - 1<1$ since $p \in (2, 3]$,
\begin{equation}\label{eq:N+1-norm}
\frac{1 - |\tau|}{4}(r^2 + |\xi|^2)^{\frac{p}{2} - 1} \leq (r^2 + N)^{\frac{p}{2} - 1} \leq 4(r^2 + |\xi|^2)^{\frac{p}{2} - 1}.
\end{equation}
Combining this with~\eqref{eq:norm-close} gives
\[
\begin{split}
A_{i\alpha}(x, z, \xi) \xi^i_\alpha \geq\ & \frac{1 - |\tau|}{4}[1 + \bep^{p-2}(r^2 + |\xi|^2)^{\frac{p}{2} - 1}] \cdot \frac{1 - |\tau|}{2}|\xi|^2,
\end{split}
\]
which gives the first estimate in part (b). Next, noting that 
\[
A_{i \alpha}(x, z, \xi) =  [1 + \bep^{p-2}(r^2 + N)^{\frac{p}{2} - 1}]  \frac{(\blambda(x))^2}{2}\pa{N}{\xi^i_\alpha}(x, z, \xi),
\]
we differentiate $A^i_\alpha$ with respect $\xi^j_\beta$ to get
\[
\begin{split}
\pa{A^i_\alpha}{\xi^j_\beta} =\ & [1 + \bep^{p-2}(r^2 + N)^{\frac{p}{2} - 1}] (g_{ij}\delta_{\alpha\beta} - \tau P_{ijk} \widehat{X}^k \epsilon_{\alpha\beta})\\
& + \frac{p-2}{2}\cdot  \bep^{p-2}(r^2 + N)^{\frac{p}{2} - 2}\frac{\blambda^2}{2}\pa{N}{\xi^j_\beta} \cdot  \pa{N}{\xi^i_\alpha}.
\end{split}
\]
The second line contributes a non-negative term in $\pa{A^i_\alpha}{\xi^j_\beta}\eta^i_\alpha \eta^j_\beta$. Thus, again combining~\eqref{eq:N+1-norm} with~\eqref{eq:norm-close}, we obtain the lower bound in part (b). Recalling that $H_1$ is chosen so that
\[
|\nabla X| + |\nabla^2 X| + |f| + |\nabla f| \leq H_1 \text{ on }\RR^3,
\]
one can verify part (c) and part (d) in a straightforward manner with the help of estimates on $N$ such as~\eqref{eq:N-norm} and~\eqref{eq:N+1-norm}. We omit the details.
\section{Some standard calculations involving the capillary functional}\label{sec:second-variation}
In this appendix we prove Lemma~\ref{lemm:area-kernel} and Lemma~\ref{lemm:A-second-variation-function}. As mentioned in the introduction and in Section~\ref{sec:index-comparison}, these are not new results of ours, but we include their proofs in the notation of this paper for the reader's convenience. We first introduce some notation to put $(\delta^2A_{H, \tau})_{u}$ into a more familiar form. Since the ambient space is just $\RR^3$, we can bypass the language of pullback bundles and pullback connections. 

Recalling that $\lambda^2 = |u_{x^1}|^2 = |u_{x^2}|^2$, on $\overline{\bB}\setminus \cS$ we set
\[
v_j = \lambda^{-1}u_{x^j}.
\]
Then $\nu = v_1 \times v_2$. Also, we mention for later use that, with $J$ being the complex structure on $E$,
\[
Jv_1 = v_2,\ \ Jv_2 = -v_1.
\]
At each point $x \in \overline{\bB} \setminus \cS$, we identify $E_x$ with $\RR^2$ via the isomorphism $du_x: \RR^2 \to E_x$, which allows us to speak of the derivative of functions defined on $\overline{\bB}$ with respect to directions in $E_x$. More explicitly, given $\psi \in C^{\infty}(\overline{\bB}; \RR^3)$ and $\varphi \in C^{\infty}_{c}(\overline{\bB}\setminus \cS;\RR^3)$ such that $\varphi(x) \in E_x$ for all $x$, we write $\nabla_{\varphi}\psi$ to mean
\[
\nabla_{\varphi}\psi = \sum_{i = 1}^2 (\varphi\cdot v_i) \lambda^{-1}\psi_{x^i}.
\]
Given two functions $\varphi, \xi$ in $C^{\infty}_{c}(\overline{\bB}\setminus \cS;\RR^3)$ such that $\varphi(x), \xi(x) \in E_x$ for all $x$, we define
\[
[\varphi, \xi] = \nabla_{\varphi}\xi - \nabla_{\xi}\varphi.
\]
Then we have $[\varphi, \xi](x) \in E_x$ for all $x$, and that $\nabla_{[\varphi, \xi]} = [\nabla_{\varphi}, \nabla_{\xi}]$. Finally, we let
\[
\Div_{E}\psi = \sum_{i = 1}^2 \nabla_{v_i}\psi \cdot v_i,
\]
for all $\psi \in C^{\infty}(\overline{\bB}; \RR^3)$, and write 
\[
\vol_{E} = \lambda^2 dx.
\]
Then~\eqref{eq:A-second-variation} can be rewritten as
\begin{equation}\label{eq:A-second-var-re}
\begin{split}
(\delta^2 A_{H, \tau})_{u}(\varphi, \varphi) =\ & \int_{\bB} \Big[ \sum_{i = 1}^2 |(\nabla_{v_i}\varphi)^{\perp} |^2 + (\Div_E \varphi)^2 - \sum_{i, j = 1}^2 (\nabla_{v_i}\varphi \cdot v_j)(\nabla_{v_j}\varphi \cdot v_i) \Big] \vol_{E}\\
& + H \int_{\bB} \varphi \cdot \big( \nabla_{v_1}\varphi \times v_2 + v_1 \times \nabla_{v_2}\varphi \big) \vol_{E}\\
& + \int_{\partial \bB}(u_r + \tau X_u \times u_\theta) \cdot A^{\Sigma}_u(\varphi, \varphi) d\theta + \tau \int_{\partial \bB} \varphi \cdot X_u \times \varphi_{\theta} d\theta.
\end{split}
\end{equation}
Likewise, when $f \in C^{\infty}(\overline{\bB} \setminus \cS)$, we may rewrite $(I_{H, \tau})_u(f, f)$ as
\begin{equation}\label{eq:I-second-var-re}
\begin{split}
(I_{H, \tau})_u(f, f)=\ & \int_{\bB} |\nabla_{v_i}f|^2 - f^2|\nabla_{v_i}\nu|^2 \vol_E\\
& + \int_{\partial \bB} \frac{\sigma}{|\sigma|} \big[ \frac{\tau}{\sqrt{1 - \tau^2}}(\nabla_\bn \nu \cdot \bn) + \frac{1}{\sqrt{1 - \tau^2}}h^\Sigma(X_u \times \bt, X_u \times \bt) \big] f^2 \lambda d\theta,
\end{split}
\end{equation}
which is perhaps closer to how this formula usually appears in the literature. 

We next derive a number of identities that are used multiple times in the proof of Lemma~\ref{lemm:area-kernel} and Lemma~\ref{lemm:A-second-variation-function}.
\begin{lemm}\label{lemm:area-preparation}
In the notation introduced in Section~\ref{sec:index-comparison} and in the above, given $\varphi, \psi \in C^{\infty}_c(\overline{\bB}\setminus \cS; \RR^3)$ with $\varphi(x) \in E_x$ for all $x$, we have the following identities: 
\begin{equation}\label{eq:area-prep-1}
\nabla_{\varphi}\Div_{E}\psi - \Div_{E}(\nabla_{\varphi}\psi) = -\big(\nabla_{v_j}\varphi \cdot v_i\big)\big(\nabla_{v_i}\psi \cdot v_j\big) + \nabla_{v_i}\psi \cdot \nabla_{v_i}^\perp\varphi,
\end{equation}
\begin{equation}\label{eq:area-prep-2}
\Div_{E}J(\varphi \times \psi)^T = \psi \cdot (\nabla_{v_1} \varphi \times v_2  + v_1 \times \nabla_{v_2}\varphi) - \varphi \cdot (\nabla_{v_1}\psi \times v_2 + v_1 \times \nabla_{v_2}\psi),
\end{equation}
\begin{equation}\label{eq:area-prep-3}
\varphi \cdot (\nabla_{v_1}\psi \times v_2 + v_1 \times \nabla_{v_2}\psi) = -\nabla_{\varphi}\psi \cdot \nu,
\end{equation}
\begin{equation}\label{eq:area-prep-4}
\nu \cdot \big( \nabla_{v_1}\varphi \times v_2 + v_1 \times \nabla_{v_2}\varphi \big) = \Div_{E}\varphi.
\end{equation}
\end{lemm}
\begin{proof}
For~\eqref{eq:area-prep-1}, a direct computation shows that
\begin{equation}\label{eq:area-prep-1-1}
\begin{split}
\nabla_{\varphi}(\Div_E \psi) - \Div_E(\nabla_{\varphi}\psi) =\ & \nabla_{\varphi}(\nabla_{v_i}\psi \cdot v_i) - \nabla_{v_i}(\nabla_{\varphi}\psi)\cdot v_i\\
=\ & [\nabla_{\varphi}, \nabla_{v_i}]\psi \cdot v_i + \nabla_{v_i}\psi \cdot \nabla_{\varphi}v_i\\
=\ & \nabla_{[\varphi, v_i]}\psi \cdot v_i + \nabla_{v_i}\psi \cdot [\varphi, v_i] + \nabla_{v_i}\psi \cdot \nabla_{v_i}\varphi.
\end{split}
\end{equation}
Noting that $[\varphi, v_i] = \big((\nabla_{\varphi} v_i - \nabla_{v_i}\varphi) \cdot v_j\big) v_j$, we obtain
\[
\begin{split}
 \nabla_{[\varphi, v_i]}\psi \cdot v_i + \nabla_{v_i}\psi \cdot [\varphi, v_i] =\ & \big(\big(\nabla_{\varphi} v_i - \nabla_{v_i}\varphi\big) \cdot v_j\big) \big( \nabla_{v_j}\psi \cdot v_i + \nabla_{v_i}\psi \cdot v_j \big)\\
 =\ & - (\nabla_{v_i}\varphi \cdot v_j) \big( \nabla_{v_j}\psi \cdot v_i + \nabla_{v_i}\psi \cdot v_j \big)
\end{split}
\]
where in getting the second equality we used the fact that $\nabla_{\varphi}v_i \cdot v_j$ is anti-symmetric in $i, j$. Substituting the above into~\eqref{eq:area-prep-1-1} gives
\[
\begin{split}
\nabla_{\varphi}\Div_E \psi - \Div_E(\nabla_{\varphi}\psi) =\ & -(\nabla_{v_i}\varphi \cdot v_j)(\nabla_{v_j}\psi \cdot v_i) + \big( \nabla_{v_i}\psi \cdot \nabla_{v_i}\varphi - (\nabla_{v_i}\psi \cdot v_j)(\nabla_{v_i}\varphi \cdot v_j) \big),
\end{split}
\]
which immediately gives~\eqref{eq:area-prep-1}. For~\eqref{eq:area-prep-2}, recall that $J$ commutes with $\nabla^T$, and thus 
\[
\begin{split}
\nabla_{v_i}(J(\varphi \times \psi)^T) \cdot v_i =\ & -\nabla_{v_i}(\varphi \times \psi)^T \cdot Jv_i\\
=\ & -\nabla_{v_i}\big( (\varphi \times \psi \cdot v_j)v_j \big) \cdot Jv_i\\
=\ & -\nabla_{v_i}(\varphi \times \psi) \cdot Jv_i \\
&- (\nabla_{v_i}v_j\cdot v_k)\big[ (\varphi \times \psi\cdot v_k)(v_j \cdot Jv_i) + (\varphi \times \psi\cdot v_j)(v_k \cdot Jv_i) \big].
\end{split}
\]
The last line vanishes since $\nabla_{v_i}v_j \cdot v_k$ is anti-symmetric in $j, k$ while the sum in square brackets is symmetric in $j, k$. Thus, we deduce that
\[
\begin{split}
\Div_E J(\varphi \times \psi)^T =\ & -\nabla_{v_i}(\varphi \times \psi) \cdot Jv_i \\
=\ & -\nabla_{v_1}(\varphi \times \psi) \cdot v_2 + \nabla_{v_2}(\varphi \times \psi) \cdot v_{1},
\end{split}
\]
from which we obtain~\eqref{eq:area-prep-2}. Next, since 
\[\varphi = (\varphi \cdot v_1)v_1 + (\varphi \cdot v_2)v_2,
\]
we have
\[
\varphi \cdot \nabla_{v_1}\psi \times v_2 = - (\varphi \cdot v_1) \nabla_{v_1}\psi \cdot v_1 \times v_2, \text{ and }
\]
\[\varphi \cdot v_1 \times \nabla_{v_2}\psi = - (\varphi \cdot v_2) \nabla_{v_2}\psi \cdot v_1 \times v_2,
\]
which gives~\eqref{eq:area-prep-3}. To finish, we observe that $\nu \times v_2= -v_1$ and $\nu \times v_1 = v_2$, so that
\[
\nu \cdot \nabla_{v_1}\varphi \times v_2 = (\nabla_{v_1}\varphi \cdot v_1),\ \ \nu \cdot v_1 \times \nabla_{v_2}\varphi= (\nabla_{v_2}\varphi \cdot v_2),
\]
and hence we get~\eqref{eq:area-prep-4}.
\end{proof}
\subsection{Proof of Lemma~\ref{lemm:area-kernel}}
First note that by~\eqref{eq:H-cmc-bvp} we have
\begin{equation}\label{eq:cmc-eq-nu}
\Delta u = H \lambda^2 \nu.
\end{equation}
Also, since $\varphi(x) \in E_x$ for all $x$, on $\partial \bB \setminus \cS$ there holds
\begin{equation}\label{eq:V-E-boundary}
\varphi = \lambda^{-2}(\varphi \cdot u_{r})u_{r} + \lambda^{-2}(\varphi \cdot u_\theta)u_\theta = (\varphi \cdot \bn)\bn +  (\varphi \cdot \bt )\bt.
\end{equation}
Now, polarizing~\eqref{eq:A-second-var-re} gives
\begin{equation}\label{eq:A-polar}
\begin{split}
(\delta^2 A_{H, \tau})_u(\varphi, \psi) =\ & \int_{\bB}\big[ (\Div_{E} \varphi)(\Div_{E}\psi) + \sum_{i=1}^2\bangle{\nabla_{v_i}^\perp \varphi, \nabla_{v_i}^\perp \psi} - \sum_{i, j = 1}^2 \big(\nabla_{v_i}\varphi \cdot v_j\big)\big(\nabla_{v_j}\psi \cdot v_i\big)\big] \vol_E \\
& + \frac{H}{2}\int_{\bB} \big[\varphi \cdot \big( \nabla_{v_1}\psi \times v_2 + v_1 \times \nabla_{v_2}\psi \big) + \psi \cdot \big( \nabla_{v_1}\varphi \times v_2 + v_1 \times \nabla_{v_2}\varphi \big) \big] \vol_E\\
& + \frac{\tau}{2}\int_{\partial\bB} X_u \cdot \big( \varphi_\theta \times \psi + \psi_\theta \times \varphi  \big)\ d\theta\\
& + \int_{\partial\bB} (u_r + \tau\cdot X_u \times u_\theta) \cdot A^\Sigma_u(\varphi, \psi)\ d\theta\\
=:\ & (I1) + (I2) + (B1) + (B2).
\end{split}
\end{equation}
We begin with the two interior integrals. Recalling the definition of $\Div_E$ and integrating by parts yields
\begin{equation}\label{eq:tangential-div-term}
\begin{split}
\int_{\bB}(\Div_E \varphi)(\Div_E \psi) \vol_E =\ & \int_{\bB}u_{x^i} \cdot \varphi_{x^i} (\Div_E \psi) dx\\
=\ & \int_{\partial \bB} (\varphi \cdot u_r) (\Div_E \psi) d\theta - \int_{\bB} (\varphi \cdot \Delta u) (\Div_E \psi) dx\\
& - \int_{\bB} (\varphi \cdot u_{x^i})\paop{x^i}(\Div_E \psi) dx\\
=\ & \int_{\partial \bB} (\varphi \cdot u_r) (\Div_E \psi) d\theta - \int_{\bB}\nabla_{\varphi}\Div_E \psi \vol_E,
\end{split}
\end{equation}
where to get the last line we used~\eqref{eq:cmc-eq-nu} and the assumption that $\varphi(x) \in E_x$ for all $x$. By~\eqref{eq:area-prep-1} and another integration by parts, we have
\begin{equation}\label{eq:tangential-div-term-1}
\begin{split}
-\int_{\bB} \nabla_{\varphi}\Div_E \psi\vol_E =\ & -\int_{\bB} \Div_E \nabla_{\varphi}\psi \vol_E + \int_{\bB} (\nabla_{v_i}\varphi \cdot v_j)(\nabla_{v_j}\psi \cdot v_i) - \nabla_{v_i}\varphi \cdot \nabla^\perp_{v_i}\psi \vol_E\\
=\ & -\int_{\partial \bB} u_r \cdot \nabla_{\varphi} \psi d\theta + \int_{\bB}\Delta u \cdot \nabla_{\varphi} \psi dx\\
& + \int_{\bB} (\nabla_{v_i}\varphi \cdot v_j)(\nabla_{v_j}\psi \cdot v_i) - \nabla_{v_i}\varphi \cdot \nabla^\perp_{v_i}\psi \vol_E\\
=\ & -\int_{\partial \bB} u_r \cdot \nabla_{\varphi} \psi d\theta + H\int_{\bB}\nu \cdot \nabla_{\varphi} \psi \vol_E\\
& + \int_{\bB} (\nabla_{v_i}\varphi \cdot v_j)(\nabla_{v_j}\psi \cdot v_i) - \nabla_{v_i}\varphi \cdot \nabla^\perp_{v_i}\psi \vol_E.
\end{split}
\end{equation}
Using also~\eqref{eq:area-prep-3} and combining the result with~\eqref{eq:tangential-div-term}, we see that 
\begin{equation}\label{eq:tangential-div-term-2}
\begin{split}
(I1) + (I2) =\ & \int_{\partial \bB} (\varphi \cdot u_r)(\Div_E \psi) d\theta - \int_{\partial \bB}(u_r \cdot \nabla_{\varphi}\psi) d\theta\\
& + \frac{H}{2}\int_{\bB} \psi \cdot (\nabla_{v_1}\varphi \times v_2 + v_1 \times \nabla_{v_2}\varphi) - \varphi \cdot (\nabla_{v_1}\psi \times v_2 + v_1 \times \nabla_{v_2}\psi) \vol_E.
\end{split}
\end{equation}
By~\eqref{eq:area-prep-2}, the second line above can be transformed into
\[
\begin{split}
\frac{H}{2}\int_{\bB}u_{x^i}\cdot\paop{x^i}(J(\varphi \times \psi)^T) dx=\ & \frac{H}{2}\int_{\partial \bB} u_r \cdot J(\varphi \times \psi)^T d\theta - \frac{H}{2}\int_{\bB} \Delta u \cdot J(\varphi \times \psi)^T dx\\
=\ & -\frac{H}{2} \int_{\partial \bB} \varphi \times \psi \cdot u_\theta d\theta,
\end{split}
\]
where to get the last line we used~\eqref{eq:cmc-eq-nu} again, and also observed that $Ju_r = u_\theta$. Substituting this back into~\eqref{eq:tangential-div-term-2} gives
\[
\begin{split}
(I1) + (I2) =\ & \int_{\partial \bB} (\varphi \cdot u_r)(\Div_E \psi) d\theta - \int_{\partial \bB}(u_r \cdot \nabla_{\varphi}\psi) d\theta - \frac{H}{2}\int_{\partial \bB} \varphi \times \psi \cdot u_\theta d\theta.
\end{split}
\]
We next turn to the two boundary integrals. Since $u_r + \tau X_u \times u_\theta \perp T_u\Sigma$, we have
\[
(u_r + \tau X_u \times u_\theta) \cdot A^\Sigma_u(\varphi, \psi) = (u_r + \tau X_u \times u_\theta) \cdot \nabla_{\varphi}\psi = u_r \cdot \nabla_{\varphi}\psi + \tau X_u \times u_\theta \cdot \nabla_{\varphi}\psi.
\]
For the last term, we expand $\varphi$ using~\eqref{eq:V-E-boundary} and compute
\[
\begin{split}
\tau X_u \times u_\theta \cdot \nabla_{\varphi}\psi - \tau X_u \times \varphi \cdot \psi_{\theta} =\ & \tau (\varphi \cdot \bn)(X_u \times \bt \cdot \psi_r - X_u \times \bn \cdot \psi_\theta).
\end{split}
\]
Therefore
\begin{equation}\label{eq:tangent-boundary-terms}
\begin{split}
(B1) + (B2) =\ & \int_{\partial \bB} u_r \cdot \nabla_{\varphi}\psi + \tau(\varphi \cdot \bn)(X_u \times \bt \cdot \psi_r - X_u \times \bn \cdot \psi_\theta) d\theta\\
& + \frac{\tau}{2}\int_{\partial \bB} X_u \cdot (\varphi_\theta \times \psi + \varphi \times \psi_\theta) d\theta\\
=\ & \int_{\partial \bB} u_r \cdot \nabla_{\varphi}\psi + \tau(\varphi \cdot \bn)(X_u \times \bt \cdot \psi_r - X_u \times \bn \cdot \psi_\theta) d\theta\\
& -\frac{\tau}{2}\int_{\partial \bB}(\nabla X)_u(u_\theta) \cdot \varphi \times \psi d\theta.
\end{split}
\end{equation}
where in getting the last line we integrated by parts. We obtain~\eqref{eq:area-kernel-formula} upon adding~\eqref{eq:tangent-boundary-terms} and~\eqref{eq:tangential-div-term-2} and observing a cancellation.

When $\varphi \in \cE_{\cS}$ and $\psi \in \cV_{\cS}$, the term in~\eqref{eq:area-kernel-formula} involving $\nabla X$ vanishes because on $\partial \bB$, the vectors $\varphi, \psi$ and $(\nabla X)_u(u_\theta)$ all lie in $T_u \Sigma$. The remaining terms in~\eqref{eq:area-kernel-formula} vanish because $E \cap T_u\Sigma$ is spanned by $\bt$, of which $\varphi$ must be a multiple, so that $\varphi \cdot u_r = 0$ and $\varphi \times u_\theta = 0$ on $\partial \bB$. 
\subsection{Proof of Lemma~\ref{lemm:A-second-variation-function}}
Letting 
\[
\varphi = f\nu + \xi,
\]
we have that $\xi \in C^{\infty}(\overline{\bB}\setminus \cS;\RR^3)$ and that $\xi(x) \in E_x$ for all $x$, which allows us to make use of the calculations in the previous proof. Specifically, note that in the two integrals over $\bB$ in~\eqref{eq:A-second-var-re}, each term is quadratic in $\varphi$ and can thus be split in the following manner (taking the divergence term as an example):
\[
(\Div_E \varphi)^2 = (\Div_E \xi)(\Div_E \varphi) + (\Div_E(f\nu))(\Div_E \xi) + (\Div_E(f\nu))^2
\]
Below we treat separately each of the three types of terms represented respectively by the three terms on the right hand side above:
\begin{enumerate}
\item[(1)] Applying~\eqref{eq:tangential-div-term-2} with $\xi$ and $\varphi$ in place of $\varphi$ and $\psi$, we get after rearranging that
\begin{equation}\label{eq:interior-xi-phi}
\begin{split}
&\int_{\bB}(\Div_E \xi)(\Div_E \varphi) + \nabla_{v_i}^\perp\xi \cdot \nabla_{v_i}^\perp \varphi - (\nabla_{v_i}\xi \cdot v_j)(\nabla_{v_j}\varphi \cdot v_i) \vol_E\\
& + H\int_{\bB} \xi \cdot (\nabla_{v_1}\varphi \times v_2 + v_1 \times \nabla_{v_2}\varphi) \vol_E\\
=\ & \int_{\partial \bB} (\xi \cdot u_r) (\Div_E \varphi) d\theta  -\int_{\partial \bB}u_r \cdot \nabla_{\xi}\varphi d\theta\\
=\ & \int_{\partial \bB} (\varphi \cdot u_r) (\Div_E \varphi) d\theta  -\int_{\partial \bB}u_r \cdot \nabla_{\xi}\varphi d\theta.
\end{split}
\end{equation}
The last equality follows because $\varphi \cdot u_r = \xi \cdot u_r$. 
\vskip 1mm
\item[(2)] Applying~\eqref{eq:tangential-div-term-2} again, but this time with $\xi$ in place of $\varphi$ and $f\nu$ in place of $\psi$, we obtain
\begin{equation}\label{eq:interior-xi-nu}
\begin{split}
&\int_{\bB} \Div_E(f\nu) \Div_E \xi + \nabla_{v_i}^\perp(f\nu) \cdot \nabla^\perp_{v_i}\xi - (\nabla_{v_i}(f\nu) \cdot v_j)(\nabla_{v_j}\xi \cdot v_i) \vol_E\\
& + H \int_{\bB} f\nu \cdot (\nabla_{v_1}\xi \times v_2 + v_1 \times \nabla_{v_2}\xi)\vol_E\\
=\ & \int_{\partial\bB}(\xi \cdot u_r)\Div_E(f\nu)d\theta - \int_{\partial \bB} u_r \cdot \nabla_{\xi}(f\nu) d\theta\\
& + H \int_{\bB} f\nu \cdot (\nabla_{v_1}\xi \times v_2 + v_1 \times \nabla_{v_2}\xi) - \xi\cdot (\nabla_{v_1}(f\nu) \times v_2 + v_1 \times \nabla_{v_2}(f\nu)) \vol_E\\
=\ & - \int_{\partial \bB} u_r\cdot \nabla_{\xi}(f\nu) d\theta.
\end{split}
\end{equation}
The cancellation leading to the last equality follows from the fact that 
\[
\Div_E(f\nu) = \lambda^{-2}u_{x^i} \cdot \paop{x^i}(f\nu) = -\lambda^{-2 }\Delta u \cdot f\nu = -Hf,
\]
and that by~\eqref{eq:area-prep-3} and~\eqref{eq:area-prep-4}, the second to last line in~\eqref{eq:interior-xi-nu} is equal to
\[
\begin{split}
H \int_{\bB} f\Div_E \xi + \nabla_{\xi}(f\nu) \cdot \nu \vol_E =\ & H \int_{\bB} f\Div_E \xi + \nabla_{\xi}f \vol_E\\
=\ & H \int_{\partial \bB} f(\xi \cdot u_r)d\theta - H \int_{\bB} f \xi \cdot \Delta u dx\\
=\ &H \int_{\partial \bB} f(\xi \cdot u_r)d\theta.
\end{split}
\]
To further simplify the last line of~\eqref{eq:interior-xi-nu}, note that since $\xi(x) \in E_x$ for all $x$ and $\xi$ is supported away from $\cS$, we have
\[
\xi = \lambda^{-2}(\xi \cdot u_{x^i})u_{x^i}.
\]
Thus we may compute
\begin{equation}\label{eq:interior-xi-nu-last}
\begin{split}
u_r \cdot \nabla_{\xi}(f\nu) = \ & f \nabla_{\xi}\nu \cdot u_r\\
=\ & f\lambda^{-2}(\xi \cdot u_{x^i})\nu_{x^i} \cdot u_r\\
=\ & f\lambda^{-2}(\xi \cdot u_{x^i})\nu_{r} \cdot u_{x^i} = f \nu_r \cdot \xi.
\end{split}
\end{equation}
\vskip 1mm
\item[(3)] By a more or less standard calculation, the details of which we omit, there holds
\begin{equation}\label{eq:interior-nu-nu}
\begin{split}
&\int_{\bB}(\Div_E (f\nu))^2 + |\nabla_{v_i}^\perp (f\nu)|^2 - (\nabla_{v_i}(f\nu) \cdot v_j)(\nabla_{v_j}(f\nu) \cdot v_i) \vol_E\\
& + H \int_{\bB} f\nu \cdot \big( \nabla_{v_1}(f\nu) \times v_2 + v_1 \times \nabla_{v_2}(f\nu) \big) \vol_{E}\\
=\ & \int_{\bB}|\nabla_{v_i}f|^2  - |\nabla_{v_i}\nu|^2 f^2 \vol_E.
\end{split}
\end{equation}
\end{enumerate}
Combining~\eqref{eq:interior-nu-nu},~\eqref{eq:interior-xi-nu-last},~\eqref{eq:interior-xi-nu} and~\eqref{eq:interior-xi-phi}, we obtain from~\eqref{eq:A-second-var-re} that
\begin{equation}\label{eq:A-second-variation-boundary}
\begin{split}
(\delta^2 A_{H, \tau})_u(\varphi, \varphi)  =\ & \int_{\bB}|\nabla_{v_i}f|^2  - |\nabla_{v_i}\nu|^2 f^2 \vol_E\\
& - \int_{\partial\bB} f ( \nu_r \cdot \xi)\ d\theta + \int_{\partial \bB} \big[(\varphi \cdot u_r)\Div_{E}\varphi - (\nabla_{\xi}\varphi \cdot u_r)\big]\ d\theta\\
& - \tau \int_{\partial \bB} \varphi_{\theta} \cdot X_u \times \varphi\  d\theta + \int_{\partial \bB} A^\Sigma_u(\varphi, \varphi)  \cdot (u_r + \tau X_u \times u_\theta)d\theta\\
=:\ & \int_{\bB}|\nabla_{v_i}f|^2  - |\nabla_{v_i}\nu|^2 f^2 \vol_E  - (B1) + (B2) - (B3) + (B4).
\end{split}
\end{equation}
Before simplifying the boundary terms, we note that
\[
\int_{\bB}|\nabla_{v_i}f|^2  - f^2|\nabla_{v_i}\nu|^2 \vol_E = \int_{\bB} |\nabla f|^2 - f^2 |\nabla \nu|^2  dx.
\]
Turning now to the boundary terms in~\eqref{eq:A-second-variation-boundary}, in terms of the notation from Section~\ref{subsec:index-notation}, on $\partial \bB \setminus \cS$ we have
\begin{equation}\label{eq:xi-t-n-exp}
\begin{split}
\xi =\ & (\varphi \cdot \bt)\bt + (\varphi \cdot \bn) \bn = \lambda^{-2}(\varphi \cdot u_\theta)u_\theta + \lambda^{-2}(\varphi \cdot u_r)u_r,
\end{split}
\end{equation}
and
\begin{equation}\label{eq:phi-t-n-exp}
\begin{split}
\varphi =\ & (\varphi \cdot \bt) \bt + (\varphi \cdot X_u \times \bt)X_u \times \bt\\
=\ & \lambda^{-2}(\varphi \cdot u_\theta)u_\theta + \lambda^{-2}(\varphi \cdot X_u \times u_\theta) X_u \times u_\theta.    
\end{split}
\end{equation}
Also, recalling that $f = \varphi \cdot \nu$, from~\eqref{eq:nu-t-n} and the fact that $\varphi \cdot X_u = 0$ on $\partial \bB$ we obtain
\begin{equation}\label{eq:f-t-n}
f = \sigma \varphi \cdot X_u \times \bt,
\end{equation}
and thus using~\eqref{eq:capillary-t-n} we deduce that
\begin{equation}\label{eq:f-t-n-2}
\varphi \cdot \bn = -\tau \varphi \cdot X_u \times \bt = -\frac{\tau}{\sigma}f.
\end{equation}
In these notation, we have from~\eqref{eq:xi-t-n-exp} and~\eqref{eq:f-t-n-2} that
\begin{equation}\label{eq:I-boundary-term-1}
\begin{split}
(B1) = \ & \int_{\partial \bB}f[(\varphi \cdot \bt)(\nu_r \cdot \bt) + (\varphi \cdot \bn)(\nu_r \cdot \bn)] d\theta\\
=\ & \int_{\partial \bB} -\frac{\tau}{\sigma}f^2(\nu_r \cdot \bn) + f(\varphi \cdot \bt)(\nu_r \cdot \bt) d\theta.
\end{split}
\end{equation}
Leaving (B1) for now and continuing to (B2), we have by~\eqref{eq:xi-t-n-exp} that
\begin{equation}\label{eq:I-boundary-term-2-pt}
\begin{split}
(\varphi \cdot u_r)\Div_E \varphi - (\nabla_{\xi}\varphi \cdot u_r)=\ & (\varphi \cdot u_r)[\lambda^{-2}\varphi_r \cdot u_r + \lambda^{-2}\varphi_\theta \cdot u_\theta]\\
& - \lambda^{-2}(\varphi \cdot u_\theta)(\varphi_\theta\cdot u_r) - \lambda^{-2}(\varphi \cdot u_r)(\varphi_r \cdot u_r)\\
=\ & \lambda^{-2}(\varphi \cdot u_r)(\varphi_\theta \cdot u_\theta) - \lambda^{-2}(\varphi \cdot u_\theta)(\varphi_\theta \cdot u_r).
\end{split}
\end{equation}
Recalling the formula for the triple cross product and using~\eqref{eq:nu-t-n}, the last line is equal to 
\[
\begin{split}
\lambda^{-2}\varphi_\theta \cdot \varphi \times (u_\theta \times u_r) =\ & -\varphi_{\theta} \cdot \varphi \times \nu = \varphi_\theta \cdot \nu \times \varphi\\
=\ & \tau\varphi_\theta \cdot X_u \times \varphi + \sigma\varphi_{\theta} \cdot (X_u \times \bt) \times \varphi\\
=\ & \tau\varphi_\theta \cdot X_u \times \varphi - \sigma(\varphi_\theta \cdot X_u)(\varphi \cdot \bt).
\end{split}
\]
Integrating and taking (B3) into account gives
\begin{equation}\label{eq:I-boundary-term-2-3}
\begin{split}
(B2) - (B3) =\ & -\int_{\partial \bB} \sigma(\varphi_\theta \cdot X_u)(\varphi \cdot \bt) d\theta = -\int_{\partial \bB}\sigma \lambda h^\Sigma((\varphi \cdot \bt)\bt, \varphi) d\theta.
\end{split}
\end{equation}
Next, by~\eqref{eq:capillary-t-n} we have
\[
\begin{split}
(B4) =\ & \int_{\partial \bB} \lambda \sigma h^\Sigma(\varphi, \varphi) d\theta.
\end{split}
\]
Adding this to~\eqref{eq:I-boundary-term-2-3} and recalling~\eqref{eq:phi-t-n-exp} and~\eqref{eq:f-t-n} yields
\[
\begin{split}
(B2) - (B3) + (B4) =\ & \int_{\partial\bB} \lambda \sigma h^\Sigma((\varphi \cdot X_u \times \bt)X_u \times \bt, \varphi) d\theta\\
=\ & \int_{\partial \bB} \frac{\lambda}{\sigma}f^2 h^\Sigma(X_u \times \bt, X_u \times \bt) + \lambda f h^{\Sigma}(X_u \times \bt, (\varphi \cdot \bt)\bt ) d\theta.
\end{split}
\]
From this and~\eqref{eq:I-boundary-term-1} we finally get
\begin{equation}\label{eq:I-boundary-total}
\begin{split}
-(B1) + (B2) - (B3) + (B4) =\ & \int_{\partial \bB} \frac{\tau}{\sigma}f^2(\nu_r \cdot \bn) +  \frac{\lambda}{\sigma}f^2 h^\Sigma(X_u \times \bt, X_u \times \bt) d\theta\\
& + \int_{\partial \bB} f(\varphi \cdot \bt)[-(\nu_r \cdot \bt) + \lambda  h^{\Sigma}(X_u \times \bt, \bt )] d\theta\\
=\ & \int_{\partial \bB} \frac{\tau}{\sigma}f^2(\nu_r \cdot \bn) +  \frac{\lambda}{\sigma}f^2 h^\Sigma(X_u \times \bt, X_u \times \bt) d\theta.
\end{split}
\end{equation}
The last equality is obtained as follows: From~\eqref{eq:capillary-t-n} and~\eqref{eq:nu-t-n} we deduce that
\[
\begin{split}
\nu_r \cdot \bt =\ & \lambda^{-1}\nu_r \cdot u_\theta = \lambda^{-1} \nu_\theta \cdot u_r\\
=\ &  (\tau (X_u)_\theta + \sigma(X_u \times \bt)_\theta) \cdot (-\tau X_u \times \bt + \sigma X_u)\\
=\ & -\tau^2 (X_u)_{\theta} \cdot X_u \times \bt + \sigma^2 (X_u \times \bt)_\theta \cdot X_u\\
=\ & \lambda h^\Sigma(X_u \times \bt, \bt),
\end{split}
\]
where to get the third line we used that fact that $X_u$ and $X_u \times \bt$ are  both unit vector fields on $\partial \bB \setminus \cS$. In getting the last line, we used the orthogonality of $X_u \times \bt$ and $X_u$, and the relation $\tau^2 + \sigma^2 = 1$.

To finish, note that from~\eqref{eq:capillary-t-n} we have
\[
\frac{\lambda}{\sigma} = \frac{\lambda \sigma}{\sigma^2} = \frac{\lambda \bn \cdot X_u}{1 - \tau^2} = \frac{u_r \cdot X_u}{1 - \tau^2}.
\]
Combining this with~\eqref{eq:tilde-t-n} gives
\begin{equation}\label{eq:last-boundary-1}
\frac{\lambda}{\sigma}h^\Sigma(X_u \times \bt, X_u \times \bt) = \frac{u_r \cdot X_u}{1 - \tau^2}h^\Sigma(\widetilde{\nu}, \widetilde{\nu}).
\end{equation}
On the other hand, again by~\eqref{eq:tilde-t-n} we see that
\[
\frac{1}{\sigma} (\nu_r \cdot \bn) = \frac{1}{|\sigma|} (\nu_r \cdot \widetilde{X}) = \frac{1}{\sqrt{1 - \tau^2}}(\nu_r \cdot \widetilde{X}).
\]
Substituting this and~\eqref{eq:last-boundary-1} into~\eqref{eq:I-boundary-total} and recalling~\eqref{eq:A-second-variation-boundary}, we get 
\[
(\delta^2 A_{H, \tau})_u(\varphi, \varphi)  = (I_{H, \tau})_u(f, f)
\]
as asserted.

\section{Standard estimates for a class of oblique derivative problems}\label{sec:linear}
For the sake of completeness, here we give an account of the linear estimates for oblique derivative problems that are used in Section~\ref{subsec:smoothness} and Section~\ref{subsec:a-priori}, taking for granted the $W^{2, p}$ theory for Dirichlet and Neumann problems. The material in this appendix is mostly taken from~\cite{Krylov2008}. Below, whenever $L^{p}$ spaces or Sobolev spaces appear, it is understood that $1 < p < \infty$. Given $u \in W^{1, p}(\RR^2_+)$, we write $u|_{\partial\RR^2_+}$ to mean its trace on $\partial \RR^2_+$. Also, for $u, v \in W^{1, p}(\RR^2_+)$, we sometimes write ``$u = v$ on $\partial\RR^2_+$'' to mean $u|_{\partial\RR^2_+} = v|_{\partial\RR^2_+}$.
We first introduce the boundary operators of interest. 
\begin{defi}\label{defi:B0-defi}
Fixing $a \in (-1, 1)$, for $u \in W^{2, p}(\RR^2_+; \RR^2)$ we define
\[
B_a^{\mp}u = \pa{u}{y} \mp \left( \begin{array}{cc} 0 & a \\ -a & 0 \end{array} \right)\pa{u}{x}.
\]
Notice that 
\begin{equation}\label{eq:B0B1}
B_a^{-}B_a^{+}u = B_a^{+}B_a^{-} u = \frac{\partial^2 u}{\partial y^2} + a^2\frac{\partial^2 u}{\partial x^2}.
\end{equation}
\end{defi}
The contents of Lemma~\ref{lemm:model-bc} through Lemma~\ref{lemm:inhomogeneous} are as follows. Lemma~\ref{lemm:model-bc} gives $W^{2,p}$-estimates when $B_a^{-}u$ vanishes on the boundary. Then, the construction in Lemma~\ref{lemm:correction} allows us to obtain a priori estimates for general boundary data in Lemma~\ref{lemm:inhomogeneous-estimate} by reducing to the homogeneous case. In Lemma~\ref{lemm:inhomogeneous} we use the estimates in Lemma~\ref{lemm:inhomogeneous-estimate} and the method of continuity to prove the solvability in $W^{2, p}(\RR^2_+; \RR^2)$ of the boundary value problem
\[
\Delta u - u = f \text{ on }\RR^2_+,\ \ B_a^{-}u = g \text{ on }\partial \RR^2_+,
\]
for $f \in L^{p}(\RR^2_+; \RR^2)$ and $g \in W^{1,p}(\RR^2_+; \RR^2)$.
\begin{lemm}\label{lemm:model-bc}
Given $a_0 \in (0, 1)$, there exists $C > 0$ depending only on $p$ and $a_0$ such that if $u \in W^{2, p}(\RR^2_+; \RR^2)$ and $B_a^{-} u = 0$ on $\partial \RR^2_+$, with $|a| \leq a_0$, then 
\[
\|u\|_{2, p; \RR^2_+} \leq C\|\Delta u - u\|_{p; \RR^2_+}.
\]
\end{lemm}
\begin{proof}
Define $v = B_a^{-} u$ and $f = \Delta u - u$. A direct computation shows that
\[
\int_{\RR^2_+} \bangle{\nabla v, \nabla \varphi} + v \cdot \varphi = \int_{\RR^2_+} f \cdot B_a^{+} \varphi,
\]
for all $\varphi \in C^{\infty}_c(\overline{\RR^2_+}; \RR^2)$ such that $\varphi = 0$ on $\partial \RR^2_+$. Since $v$ lies in $W^{1, p}(\RR^2_+; \RR^2)$ and satisfies a homogeneous Dirichlet condition on $\partial \RR^2_+$, we deduce that
\[
\|v\|_{1, p; \RR^2_+} \leq C\|f\|_{p; \RR^2_+},
\]
where $C$ depends only on $p$. Applying $B_a^{+}$ to $v$ and using the identity~\eqref{eq:B0B1}, we get 
\begin{equation}\label{eq:B0B1-rearranged}
\frac{\partial^2 u}{\partial y^2} = B_a^{+} v  - a^2\frac{\partial^2 u}{\partial x^2},
\end{equation}
so that 
\[
(1 - a^2)\frac{\partial^2 u}{\partial x^2} - u = \Delta u - u - B_a^{+} v =  f  - B_a^{+} v.
\]
A scaling argument together with $L^p$-estimates for $(\paop{x})^2 - 1$ on $\RR$ shows that
\[
\|u(\cdot, y)\|_{p; \RR} + (1 - a^2)\Big\| \frac{\partial^2 u}{\partial x^2}(\cdot, y) \Big\|_{p; \RR} \leq C\big(\|f(\cdot, y)\|_{p; \RR} + \|B_a^{+} v(\cdot, y)\|_{p; \RR}\big),
\]
for almost every $y > 0$. Raising both sides to the $p$-th power and integrating with respect to $y$ on $(0, \infty)$, we get  
\[
\|u\|_{p; \RR^2_+} + (1 - a^2)\Big\| \frac{\partial^2 u}{\partial x^2} \Big\|_{p; \RR^2_+} \leq C\big( \|f\|_{p; \RR^2_+} + \|v\|_{1, p; \RR^2_+} \big) \leq C\|f\|_{p; \RR^2_+},
\]
where $C$ depends only on $p$. Going back to~\eqref{eq:B0B1-rearranged}, we see that 
\[
\Big\| \frac{\partial^2 u}{\partial y^2} \Big\|_{p; \RR^2_+} \leq C\big( \|v\|_{1, p; \RR^2_+} + \big\| \frac{\partial^2 u}{\partial x^2} \big\|_{p; \RR^2_+} \big) \leq C\Big( 1 + \frac{1}{1 - a^2} \Big)\|f\|_{p; \RR^2_+}.
\]
Similarly, differentiating the relation $v = B_a^{-}u$ with respect to $x$ allows us to deduce that 
\[
\Big\| \frac{\partial^2 u}{\partial x \partial y} \Big\|_{p; \RR^2_+} \leq C\big( \|v\|_{1, p; \RR^2_+} + \big\| \frac{\partial^2 u}{\partial x^2} \big\|_{p; \RR^2_+}  \big) \leq C\Big( 1 + \frac{1}{1 - a^2} \Big)\|f\|_{p; \RR^2_+}.
\]
\end{proof}
\begin{lemm}
\label{lemm:correction}
Let $g \in W^{1, p}(\RR^2_+; \RR^2)$. There exists $v \in W^{2, p}(\RR^2_+; \RR^2)$ such that $B_a^{-} v = g$ on $\partial \RR^2_+$ and that 
\begin{equation}\label{eq:correction-estimate}
\|v\|_{2, p; \RR^2_+} \leq C\|g\|_{1, p; \RR^2_+},
\end{equation}
where $C$ depends only on $p$.
\end{lemm}
\begin{proof}
Let $w$ be the solution in $W^{2, p}(\RR^2_+; \RR^2)$ of the Dirichlet problem below:
\[
\left\{
\begin{array}{ll}
\Delta w - w = g, &\text{ in }\RR^2_+,\\
w = 0, & \text{ on }\partial \RR^2_+.
\end{array}
\right.
\]
Then $w$ in fact lies in $W^{3, p}(\RR^2_+; \RR^2)$ and satisfies the estimate
\[
\|w\|_{3, p; \RR^2_+} \leq C_p\|g\|_{1, p; \RR^2_+}.
\]
Also, $\frac{\partial^2 w}{\partial x^2} = 0$ on $\partial \RR^2_+$. Next we define $v = B_a^{+}w$ and observe that by~\eqref{eq:B0B1} we have
\[
B_a^{-} v  = g + w - (1 - a^2)\frac{\partial^2 w}{\partial x^2},
\] 
so that $B_a^{-} v - g = 0$ on $\partial \RR^2_+$. Moreover, it is straightforward from its definition that $v \in W^{2, p}(\RR^2_+; \RR^2)$ and 
\[
\|v\|_{2, p; \RR^2_+} \leq C\|w\|_{3, p; \RR^2_+} \leq C\|g\|_{1, p; \RR^2_+}.
\]
\end{proof}

\begin{lemm}\label{lemm:inhomogeneous-estimate}
Given $a_0 \in (0, 1)$, there exists $C > 0$ depending only on $p$ and $a_0$ such that for all $a \in [-a_0, a_0]$ and $u \in W^{2, p}(\RR^2_+; \RR^2)$ we have
\[
\|u\|_{2, p; \RR^2_+} \leq C\big( \| \Delta u - u \|_{p; \RR^2_+} + \|g\|_{1, p; \RR^2_+} \big),
\]
where $g$ is any function in $W^{1, p}(\RR^2_+; \RR^2)$ satisfying $g|_{\partial \RR^2_+} = (B_a^{-}u)|_{\partial\RR^2_+}$.
\end{lemm}
\begin{proof}
Let $g$ be as in the statement and $v$ as given by Lemma~\ref{lemm:correction}. Then $w:= u - v$ lies in $W^{2, p}(\RR^2_+; \RR^2)$ and
\[
\Delta w - w = \Delta u - u - \Delta v + v;\ \  (B_a^{-}w)|_{\partial \RR^2_+} = 0.
\]
Thus, according to Lemma~\ref{lemm:model-bc},
\[
\|w\|_{2, p; \RR^2_+}  \leq C( \|\Delta u - u\|_{p; \RR^2_+} + \|v\|_{2, p; \RR^2_+} ),
\]
with $C$ depending only on $p$ and $a_0$. Since $\|u\|_{2, p; \RR^2_+} \leq \|w\|_{2, p; \RR^2_+} + \|v\|_{2, p; \RR^2_+}$, combining the above with~\eqref{eq:correction-estimate} gives the desired estimate on $u$.
\end{proof}

\begin{lemm}\label{lemm:inhomogeneous}
For any $a \in (-1, 1)$, given $f \in L^p(\RR^2_+; \RR^2)$ and $g \in W^{1, p}(\RR^2_+; \RR^2)$, there exists a unique $u \in W^{2, p}(\RR^2_+; \RR^2)$ such that 
\begin{equation}\label{eq:inhomogeneous-bvp}
\Delta u - u = f \text{ in }\RR^2_+;\ (B_a^{-} u)|_{\partial\RR^2_+} = g|_{\partial\RR^2_+}.
\end{equation}
\end{lemm}
\begin{proof}
Uniqueness follows from Lemma~\ref{lemm:model-bc}. To establish existence, we introduce the space
\[
\cZ_p = L^p(\RR^2_+; \RR^2) \times W^{1-\frac{1}{p}, p}(\partial \RR^2_+; \RR^2).
\]
Elements in the second factor have the form $g|_{\partial \RR^2_+}$ for $g \in W^{1, p}(\RR^2_+; \RR^2)$, and we equip $\cZ_p$ with the norm
\begin{equation}\label{eq:Zp-norm}
\|(f, g|_{\partial \RR^2_+})\|_{\cZ_p} = \|f\|_{p; \RR^2_+} + \inf \big\{ \| h \|_{1, p; \RR^2_+}\ \big|\ h \in W^{1, p}(\RR^2_+; \RR^2) \text{ and }h|_{\partial \RR^2_+} = g|_{\partial \RR^2_+}\big\},
\end{equation}
which makes it a Banach space. Next we define $L_a: W^{2, p}(\RR^2_+; \RR^2) \to \cZ_p$ by letting
\[
L_a u = (\Delta u - u, (B_a^{-} u)|_{\partial \RR^2_+}).
\]
This is a bounded linear operator. Moreover, fixing $a_0 \in (0, 1)$, Lemma~\ref{lemm:inhomogeneous-estimate} implies that for all $a \in [-a_0, a_0]$ we have 
\begin{equation}\label{eq:continuity-method}
\|u\|_{2, p; \RR^2_+} \leq C\|L_a u\|_{\cZ_p},
\end{equation}
with $C$ depending only on $p$ and $a_0$. Since~\eqref{eq:inhomogeneous-bvp} reduces to a Neumann problem when $a = 0$, standard theory gives that $L_0$ is invertible. We may then use~\eqref{eq:continuity-method} and the continuity method to deduce that $L_a$ is invertible for all $a \in [-a_0, a_0]$. Since $a_0 \in (0,1)$ is arbitrary, we get the solvability of~\eqref{eq:inhomogeneous-bvp} for all $a \in (-1, 1)$.
\end{proof}
For the proof of Proposition~\ref{prop:boundary-regularity} we also need a version of Lemma~\ref{lemm:inhomogeneous} where~\eqref{eq:inhomogeneous-bvp} is solved in $W^{2, p} \cap W^{2, q}$. 
\begin{lemm}\label{lemm:intersection}
Suppose $1 < p, q < \infty$. For all $a \in (-1, 1)$, $f \in (L^p \cap L^q)(\RR^2_+; \RR^2)$ and $g \in (W^{1, p} \cap W^{1, q})(\RR^2_+; \RR^2)$, the unique solution to~\eqref{eq:inhomogeneous-bvp} in $W^{2, p}(\RR^2_+; \RR^2)$ given by Lemma~\ref{lemm:inhomogeneous} actually lies in $(W^{2, p} \cap W^{2,q})(\RR^2_+; \RR^2)$.
\end{lemm}
\begin{proof}[Sketch of proof]
Fix some $a_0 \in (0, 1)$ and let $A$ be the set of $a \in [-a_0, a_0]$ for which the conclusion of the lemma holds for all $f \in (L^p \cap L^q)(\RR^2_+; \RR^2)$ and $g \in (W^{1, p} \cap W^{1, q})(\RR^2_+; \RR^2)$. By the uniqueness part of Lemma~\ref{lemm:inhomogeneous}, we see that $a \in A$ if and only if~\eqref{eq:inhomogeneous-bvp} admits a solution in $(W^{2, p} \cap W^{2,q})(\RR^2_+; \RR^2)$. Standard theory for the Neumann problem then gives $0 \in A$, while Lemma~\ref{lemm:inhomogeneous-estimate} provides $W^{2, p}$ and $W^{2, q}$ estimates that allows us to show that $A$ is both open and closed, using the contraction mapping principle for openness. Consequently $A = [-a_0, a_0]$, and we are done since $a_0$ is arbitrary.
\end{proof}
We end this appendix with a regularity result (Lemma~\ref{lemm:regularity-on-B}) and a global a priori estimate (Lemma~\ref{lemm:estimate-on-B}) on $\bB$.
\begin{lemm}\label{lemm:regularity-on-B}
Suppose $a \in (-1, 1)$ and $1 < p < q < \infty$. Given $f \in W^{2, q}(\bB; \RR^2)$ and $g \in W^{1, q}(\bB; \RR^2)$, if $u \in W^{2, p}(\bB; \RR^2)$ solves 
\begin{equation}\label{eq:bvp-on-B}
\left\{
\begin{array}{l}
\Delta u = f \text{ on }\bB, \\
\Big[\pa{u}{r} + \left( \begin{array}{cc} 0 & a \\ -a & 0 \end{array} \right)\pa{u}{\theta}\Big]\Big|_{\partial \bB} = g|_{\partial\bB},
\end{array}
\right.
\end{equation}
then $u \in W^{2, q}(\bB; \RR^2)$.
\end{lemm}
\begin{proof}[Sketch of proof]
Below, we write $L^p_{\loc}(\RR^2_+)$ to mean the space of functions having finite $L^p$-norm on $\bB^+_R$ for all $R > 0$, and similarly for $W^{k, p}_{\loc}(\RR^2_+)$. To make use of the previous estimates, consider the conformal map $F: (\RR^2_+, \partial \RR^2_+) \to (\bB, \partial \bB \setminus \{-e_1\})$ introduced at the end of Section~\ref{subsec:function-spaces} and suppose $F^*g_{\RR^2} = \lambda^2 g_{\RR^2}$. If we define 
\[
v = u \circ F,
\]
then $v \in W^{2, p}_{\loc}(\RR^2_+)$, and we compute based on~\eqref{eq:bvp-on-B} that
\[
\left\{
\begin{array}{l}
\Delta v = \lambda^2 (f \circ F) \text{ on }\RR^2_+,\\
(B_a^- v)|_{\partial\RR^2_+} = -\lambda(g\circ F)|_{\partial\RR^2_+}.
\end{array}
\right.
\]
Thus, for any smooth function $\zeta$ with compact support in $\overline{\RR^2_+}$, we have
\begin{equation}\label{eq:cutoff-bvp}
\left\{
\begin{array}{l}
\Delta (\zeta v) - \zeta v   = -\zeta v + 2\nabla\zeta \cdot \nabla v + v \Delta \zeta + \zeta \lambda^2 (f \circ F) = : \widetilde{f},  \text{ on }\RR^2_+,\\
B_a^- (\zeta v)  = \pa{\zeta}{y} v - \pa{\zeta}{x}\left( \begin{array}{cc} 0 & a \\ -a & 0 \end{array} \right)v  -\zeta\lambda(g\circ F)  = : \widetilde{g}, \text{ on }\partial \RR^2_+.
\end{array}
\right.
\end{equation}
Next, with $p^*$ defined by
\[
p^* = \left\{
\begin{array}{l}
\min\{\frac{2p}{2-p}, q\}, \text{ if }1 < p < 2,\\
q, \text{ if }p \geq 2,
\end{array}
\right.
\]
we have by Sobolev embedding that $v \in (W_{\loc}^{1, p^*} \cap W_{\loc}^{1, p})(\RR^2_+)$. Combining this with our assumption on $f$ and $g$ gives
\[
\widetilde{f} \in (L^{p^*} \cap L^p)(\RR^2_+),\ \ \widetilde{g} \in (W^{1, p^*} \cap W^{1, p})(\RR_+^2).
\]
Thus we deduce from~\eqref{eq:cutoff-bvp} and Lemma~\ref{lemm:intersection} that $\zeta v \in (W^{2, p^*} \cap W^{2, p})(\RR^2_+)$, and consequently $v \in W^{2, p^*}_{\loc}(\RR^2_+)$ since $\zeta$ is arbitrary. Since $p > 1$, iterating the above argument finitely many times gives $v \in W_{\loc}^{2, q}(\RR^2_+)$, so that $u$ is $W^{2, q}$ locally near $e_1$. Repeating this at other boundary points and invoking standard interior regularity gives $u \in W^{2, q}(\bB; \RR^2)$. 
\end{proof}

\begin{lemm}\label{lemm:estimate-on-B}
Given $a_0 \in (0, 1)$, there exists $C = C(p, a_0) > 0$ such that for all $a \in [-a_0, a_0]$, $f \in L^p(\bB; \RR^2)$ and $g \in W^{1, p}(\bB;\RR^2)$, if $u \in W^{2, p}(\bB; \RR^2)$ solves the boundary value problem~\eqref{eq:bvp-on-B}, then we have
\begin{equation}\label{eq:estimate-on-B}
\|\nabla u\|_{1, p; \bB} \leq C(\|f\|_{p; \bB} + \|g\|_{1, p; \bB}).
\end{equation}
\end{lemm}
\begin{proof}[Sketch of proof]
Without loss of generality we assume that $\int_{\bB}u = 0$. Repeating the previous proof up to equation~\eqref{eq:cutoff-bvp} and using Lemma~\ref{lemm:inhomogeneous-estimate}, we obtain local boundary estimates for $u$ near $e_1$. Repeating this at other points on $\partial \bB$ and recalling interior $W^{2, p}$-estimates, we obtain a global estimate of the form
\begin{equation}\label{eq:estimate-with-w1p}
\|\nabla u\|_{1, p;\bB} \leq C_{p, a_0}(\|f\|_{p; \bB} + \|g\|_{1, p; \bB} + \|\nabla u \|_{p; \bB}).
\end{equation}
To finish, we claim that
\begin{equation}\label{eq:nabla-u-bound}
\|\nabla u\|_{p; \bB} \leq C (\|f\|_{p; \bB} + \|g\|_{1, p; \bB}),
\end{equation}
where $C$ again depends only on $p$ and $a_0$. Assume by contradiction that for each $j \in \NN$ there exist $a_j \in [-a_0, a_0]$, $u_j \in W^{2, p}(\bB;\RR^2)$, $f_j \in L^p(\bB;\RR^2)$ and $g_j \in W^{1, p}(\bB; \RR^2)$ such that~\eqref{eq:bvp-on-B} holds but that
\begin{equation}\label{eq:contradiction-assumption}
1 = \|\nabla u_j\|_{p; \bB} \geq j (\|f_j\|_{p; \bB} + \|g_j\|_{1, p; \bB}).
\end{equation}
Subtracting a suitable constant from $u_j$, we may also assume $\int_{\bB}u_j = 0$. Now, up to taking a subsequence, $(a_j)$ converges to some $a \in [-a_0, a_0]$. Moreover, by~\eqref{eq:contradiction-assumption} and~\eqref{eq:estimate-with-w1p}, passing to a further subsequence if necessary, we have that $(u_j)$ converges weakly in $W^{2, p}$ and strongly in $W^{1, p}$ to a limit $u$ which is non-constant and satsifies~\eqref{eq:bvp-on-B} with $f = 0$ and $g = 0$. Lemma~\ref{lemm:regularity-on-B} then implies that $u \in W^{2, q}$ for all $q < \infty$, and consequently we may use $u$ itself as a test function and integrate by parts on $\bB$ to get
\[
\begin{split}
0 =\ & -\int_{\bB} u \Delta u = -\int_{\partial \bB} u \cdot \pa{u}{r} + \int_{\bB}|\nabla u|^2\\
=\ &  \int_{\partial \bB} u \cdot \left( \begin{array}{cc} 0 & a \\ -a & 0 \end{array} \right)\pa{u}{\theta} + \int_{\bB} |\nabla u|^2\\
=\ & \int_{\bB} |\nabla u|^2 - 2\pa{u}{y} \cdot  \left( \begin{array}{cc} 0 & a \\ -a & 0 \end{array} \right)\pa{u}{x}\\
\geq\ & (1 - a_0)\int_{\bB}|\nabla u|^2.
\end{split}
\]
Since $a_0 < 1$, this forces $u$ to be constant, a contradiction. Hence~\eqref{eq:nabla-u-bound} holds for some $C$ depending only on $p$ and $a_0$, which together with~\eqref{eq:estimate-with-w1p} yields~\eqref{eq:estimate-on-B} as asserted.
\end{proof}
\bibliographystyle{amsplain}
\bibliography{cmc-capillary-arxiv}
\end{document}